\documentclass[10pt]{amsart}

\usepackage{mathabx}

\usepackage{amsmath}
\usepackage{graphicx,array,url}
\usepackage[vcentermath]{genyoungtabtikz}
\usepackage{graphicx}
\usepackage{eufrak}
\usepackage{dsfont}
\usepackage[english]{babel}
\usepackage{color}
\usepackage{mathabx}

\usepackage{palatino, mathpazo}
\usepackage{amscd}      
\usepackage{amssymb}
\usepackage{dsfont}
\usepackage{xypic}      
\LaTeXdiagrams          
\usepackage[all,v2]{xy}
\xyoption{2cell} \UseAllTwocells \xyoption{frame} \CompileMatrices
\usepackage{latexsym}
\usepackage{epsfig}
\usepackage{enumerate}

\usepackage{latexsym}
\usepackage{epsfig}
\usepackage{amsfonts}
\usepackage{enumerate}
\usepackage{mathrsfs}

\allowdisplaybreaks[3]

\newtheorem{prop}{Proposition}[section]
\newtheorem{lem}[prop]{Lemma}

\newtheorem{cor}[prop]{Corollary}
\newtheorem{thm}[prop]{Theorem}
\newtheorem{rmk}[prop]{Remark}

\newtheorem{defn}[prop]{Definition}

            {\nolinebreak $\Box$ \end{trivlist}}

\newcommand{\noprint}[1]{}

\newcommand{\Ext}{\mbox{Ext}}

\newcommand{\Hom}{\mbox{Hom}}

\newcommand{\Y}{\mathop{\sf Y}\nolimits}

\newcommand{\zz}{{\mathbb Z}}

\renewcommand{\i}{{\mathbb I}}
\renewcommand{\j}{{\mathbb J}}
\newcommand{\qq}{{\mathbb Q}}

\newcommand{\pp}{{\mathbb P}}
\newcommand{\cc}{{\mathbb C}}

\newcommand{\ff}{{\mathbb F}}

\newcommand{\sD}{{\mathcal D}}

\newcommand{\sE}{{\mathcal E}}
\newcommand{\sI}{{\mathcal I}}

\newcommand{\sO}{{\mathcal O}}

\newcommand{\sM}{{\mathcal M}}

\newcommand{\sA}{{\mathcal A}}
\newcommand{\Ac}{{\mathcal A}}
\newcommand{\Bc}{{\mathcal B}}
\newcommand{\Cc}{{\mathcal C}}

\newcommand{\rH}{\mathscr{H}}

\DeclareMathOperator{\Hilb}{Hilb}
\DeclareMathOperator{\pic}{Pic}

\DeclareMathOperator{\Ob}{Ob}

\DeclareMathOperator{\vir}{vir}
\DeclareMathOperator{\mov}{mov}

\DeclareMathOperator{\vd}{vd}
\DeclareMathOperator{\Ch}{ch}

\DeclareMathOperator{\fix}{fix}

\DeclareMathOperator{\dimension}{dimension}
\DeclareMathOperator{\SO}{SO}
\DeclareMathOperator{\image}{Im}

\newcommand{\rk}{\mathop{\rm rk}}

\newcommand{\td}{\mathop{\rm td}}

\renewcommand{\Im}{\mathop{\rm Im}}

\newcommand{\rank}{\mathop{\rm rank}\nolimits}

\newcommand{\spec}{\mathop{\rm Spec}\nolimits}

\newcommand{\Sym}{\mathop{\rm Sym}\nolimits}

\newcommand{\bb}[1]{\mathbb{#1}}
\newcommand{\bu}{\bullet}
\newcommand{\bff}[1]{\mathbf{#1}}

\newcommand{\lra}[1]{\langle #1 \rangle}
\newcommand\beq[1]{\begin{equation}\label{#1}}
\newcommand\eeq{\end{equation}}
\newcommand{\op}[1]{\operatorname{#1}}

\numberwithin{equation}{subsection}

\newcommand {\mat}      [1] {\left(\begin{array}{#1}}
\newcommand {\rix}          {\end{array}\right)}

\title[Stable pairs in 4-folds]{Stable pairs of 2-dimensional sheaves on 4-folds}
\author{Amin Gholampour, Yunfeng Jiang and Jason Lo}
\address{Department of Mathematics\\ University of Maryland\\ 4176 Campus Drive 
William E. Kirwan Hall
\\College Park, MD 20742-4015405 USA} 
\email{amingh@umd.edu}
\address{Department of Mathematics\\ University of Kansas\\ 405 Snow Hall 1460 Jayhawk Blvd\\Lawrence KS 66045 USA} 
\email{y.jiang@ku.edu}

\address{Department of Mathematics\\ California State University Northridge \\
18111 Nordhoff Street, Northridge, CA 91330 USA} 
\email{jason.lo@csun.edu}

\begin{document}
\sloppy \maketitle
\begin{abstract}
We identify Le Potier's moduli spaces of limit stable pairs $(F,s)$, where $F$ is a 2-dimensional sheaf on a nonsingular projective 4-fold $X$ and $s\in H^0(F)$, with the moduli spaces of polynomial stable 2-term complexes in derived category. These stable pairs are 2-dimensional analogs of Pandharipande-Thomas' stable pairs defined for 3-folds. We establish categorical correspondences involving these stable pairs, ideal sheaves of 2-dimensional subschemes of $X$, and 1-dimensional sheaves on $X$. Under some conditions on the Chern character, these lead to Hall algebra correspondences. The generalization of most of these results to higher ranks is also given.

 In case $X$ is Calabi-Yau, Oh-Thomas' construction gives a new set of invariants of $X$ counting these stable pairs. For certain Chern characters, these are related to the invariants of 2-dimensional stable sheaves. We calculate and study them in some cases and examples such as fibrations by abelian surfaces, local surfaces, and local Fano 3-folds. The last case in particular leads to new invariants of Fano 3-folds counting 2-dimensional stable pairs with reduced supports.
\end{abstract}

\maketitle

\tableofcontents

\section{Introduction}
The stable pair theory of Pandharipande and Thomas \cite{PT} has been a successful \emph{curve counting theory} for 3-folds.  It studies pairs $(F,s)$, where $F$ is a pure 1-dimensional sheaf on a nonsingular (quasi-)projective threefold $X$ and $s\in \Gamma(X,F)$ is a global section such that the induced map $\sO_X\to F$ has at most a 0-dimensional cokernel. The purity condition implies that the sheaf $F$ is supported  over a Cohen-Macaulay curve $C\subset X$.  If $C$ is a nonsingular curve, then $F$ is necessarily the pushforward of a line bundle on $C$. The moduli space of such pairs is a Le Potier's moduli space of pairs $(F,s)$, when the pair stability parameter is chosen to be sufficiently large. At the  same time, this is also a moduli space of stable 2-term complexes $[\sO_X\to F]$ in derived category.
This is used to equip the moduli space of PT pairs with a perfect obstruction theory. One can define the PT invariant by integrating the resulting virtual cycle. There are very simple to state but highly nontrivial connections among PT invariants and other curve counting invariants such as Gromov-Witten invariants and Donaldson-Thomas invariants of rank 1 torsion free sheaves. PT theory has an advantage as it avoids the ``roaming points" present in the corresponding DT theory.

The fact that the support of $F$ in PT theory has codimension 2 proved to be crucial for the theory to work. A natural question is whether one can develop a parallel \emph{surface counting theory} if one replaces the ambient 3-fold by a 4-fold.
Motivated by this, in this paper we study pairs $(F,s)$, where $F$ is a pure 2-dimensional sheaf on a nonsingular (quasi-)projective 4-fold $X$ and $s\in \Gamma(X,F)$ is a global section, such that the induced map $\sO_X\to F$ has at most a 1-dimensional cokernel. 
Unlike the PT pairs the scheme theoretic support of $F$, denoted by $S$, is not Cohen-Macaulay in general. If the support $S$ is Cohen-Macaulay then the cokernel of  $\sO_X\to F$ is forced to be pure 1-dimensional (see  \S \ref{sec_description}).  There is a simple description for the pairs above in the special case that $S$ is nonsingular as discussed in \S \ref{sec_description}. Any such a pair is equivalent to the data of a PT stable pair  on $S$ or equivalently to the data of nested subschemes $Z\subset D\subset S$, where $Z$ is 0-dimensional and $D$ is 1-dimensional.

The moduli space of these pairs can again be constructed via Le Potier's coherent system theory \cite{LeP1, LeP2}. One of the main results of this paper (\emph{Theorem \ref{thm_moduli}}) shows that this moduli space is identified with the moduli space of 2-term complexes in derived category, which are stable under a certain \emph{polynomial stability condition}. Let $P:=P(X,\Ch)$ denote this moduli space of pairs $(F,s)$ with $\Ch(F)=\Ch$.  By the description above, for certain $\Ch$ one can view $P(X,\Ch)$ as a compactification of the space of embedded nonsingular surfaces $S\subset X$ that is different from the Hilbert scheme compactification. This is carried out in \S \ref{sec_stability_moduli}, where we work more generally and construct the moduli spaces of \emph{higher rank} analogs of these stable pairs as well (\emph{Theorem \ref{thm_exist}}). These moduli spaces are first shown to exist as Artin stacks locally of finite presentations. We show that they are in fact of finite type over $\bb C$ (\emph{Theorem \ref{prop:bdd-2}}).

In \S \ref{subsec_polynomial_ideal_sheaves} we introduce another polynomial stability condition that gives the ideal sheaves of 2-dimensional subschemes of $X$ leading to the Hilbert schemes of such subschemes.  Stable pair theory compactification avoids the ``roaming curves and points"  present in the Hilbert scheme compactification. This is again done along with the higher ranks generalization (\emph{Theorem \ref{thm_moduli2}}). A pictorial comparison between two polynomial stability conditions mentioned above can be found in Figure 1 at the end of \S \ref{subsec_polynomial_ideal_sheaves}.

In \S \ref{sec_cat} we work out \emph{categorical correspondences} involving the categories of 2-term complexes which are stable with respect to two polynomial stability conditions mentioned above, as well as the categories of semistable 1-dimensional sheaves (\emph{Proposition  \ref{lem:AG50-17-2}, Theorem \ref{thm:catwc}, Corollaries \ref{cor:AG50-17-7}, \ref{lem7}}). A novelty of Theorem \ref{thm:catwc} is that it can be applied to several other similar situations; see \S \ref{sec_cat} for a list of applications involving PT pairs that can be directly deduced from this Theorem.

In \S \ref{sec_Hall}, under some restrictions on the Chern character vector $\Ch$, we deduce \emph{Hall algebra correspondences} from the categorical relations in \S \ref{sec_cat} (\emph{Theorem \ref{thm_hall}}). These are analogs of PT/DT Hall algebra correspondence of \cite{Bridgeland10,Toda2} in dimension 3. Unlike in dimension 3, where the correspondence is expressed by a single relation, we need two relations to state our Hall algebra stable pair/ideal correspondence. 

In order to equip $P(X,\Ch)$ with a \emph{virtual cycle}, we confine ourselves to \emph{Calabi-Yau 4-folds} $X$ and use the recent result of Oh-Thomas \cite{OT}. This is done in \S \ref{subsec_virtual_cycle_Oh-Thomas}. Since our pairs are equivalent to stable and hence simple objects of the derived category the construction \cite{OT}  directly applies to our case resulting in a virtual cycle 
\begin{align*}&[P(X,\Ch)]^{\vir} \in A_{\vd/2}(P, \bb Z [\frac{1}{2}]),\\ &\vd=2\Ch_4+2\Ch_2\cdot \operatorname{td}_2(X)-\Ch_2^2.\end{align*} The pair invariants are defined by integrating against this class. By the description above these invariants can be thought of as the virtual counts of nested triples $Z \subset D \subset S$ of $0,1,2$-dimensional  subschemes of $X$ in which $S$ is nonsingular. 
One can similarly equip the Hilbert scheme of 2-dimensional subschemes $Y\subset X$ such that $\Ch(\sO_Y)=\Ch$ with a virtual cycle that has the same virtual dimension and define the DT invariants. It would be interesting to formulate a general relation among  these and the stable pair invariants. In a work in preparation \cite{BKP}, Bae, Kool and Park are investigating this problem \cite{BKP}.

In \S \ref{sec_degen} under some conditions on $\Ch$, we realize $P(X,\Ch)$ as the zero locus of a section of a vector bundle over another moduli space of stable pairs, which is a projective bundle over the moduli space of stable 2-dimensional sheaves $M(X,\Ch)$. We then find a Chern class formula relating the virtual cycles (\emph{Theorem \ref{thm_degen}}). This in turn leads to a relation to Oh-Thomas' virtual cycle $[M(X,\Ch)]^{\vir}$ (\emph{Corollary \ref{cor_degen}}).

 The rest of the paper is devoted to working out a few examples and an application. In our concrete examples, either the moduli space $P(X,\Ch)$ is smooth in which case the virtual cycle can be expressed as a (square root) Euler class of its obstruction bundle (\S \ref{sec_local_surfaces}, \S\ref{subsec_planes}, \S\ref{sec_compact}), or one can apply Oh-Thomas' virtual localization formula (\S \ref{sec_virloc}).  In these cases, the stable pair invariants are expressed in terms of the product of Seiberg-Witten invariants of a nonsingular surface and well-known integrals over the Hilbert scheme of points on the surface.
 
 In \S \ref{sec_Fano_threefolds} we take $X$ to be the total space of the canonical bundle of a Fano threefold $Y$ and $\Ch$ to be the pushforward of a Chern character vector $\Ch_Y$ on $Y$. Due to higher obstructions, in general the natural obstruction theory of the moduli space of 2-dimensional stable pairs $P(Y,\Ch_Y)$ is not perfect. By considering $P(Y,\Ch_Y)$ as a component of $P(X,\Ch)$ for certain choices of $\Ch_Y$, we can deduce a virtual cycle for  $P(Y,\Ch_Y)$ from $[P(X,\Ch)]^{\vir}$ (\emph{Theorem \ref{fano}}).

\smallskip\noindent\textbf{Acknowledgements.} This paper has gotten several ideas and inspirations from \cite{CMT, OT, PT, Toda2}. The authors would like to thank Martijn Kool,  Yinbang Lin, Hyeonjun Park, Feng Qu, Richard Thomas, and Ziyu Zhang for helpful discussions. A.G. was partially supported by Simons Foundation Collaboration Grant 703680
, Y. J. was partially supported by Simons Foundation Collaboration Grant 714191, and J. L. was partially supported by NSF Grant DMS-2100906.

\section{Stability conditions and moduli spaces} \label{sec_stability_moduli}
In this Section we let $X$ be a nonsingular projective 4-fold with a choice of a polarization  $\omega$.

\subsection{Moduli space of stable pairs} \label{sec_mst}
Let 
\begin{equation} \label{chern_char} \Ch=(0,0,\gamma,\beta,\xi)\in H^{2*}(X,\mathbb Q)\end{equation} be a Chern Character vector. 
 A 2-dimensional pair $(F,s)$ is given by a  2-dimensional coherent  sheaf $F$ and 
a section $s: \sO_X\to F$.   Let $F(m)=F\otimes \omega^m$. For a  pair $(F,s)$ with the Chern character 
$\Ch$ the Hilbert polynomial of $F$  is given by
$$\chi(F(m))=m^2\frac{\gamma\cdot \omega^2}2 +m(\beta-\frac{\gamma\cdot K_X}2)\cdot \omega +\xi-\frac{\beta\cdot K_X}2+\gamma\cdot \text{td}_2(X).$$
For a 2-dimensional sheaf $F$, let 
$r_2(E), r_1(E)$ denote the leading and linear coefficients of its Hilbert polynomial. Let $q\in\qq[m]$ be a polynomial with positive leading coefficient. The pair $(F,s)$ is called $q$-stable if for every proper subsheaf $E\subset F$, 
\begin{equation}\label{eqn_stability_1}
\frac{\chi(E(m))}{r_2(E)}<\frac{\chi(F(m))+q(m)}{r_2(F)}, \quad m\gg 0
\end{equation}
holds, and for every  proper subsheaf $E\subset F$ through which $s$ factors, 
\begin{equation}\label{eqn_stability_2}
\frac{\chi(E(m))+q(m)}{r_2(E)}<\frac{\chi(F(m))+q(m)}{r_2(F)}, \quad m\gg 0.
\end{equation}
The $q$-semistability conditions are defined by allowing equalities in \eqref{eqn_stability_1} and \eqref{eqn_stability_2}.
Let $P^q(X,\Ch)$ be the moduli space of $q$-semistable pairs $(F,s)$ in class $\Ch$.  This moduli space was constructed by Le Potier in \cite{LeP1, LeP2} by GIT techniques. 
For  $q\gg 0$ there is a simple description of $q$-semistability similar to  \cite[Lemma 1.3]{PT}:
\begin{lem}\label{lem_stable_pair}
Assume that $\gamma\neq 0$ and $q$ is sufficiently large then the $q$-stability and $q$-semistability coincide, and a pair 
$(F,s)$ is $q$-stable if and only if 
\begin{enumerate}
\item the sheaf $F$ is pure,
\item the section $\sO_X\stackrel{s}{\longrightarrow}F$ has at most $1$-dimensional cokernel.
\end{enumerate}
\end{lem}
\begin{proof}
If $q\gg 0$ \eqref{eqn_stability_1} is always satisfied.
We rearrange the inequality for (\ref{eqn_stability_2}) as:
\begin{equation}\label{eqn_stability_3}
(r_2(F)-r_2(E))(\chi(F(m))+q(m))\leq r_2(F)\chi((F/E)(m)). 
\end{equation}
Then for $q\gg 0$ we get 
$r_2(F)-r_2(E)=0$ and the equality cannot occur. 
This in particular for $E=\Im(s)$ implies that $\op{coker}(s)$ has to be at  most $1$-dimensional. 
\end{proof}

We refer to $q$-stability for $q\gg 0$ as the \emph{limit stability}, and a limit stable pair is simply called a \emph{stable pair}. If  
$S$ is the scheme theoretic support of $F$ Lemma \ref{lem_stable_pair} implies that $F$ is isomorphic to $\sO_S$ away from at most  a 1-dimensional subscheme of $S$.

Denote the moduli space of (limit) stable pairs by $P(X,\Ch)$. 
In \S \ref{subsec_polynomial_stability} we will realize the limit stability of pairs as an instance of polynomial stability for certain objects of derived category. 

\subsection{Stable pairs with nice supports} \label{sec_description} For better understanding the objects of our study, in this Subsection we give a description of some of the stable pairs parameterized by the moduli space $P(X,\Ch)$. This will also play an important role in the examples that we will later consider in this paper. 

 First, suppose that for a given stable pair $(F,s)$ in class $\Ch$ the sheaf $F$ is supported on a surface $i\colon S\subset X$ with $\Ch_2(F)=[S]$ satisfying Serre's $S_{2,2}$ condition. This is equivalent to $\op{depth}(\sO_{S,s})=2$ for all closed points $s\in S$ or that $S$ is a Cohen-Macaulay surface. This also implies $\mathcal Ext^{i\neq 2}(\sO_S, \sO_X)=0$. Using the short exact sequence $$0\to \sO_S\xrightarrow{s} F\to Q\to 0,$$ we find that $\mathcal Ext^4(Q, \sO_X)=0$, and hence $Q$ is a pure 1-dimensional sheaf by \cite[Proposition 1.1.10]{HL}. 
 Let $\omega_S$ be the dualizing sheaf of $S$, and $$F^{**}:=\mathcal Hom_S(\mathcal Hom_S(F,\omega_S)).$$The composition of the section $s$ and the natural embedding $F\hookrightarrow F^{**}$ gives a stable pair $(F^{**},s')$. This data is equivalent to an effective generalized divisor in the sense of \cite[Proposition 2.4]{H} . 
 
 Next, suppose that the surface $S$ is in particular nonsingular. Since $F$ is pure and $Q$ is a torsion sheaf, $F$ must be the pushforward of a rank 1 torsion free sheaf on $S$. Since $S$ is nonsingular we conclude that $F=i_*\sI_Z(D)$, where $D\subset S$ is an effective divisor and  $\sI_Z<\sO_S$ is the ideal of a 0-dimension subscheme $Z\subset S$. By adjunction the section $s$ corresponds to a nonzero section $\sO_S\to \sI_Z(D)$ that we also denote by $s$. We obtain a commutative diagram in which the last two rows and the first two columns are the natural short exact sequences
 $$
 \xymatrix{\sO_S \ar@{=}[r] \ar[d]^-s & \sO_S \ar[d]^-{s'} & \\
   \sI_Z(D) \ar[r] \ar[d] & \sO_S(D) \ar[r] \ar[d] & \sO_Z  \ar@{=}[d] \\
  \sI_{Z\subset D}(D) \ar[r] & \sO_D(D) \ar[r] & \sO_Z.}
 $$
We conclude that $Z$ is a subscheme of $D$ with the ideal $\sI_{Z\subset D} <\sO_D$. Conversely, given a chain of subschemes $Z\subset D\subset S$ we can recover the stable pair $(F,s)$ by chasing the diagram above backward.  By \cite[Lemma A.4]{KT}, after dualizing, the first column in diagram  is equivalent to a stable pair on $S$ \cite{PT3, KT} i.e. the data of a pure 1-dimensional sheaf and a section with 0-dimensional cokernel. If the class of $D$ on $S$ is $\beta$ and the length of $Z$ is $n$, we will therefore, identify the locus $P_S \subset P(X,\Ch)$ consisting of stable pairs supported on $S$ with the nested Hilbert scheme (see  \cite{GSY, GT1, GT2})
  \begin{equation} \label{nested} j\colon S^{[0,n]}_\beta\subset  S^{[n]}\times S_\beta.\end{equation}

\subsection{Polynomial stability and stable pairs}\label{subsec_polynomial_stability} 
We define the following full subcategories of $\op{Coh} (X)$:
\begin{itemize}
\item  $\op{Coh}^{\leq d}(X)$ is the  subcategory of sheaves supported in dimension at most $d$;
\item $\op{Coh}^{=d}(X)$ is the subcategory of pure $d$-dimensional sheaves;
\item $\op{Coh}^{\geq d}(X)$ is the subcategory of sheaves without nonzero subsheaves in $\op{Coh}^{\leq d-1}(X)$.
\end{itemize}
For any interval $I$, we will write $$D^I_{\op{Coh}(X)}$$ to denote the full subcategory of $D^b(X)$ consisting of
objects $E$ such that $H^i(E) = 0$ for all $i\not \in I$.

We refer to \cite[Section 3]{BayerPBSC} for the construction of the heart $\sA^p$ associated to a \emph{perversity function} $p\colon \{0, 1, 2, 3, 4\}\to \bb Z$ and other relevant notation.

In this Section, we will describe some polynomial stability conditions $(Z_P,\sA^p)$ where the heart $\sA^p$ contains stable pairs of \S \ref{sec_mst}.  
Note that by Lemma \ref{lem_stable_pair}  we may think of a stable pair as a 2-term complex $J := [\sO_X \to F]$ satisfying:
\begin{itemize}
\item $H^{-1}(J)$ is a rank-one torsion-free sheaf,
\item $H^0(J) \in \op{Coh}^{\leq 1}(X)$.
\end{itemize}
Therefore, in order for $\sA^p$ to contain these objects, a sufficient condition is
\begin{equation}\label{eq1}
\op{Coh}^{=4}(X) [1]\subset \sA^p \text{ and } \op{Coh}^{\leq 1}(X) \subset \sA^p.
\end{equation}  Since $\sA^p = D^{p,\leq 0}\cap D^{p,\geq 0}$, where
\begin{align*}
\sA^{p,\leq k} &:= \{ F \in \op{Coh} (X) : p(\dim F) \geq -k\},\\
D^{p,\leq 0}&:=\{ E \in D^b(X) : H^{-k}(E) \in \sA^{p,\leq k} \;\; \forall k \in \mathbb Z\},\\
D^{p,\geq 0}&:=\{ E \in D^b(X) : \Hom(A,E)=0 \;\; \forall k\in \bb Z, \; \;\forall A \in \sA^{p,\leq k}[k+1] \},
\end{align*}
the condition \eqref{eq1} would hold if we require the perversity function $p$ to satisfy
\begin{align}
 p(4) &\geq -1 \label{eq2}\\
  p(1) &\geq 0. \label{eq3}
\end{align}
Let us suppose from now on that $p$ satisfies \eqref{eq2} and \eqref{eq3}. The condition \eqref{eq2} then implies $\sA^{p, \leq k}=\op{Coh} (X)$ for all $k \geq 1$.  Then, since $\sA^p \subset D^{p,\geq 0}$, we have $\Hom_{D^b(X)}(D^{\leq -2}_{\op{Coh} (X)}, \sA^p)=0$, implying $\sA^p \subseteq D^{\geq -1}_{\op{Coh} (X)}$.  On the other hand, let us further impose
\begin{equation}\label{eq4}
p(0)=0.
\end{equation}
Then $\sA^{p,\leq i}=\{0\}$ for all $i<0$, and so the condition $\sA^p \subset D^{p,\leq 0}$ implies $\sA^p \subset D^{\leq 0}_{\op{Coh} (X)}$.  Overall, we obtain \begin{equation}\label{equ_heart} \sA^p = \langle \op{Coh}^{\geq m+1}(X)[1], \op{Coh}^{\leq m}(X)\rangle, \quad \text{ where $m=1,2$ or $3$.} \end{equation} Thus we have

\begin{prop}\label{lem1}
Suppose $p$ is a perversity function satisfying \eqref{eq2}, \eqref{eq3} and \eqref{eq4}.  Then the heart $\sA^p$ contains all stable pairs.
\end{prop} \qed


\begin{prop}\label{lem3}
Suppose $E \in D^{[-1,0]}_{\op{Coh} (X)}$ satisfies:
\begin{itemize}
\item $\Ch(E)=(-1,0,\gamma,\beta,n)$ where $\gamma$ is effective, and $\det{E}\cong \sO_X$.
\item $H^{-1}(E)$ is a torsion-free sheaf.
\item $H^0(E)\in \op{Coh}^{\leq 1}(X)$.
\end{itemize}
Then $E$ is isomorphic to a 2-term complex of the form $[\sO_X \overset{s}{\to} F]$ for some nonzero sheaf $F$ supported in dimension 2.
\end{prop}

\begin{proof}
In the canonical exact sequence $0\to H^{-1}(E)\to (H^{-1}(E))^{\ast\ast} \to H\to 0$, the sheaf $H$ is necessarily supported in codimension at least 2, and so the rank-one reflexive sheaf $(H^{-1}(E))^{\ast\ast}$, which is necessarily a line bundle and has trivial determinant, must be the structure sheaf $\sO_X$.  Hence $H= \sO_W$ is the structure sheaf of some closed subscheme of $X$ of dimension at most 2, while $H^{-1}(E)=I_W$ is the corresponding ideal sheaf.  Writing $R=H^0(E)$, we have the exact triangle
\begin{equation}\label{eq5}
I_W [1] \to E \to R \overset{\alpha}{\to} I_W [2]
\end{equation}
Since $R$ is supported in dimension at most 1, it follows that $$\Ext^1(R,\sO_X)=\Ext^2(R,\sO_X)=0.$$  Applying $\Hom (R,-)$ to the structure exact sequence
\begin{equation}\label{eq6}
0 \to I_W \to \sO_X \to \sO_W \to 0
\end{equation}
 then gives
\begin{equation}\label{eq7}
  \Ext^1(R,\sO_W) \cong \Ext^2(R,I_W).
\end{equation}
From \eqref{eq5}, we know  $E[1]$ is the cone of $\alpha$ which represents an element of  $\Ext^2(R,I_W)$.  The isomorphism \eqref{eq7} then implies $\alpha$ factors as
\[
  R \to \sO_W [1] \to I_W [2].
\]
Applying the octahedral axiom to this composition then yields an exact sequence of sheaves
\begin{equation}\label{eq8}
0 \to \sO_W \to F \to R \to 0
\end{equation}
as well as  an exact triangle
\begin{equation}\label{eq9}
\sO_X \to F \to E \to \sO_X[1]
\end{equation}
which shows that $E$ is isomorphic in $D^b(X)$ to the 2-term complex $[\sO_X \to F]$.
From \eqref{eq8}, we know $F \in \op{Coh}^{\leq 2}(X)$.  From the Chern classes of the terms in \eqref{eq9}, we know $F$ is a nonzero sheaf supported in dimension 2.
\end{proof}

\begin{lem}\label{lem2}
Given an exact triangle in $D^b(X)$
\[
\sO_X \to F \to E \to \sO_X[1]
\]
where $F$ is a coherent sheaf, we have
\[
  \Hom (A,F)\cong \Hom (A,E)
  \]
  for any $A \in \op{Coh}^{\leq 1}(X)$.
\end{lem}

\begin{proof}
This follows from the vanishings $\Hom (A,\sO_X)=\Hom (A,\sO_X[1])=0$ for any $A \in \op{Coh}^{\leq 1}(X)$.
\end{proof}


We briefly describe the notion of a polynomial stability condition in the sense of Bayer  \cite{BayerPBSC}.  Since we use only very specific polynomial stability conditions in this article, we will not define them in the fullest generality:

Given the heart $\Ac$ of a t-structure on $D^b(X)$, a \emph{polynomial stability function} (in the restricted sense in this article) on $\Ac$ is a group homomorphism $Z : K(\Ac) \to \mathbb{C}[x]$ such that for every nonzero object $E \in \Ac$, we have
\[
  Z(E)(x) \in \{ re^{i\pi\theta} : r \in \mathbb{R}_{>0}, \theta \in (0,1] \} \text{\quad for $x \gg 0$}.
\]
Given a polynomial stability function $Z$ on $\Ac$, we can associate to every nonzero $E \in \Ac$ a phase function $\phi_Z (E)$ (or $\phi (E)$), which is a function germ $$\phi (E) : \mathbb{R}_{>0} \to (0,1]$$ defined by the relation
\[
  Z(E)(x) \in \mathbb{R}_{>0}\cdot e^{i \pi \phi (E)(x)} \text{\quad for $x \gg 0$}.
\]
Given two nonzero objects $E, F \in \Ac$, we write $\phi (E) \prec \phi (F)$ to mean $\phi (E)(x) < \phi (F)(x)$ for $x \gg 0$.

When $Z$ is a polynomial stability function on $\Ac$ and $0 \neq E \in \Ac$, we say $E$ is $Z$-semistable (resp.\ $Z$-stable) if, for every $\Ac$-short exact sequence
\[
0 \to E' \to E \to E'' \to 0
\]
where $E', E'' \neq 0$, we have $\phi (E') \preceq  \phi (E'')$ (resp.\ $\phi (E') \prec \phi (E'')$).  We say $Z$ has the Harder-Narasimhan (HN) property on $\Ac$ if every nonzero $E \in \Ac$ admits a filtration in $\Ac$
\[
 0 = E_0 \subsetneq E_1 \subsetneq \cdots \subsetneq E_n = E
 \]
 where each $E_i/E_{i-1}$ is $Z$-semistable and $$\phi (E_1/E_0) \succ \phi (E_2/E_1) \succ \cdots \succ \phi (E_n/E_{n-1}).$$
 
 A \emph{polynomial stability condition} on $D^b(X)$ is then a pair $(Z,\Ac)$ where $\Ac$ is the heart of a t-structure on $D^b(X)$, and $Z$ is a polynomial stability function on $\Ac$ with the HN property.  For the polynomial stability conditions we use in this article, the heart $\Ac$ will always take the form $\Ac^p$, for some perversity function $p$, as constructed  in \eqref{equ_heart}, while the central charge $Z$ will always take the form
 \begin{equation} \label{central}
   Z (E)(x) = \int_X \sum_{i=0}^4 \rho_i \omega^i  \Ch(E) x^i
\end{equation}
for some complex numbers $\rho_i$ such that $\rho_{i}/\rho_{i+1}$ for $0\le i <4$ belong to the open upper half plane, and  $(-1)^{p(i)} \rho_i$ for $0\le i\le 4$ belong to the semiclosed upper half plane obtained by removing the nonnegative real axis.

\begin{prop}\label{lem4}
Suppose $p$ is a perversity function satisfying
\begin{equation*}
p(0)=p(1)=0, \text{\quad} p(2)=p(3)=p(4)=-1,
\end{equation*}
 and that  $p$ is associated to a stability vector $\rho$ with
\begin{equation}\label{eq10}
\phi (\rho_0) > \phi(\rho_1) > \phi (-\rho_4)
\end{equation}
and
\begin{equation}\label{eq11}
\phi (-\rho_2), \phi (-\rho_3)>\phi (-\rho_4).
\end{equation}
Let $Z_P$ be the central charge of the resulting polynomial stability condition according to \eqref{central}.  Then for any object $E \in D^b(X)$ with $\Ch(E)=(-1,0,\gamma,\beta, \xi)$, where $\gamma$ is effective and $\det{E} =\sO_X$, we have that $E$ is a stable pair if and only if $E$ is a $Z_P$-semistable object in the heart $\sA^p$. 
\end{prop}

\begin{proof}
 Our choice of perversity function $p$ implies $$\sA^p = \langle \op{Coh}^{\geq 2}(X)[1], \op{Coh}^{\leq 1}(X)\rangle.$$  Let  $E$ be an object in $D^b(X)$ with the prescribed Chern character and determinant.
 
Suppose first that $E$  is a $Z_P$-semistable object in $\sA^p$. Then $H^0(E) \in \op{Coh}^{\leq 1}(X)$ and so $H^{-1}(E)$ is of rank one.  By \eqref{eq11} and the $Z_P$-semistability of $E$, it follows that $H^{-1}(E)$ is torsion-free.  Then by Proposition \ref{lem3}, $E$ is isomorphic to a complex of the form $[\sO_X \overset{s}{\to} F]$ where $F$ is a 2-dimensional sheaf sitting at degree 0.  Since $\op{Coh}^{\leq 1}(X)$ is a Serre subcategory of $\sA^p$, the $Z_P$-semistability of $E$ implies $$\Hom_{D^b(X)} (\op{Coh}^{\leq 1}(X),E)=0.$$  By Lemma \ref{lem2}, it follows that $\Hom (\op{Coh}^{\leq 1}(X),F)=0$, i.e.\ $F$ is a pure 2-dimensional sheaf, i.e.\ $[\sO_X \overset{s}{\to} F]$ is a stable pair.

Next, suppose $E \in D^b(X)$ is represented by a stable pair $[\sO_X \overset{s}{\to} F]$ with $F$ sitting at degree $0$.  Clearly $E$ is an object in $\sA^p$ and is supported in dimension $4$.  To prove that $E$ is $Z_P$-semistable, take any $\sA^p$-short exact sequence
\[
0 \to A \to E \to B \to 0
\]
where $A, B \neq 0$.  Consider the associated long exact sequence of sheaves
\[
0 \to H^{-1}(A) \to H^{-1}(E) \to H^{-1}(B) \to H^0(A) \to H^0(E) \to H^0(B) \to 0.
\]

If $0 \neq H^{-1}(A) \subsetneq H^{-1}(E)$, then $H^{-1}(B)$ has rank zero.  If $H^{-1}(B)$ is nonzero, then from \eqref{eq11} we have $\phi (E) \prec \phi (B)$; if $H^{-1}(B)=0$, then $B=H^0(B) \in \op{Coh}^{\leq 1}(X)$ and from \eqref{eq10} we also have $\phi (E)\prec \phi (B)$.

If $H^{-1}(A)=H^{-1}(E)$, then $H^{-1}(B)=0$ for dimension reason and $B=H^0(B)$. Then $\phi (E) \prec \phi (B)$ as in the previous paragraph.

If $H^{-1}(A)=0$, then $A = H^0(A) \in \op{Coh}^{\leq 1}(X)$.  The purity of $F$ implies $\Hom (A,F)=0$.  Then $\Hom (A,E)=0$ by Lemma \ref{lem2},  contradicting the assumption that $A$ is a nonzero subobject of $E$ in $\sA^p$.
Hence $E$ is $Z_P$-semistable whenever it is a stable pair.
\end{proof}

\subsection{Moduli space of $Z_P$-stable objects}
In this Subsection, we discuss the existence of the moduli space of $Z_P$-stable objects of arbitrary nonzero ranks as an Artin stack.
Recall the slope stability  for coherent sheaves corresponding to $\omega$ uses the slope function
\[
  \mu_\omega(E) = \begin{cases} \frac{\omega^3  \Ch_1(E)}{\Ch_0(E)} &\text{ if $\Ch_0(E)\neq 0$.} \\
  \infty &\text{ if $\Ch_0(E)=0$,}
  \end{cases}
\]
We say a nonzero coherent sheaf $E$ on $X$ is $\mu_\omega$-(semi)stable if, for every short exact sequence of sheaves $0 \to M\to E \to N \to 0$ on $X$ where $M, N \neq 0$, we have $\mu_\omega (M) < (\leq) \mu_\omega (N)$.

\begin{prop}\label{prop1}
Let  $(Z_P,\Ac^p)$ be a polynomial stability condition on $D^b(X)$ as in Proposition \ref{lem4}.  Suppose $E$ is an object of $D^b(X)$ with $\Ch(E)=(-r,-D,\gamma,\beta, n)$ where $r\neq 0$ and $r, \omega^3D$ are coprime (e.g. $r=-1, D=0$).  Then the following are equivalent:
\begin{itemize}
\item[(i)] $E$ is a $Z_P$-stable object in $\Ac^p$.
\item[(ii)] $E$ is a $Z_P$-semistable object in $\Ac^p$.
\item[(iii)] $H^{-1}(E)$ is a $\mu_\omega$-stable torsion-free sheaf, $H^0(E) \in\op{Coh}^{\leq 1}(X)$ and $\Hom (\op{Coh}^{\leq 1}(X), E)=0$.
\end{itemize}
\end{prop}

\begin{proof}
(i) $\Rightarrow$ (ii): clear.

(ii) $\Rightarrow$ (iii): Suppose $E$ is a $Z_P$-semistable object in $\Ac^p$.  The assumption \eqref{eq11} forces $H^{-1}(E)$ to be torsion-free.  For any short exact sequence of torsion-free sheaves of the form $0 \to F' \to H^{-1}(E) \to F'' \to 0$, we have an $\Ac^p$-short exact sequence $0 \to F'[1] \to H^{-1}(E)[1] \to F''[1] \to 0$.  This induces an $\Ac^p$-injection $F'[1] \hookrightarrow E$ and the $Z$-semistability of $E$ forces $\mu_\omega (F') \leq \mu_\omega (H^{-1}(E))$.  Hence $H^{-1}(E)$ is $\mu_\omega$-semistable.  The coprime condition on $r$ and $\omega^3 D$ then implies $H^{-1}(E)$ is $\mu_\omega$-stable.

Next, since $\op{Coh}^{\leq 1}(X)$ is a Serre subcategory of $\Ac^p$, assumption \eqref{eq10} and the stability of $E$ force $\Hom (\op{Coh}^{\leq 1}(X),E)=0$.

(iii) $\Rightarrow$ (i): Suppose $E$ satisfies all the conditions in (iii).  Take any $\Ac^p$-short exact sequence $0 \to M \to E \to N \to 0$ where $M, N \neq 0$.  We have three cases:
\begin{itemize}
\item If $\Ch_0(H^{-1}(M))=\Ch_0(H^{-1}(E))$, then $\dimension N \leq 3$ and $\phi (E) \prec \phi (N)$ by \eqref{eq10} or \eqref{eq11}.
\item If $0 < \Ch_0(H^{-1}(M))<\Ch_0(H^{-1}(E))$ then the $\mu_\omega$-stability of $H^{-1}(E)$ and the coprime assumption on $r,\omega^3D$ imply that $\mu_\omega (H^{-1}(M)) < \mu_\omega (H^{-1}(E))$, which in turn gives $\phi (M) \prec \phi (E)$.
\item If  $\Ch_0(H^{-1}(M))=0$, then $M=H^0(M) \in \op{Coh}^{\leq 1}(X)$, which is impossible by the assumption $\Hom (\op{Coh}^{\leq 1}(X),E)=0$.
\end{itemize}
Hence $E$ is $Z$-stable in $\Ac^p$.
\end{proof}

\begin{lem}\label{lem:Hileqdopen}
Let $Y$ be a nonsingular projective variety and $0\leq d \leq \dim Y$  an integer.  Then for objects in $D^{\leq 0}_{\op{Coh} (Y)}$, the property
\[
  H^0(E) \in \op{Coh}^{\leq d}(Y)
\]
is an open property.
\end{lem}

\begin{proof}
The proof of \cite[Example A.4]{ABL} applies.
\end{proof}

For the next  result, we will use the fact that a subspace of a Zariski topological space is open if and only if it is constructible and is stable under  generization.  When $Y$ is a nonsingular projective variety and $R$ is a discrete valuation ring with uniformizer $\eta$, we will write $k=R/(\eta)$ to denote the residue field and $K$ to denote the field of fractions of $R$.  We will also  write  $i : Y_k \hookrightarrow Y_R$ and  $j : Y_K \hookrightarrow Y_R$ to denote the closed and open immersions of the central and generic fibers of $Y_R = Y \otimes R \to \op{Spec} R$, respectively.

\begin{prop}\label{prop:PTopen4fold}
Suppose $(Z_P,\Ac^p)$ is a polynomial stability condition on $D^b(X)$ as in Proposition \ref{lem4}, and  $\Ch':=(-r, -D, \gamma, \beta, \xi)$ is a Chern character vector on $X$, where $r\neq 0, \omega^3D$ are coprime.  Then, for families of objects in $\Ac^p$ in class $\Ch'$, being $Z_P$-stable is an open property.
\end{prop}

For convenience, let us label the following properties for objects in $\Ac^p$:

\begin{itemize}
\item[A.] $H^{-1}(E)$ is torsion-free and semistable in $\op{Coh}_{4,2}(X)$.
\item[B.] $H^0(E) \in \op{Coh}^{\leq 1}(X)$.
\item[C.] $\Hom_{D^b(X)}(A,E)=0$ for all $A \in \op{Coh}^{\leq 1}(X)$.
\end{itemize}
and
\begin{itemize}
\item[A'.] $E$ lies in $\langle \op{Coh}^{= 4}(X)[1], \op{Coh}^{\leq 3}(X)\rangle$.
\item[A''.] $H^{-1}(E)$ is semistable in $\op{Coh}_{4,2}(X)$.
\end{itemize}

Here, $\op{Coh}_{4,2}(X)$ denotes the quotient category $\op{Coh}(X)/\op{Coh}^{\leq 2}(X)$.  To every object $F$ in $\op{Coh}_{4,2}(X)$, we can associate an analog $p_{4,2}(F)$ of the reduced Hilbert polynomial which defines a notion of semistability on $\op{Coh}_{4,2}(X)$ -- see \cite[Definition 1.6.3]{HL}.

\begin{proof}
By Proposition \ref{prop1}, for objects of nonzero ranks in $\Ac^p$, being $Z_P$-stable is equivalent to having all of properties A, B, and C, which in turn is equivalent to having all of properties A', A'', B, and C.  Note that property A' itself is an open property for flat families of objects in $D^b(X)$ by \cite[Example A.4]{ABL}, while property B is an open property by Lemma \ref{lem:Hileqdopen}.

Next, we claim that for flat families of objects in $\Ac^p$, property C is an open property.  To see this, note that if we write $-^\vee$ to denote the derived dual functor $R\mathcal{H}om_{D^b(X)}(-,\sO_X)$, then for any $0 \leq d \leq 4$  we have
\[
(\op{Coh}^{=d}(X))^\vee \subseteq \begin{cases} D^{[4-d,3]}_{\op{Coh} (X)} &\text{ if $d=2, 3, 4$}\\
\op{Coh}^{=d}(X)[-(4-d)] &\text{ if $d=0, 1$}
\end{cases}
\]
by \cite[Proposition 1.1.6]{HL}.  As a result, we have  $(\Ac^p)^\vee \subset D^{\leq 4}_{\op{Coh} (X)}$.  Then for any $E \in \Ac^p$, the vanishing  $\Hom_{D^b(X)} (\op{Coh}^{\leq 1}(X),E)=0$ holds if and only if both of the following hold: \begin{enumerate}[(i)] \item $\Hom_{D^b(X)}  (\op{Coh}^{= 0}(X),E)=0$, \item $\Hom_{D^b(X)}  (\op{Coh}^{=1}(X), E)=0$. \end{enumerate}  In addition, for $d=0, 1$, the fact that derived dual is an involution up to a shift gives us  $$(\op{Coh}^{=d}(X))^\vee [4-d] = \op{Coh}^{=d}(X)$$ by the same argument as in \cite[Lemma 5.6]{bridgeland2011hall}.  Therefore,  (i) is equivalent to $$\Hom_{D^b(X)} (E^\vee, \op{Coh}^{= 0}(X)[-4])=0,$$ which in turn is equivalent to $H^4(E^\vee)=0$, while (ii) is equivalent to $$\Hom_{D^b(X)}  (E^\vee, \op{Coh}^{=1}(X)[-3])=0.$$  Overall, for any $E \in \Ac^p$, the property $\Hom_{D^b(X)}  (\op{Coh}^{\leq 1}(X),E)=0$ holds if and only if both $H^4(E^\vee)=0$ and $H^3(E^\vee) \in \op{Coh}^{= 0}(X)$ hold.  Now, derived dual preserves flat families of objects in $D^b(X)$ while $H^4(-)=0$ is an open property for  objects in $D^{\leq 4}_{\op{Coh} (X)}$; in addition,  the property  $H^3(-)\in \op{Coh}^{=0}(X)$ is an open property for objects in $D^{\leq 3}_{\op{Coh} (X)}$ by Lemma \ref{lem:Hileqdopen}.     Therefore, C is an open property for objects in $\Ac^p$.


To finish the proof, it suffices to  show that for a flat family of objects in $\Ac^p$ that already satisfy A', B, and C,  property A'' is an open property.  Suppose $R$ is a discrete valuation ring.  Let $E$ be a flat family of objects in $\Ac^p$ over $\op{Spec} R$ such that the central fiber $Li^\ast E$ satisfies A', A'', B, and C.  If $H^{-1}(j^\ast E)$ is not semistable in $\op{Coh}_{4,2}(X)$, then it has a maximal destabilising subsheaf $F$ that is necessarily torsion-free (since A' is an open property) and semistable in $\op{Coh}_{4,2}(X_K)$.  Now we can proceed as in the proof of openness of PT stability on threefolds  \cite[Proposition 3.1]{Lo2}: we  extend $F$ to an $R$-flat family of sheaves $F_R$ such that $i^\ast F_R$ is semistable in $\op{Coh}_{4,2}(X_k)$ \cite[Theorem 2.B.1]{HL}, and then use it to produce a destabilising subsheaf of  $H^{-1}(Li^\ast E)$, giving us a contradiction.

Lastly, suppose $T$ is a Noetherian scheme and  $E \in D^b(X \times T)$ is a $T$-flat family of objects in $\Ac^p$ such that $E|^L_t$ already satisfies A', B, and C for all $t \in T$.  It remains to show that  the locus
\begin{equation}\label{eq:App}
  \{ t \in T : E|^L_t \text{ satisfies } A''\}
\end{equation}
is constructible.  Using a flattening stratification and replacing $T$ by a locally closed subscheme if necessary, we can assume each cohomology sheaf $H^i(E)$ is flat over $T$.  Then the openness of A'' follows from the openness of semistability in $\op{Coh}_{4,2}(X)$ for sheaves \cite[Proposition 2.3.1]{HL}.
\end{proof}

Proposition \ref{prop:PTopen4fold} together with Lieblich's result \cite[Corollary A.9]{ABL} now imply that

\begin{thm} \label{thm_exist}Suppose that $r, \omega^3 D$ are coprime. There is an Artin stack $\sM^{Z_P}(X,\Ch')$ locally of finite presentation over $\mathbb{C}$ that parametrizes $Z_P$-semistable objects with Chern character $\Ch'=(-r, -D, \gamma, \beta, n)$. 
\end{thm} \qed

In the next Subsection, we prove that the families of $Z_P$-semistable objects are bounded and hence the Artin stack $\mathcal{M}^{Z_P}(X,\Ch')$ is in fact of finite type over $\mathbb C$.

Let $l\colon \op{Spec} \bb C\to \op{Pic} X$ correspond to a line bundle $L$ on $X$, and $$\sM_L^{Z_P}(X,\Ch'):=\sM^{Z_P}(X,\Ch')\times_{\det, \op{Pic}X, l} \op{Spec} \bb C$$ be the moduli space of $Z_P$-stable objects in class $\Ch'$ with fixed determinant $L$. Since $Z_P$-stable objects are in particular simple (see Lemma \ref{lem_simple}) the objects of $\sM_L^{Z_P}(X,\Ch')$ have stabilizers isomorphic to $\mu_r$. In particular, we have

\begin{cor} \label{cor_mod} When $r=1$ $\sM_L^{Z_P}(X,\Ch')$ is a finite type algebraic space over $\mathbb{C}$.
\end{cor}\qed

\subsection{Boundedness of family of $Z_P$-semistable objects}
The argument in this Subsection implies that the Artin stack $\mathcal{M}^{Z_P}(X,\Ch')$ in Theorem \ref{thm_exist}  is of finite type over $\mathbb{C}$.  The strategy of the proof is to first show the boundedness of the set of all possible $H^{-1}(E)$ for $Z_P$-semistable objects $E$ of Chern character $\Ch'$, and then show the boundedness for all possible $H^0(E)$.  

For $H^{-1}(E)$, we adapt an argument of Maruyama's for showing boundedness of semistable sheaves in \cite{nmj/1118786090, 10.1215/kjm/1250521908}.   Even though Maruyama's original argument that is the most relevant to us was written for  reflexive sheaves, which $H^{-1}(E)$ is not, it can be adapted to our case since $H^{-1}(E)$ is ``close to being reflexive.''

Once we have boundedness for $H^{-1}(E)$, we use a long exact sequence of cohomology to produce morphisms of the form
\[
  \mathcal{E}xt^1(H^{-1}(E),\mathcal O_X) \to \mathcal{E}xt^3(H^0(E),\mathcal O_X).
\]
We then show that such pairs are stable pairs in the sense of Lin \cite{lin2018moduli} (and also Wandel \cite{wandel2015moduli}), and hence are bounded.  The boundedness of $H^0(E)$ then follows.

We first work out some properties of $H^{-1}(E)$ for a PT semistable object $E \in \Ac^p$ of nonzero rank:

\begin{lem}\label{lem:bdd-1}
Let $(Z_P, \Ac^p)$ be as in Proposition  \ref{lem4}.  Suppose $E \in \Ac^p$ is a $Z_P$-semistable object of nonzero rank, and let $F = H^{-1}(E)$.  Then $F$ has homological dimension at most 2, and
\[
  H^0(F^\vee) \in \op{Coh}^{=4}(X), \text{\quad} H^1(F^\vee) \in \op{Coh}^{\leq 2}(X), \text{\quad} H^2(F^\vee) \in \op{Coh}^{=0}(X).
\]
\end{lem}

\begin{proof}
From the proof of Proposition 2.6, we know $H^{-1}(E)$ is a torsion-free sheaf and the vanishing $\Hom (\op{Coh}^{\leq 1}(X), E)=0$ holds.  Since $\op{Coh}^{\leq 1}(X)$ is a Serre subcategory of $\Ac^p$ and $H^{-1}(E)[1]$ is an $\Ac^p$-subobject of $E$, it follows that $\Hom (\op{Coh}^{\leq 1}(X), H^{-1}(E)[1])=0$.  Let us write $F = H^{-1}(E)$.

From the proof of Proposition 2.8, we have $$(\op{Coh}^{=1}(X))^\vee \cong \op{Coh}^{= 1}(X)[-3]\quad \text{and}\quad (\op{Coh}^{= 0}(X))^\vee \cong\op{Coh}^{= 0}[-4],$$ so we have the vanishings
\begin{align*}
  \Hom (F^\vee, \op{Coh}^{=1}(X)[-2])=0, \\
  \Hom (F^\vee, \op{Coh}^{=0}(X)[-3])=0.
\end{align*}
Since $F$ is torsion-free, we have $F^\vee \in D^{[0,3]}_{\op{Coh} (X)}$.  Then the second vanishing implies $H^3(F^\vee)=0$, i.e.\ $F$ has homological dimension at most 2.  The first vanishing then  implies $H^2(F^\vee) \in \op{Coh}^{=0}(X)$. That $H^1(F^\vee)\in\op{Coh}^{\leq 2}(X)$ follows from the purity of $F$ \cite[Proposition 1.1.10]{HL} while $H^0(F^\vee)\cong \mathcal{E}xt^0(F,\mathcal O_X)$ is again torsion-free.
\end{proof}

\begin{lem}\label{lem:bdd-2}
Let $V$ be a smooth projective variety of dimension $n \geq 4$.  Suppose $F$ is a torsion-free sheaf of homological dimension at most 2.
\begin{itemize}
\item[(i)] Then $H^0(V, F(-l))=0=H^1(V,F(-l))$  for $l \gg 0$.
\item[(ii)] If we further assume $H^2(F^\vee) \in \op{Coh}^{= 0}(V)$, then for any positive integer $m$ and a generic smooth hypersurface $Y \in |\mathcal O_V(m)|$, the restriction $F|_Y$ is a  torsion-free sheaf of homological dimension at most 1 and satisfies
    \[
      H^0(Y, (F|_Y)(-l))=0=H^1(Y, (F|_Y)(-l)) \text{ for } l \gg 0.
    \]
\end{itemize}
\end{lem}

\begin{proof}
(i) For any $0 \leq i \leq n$ we have
\[
  H^i(V,F(-l)) \cong \Ext^i (\mathcal O_V, F(-l)) \cong \Ext^{n-i}(F,\omega_V (l))
\]
which is isomorphic to $H^0(V,\mathcal{E}xt^{n-i}(F,\omega_V(l)))$ for $l \gg 0$.  Since $F$ has homological dimension at most 2, the sheaf $\mathcal{E}xt^{n-i}(F,\omega_V(l))$ is zero when $n-i >2$, i.e.\ $i<n-2$, proving (i).

(ii) Suppose $$0 \to F^{-2} \to F^{-1} \to F^0 \to F \to 0$$ is  a locally free resolution of $F$ on $V$.  Let us write $F^\bullet$ to denote the complex $F^{-2} \to F^{-1} \to F^0$ with $F^i$ sitting at degree $i$, and let $Y \in |\mathcal O_V (m)|$ be a generic smooth hypersurface so that the restricted sequence
\[
0 \to F^{-2}|_Y \to F^{-1}|_Y \to F^0|_Y \to F|_Y \to 0
\]
is exact on $Y$, with $F|_Y$ torsion-free on $Y$ and $\rank_V (F)=\rank_Y(F|_Y)$.  
Then $F^\bullet$ is isomorphic to $F$ in $D^b(V)$ while $F^\bullet|_Y = [F^{-2}|_Y \to F^{-1}|_Y \to F^0|_Y]$ is isomorphic to $F|_Y$ in $D^b(Y)$.  If we let $f$ denote the closed immersion $Y \hookrightarrow  V$, then this means $$Lf^\ast F \cong F^\bullet |_Y\cong F|_Y=f^\ast F$$ in $D^b(Y)$.

On the other hand, we have
\[
  (Lf^\ast F)^\vee \cong R\mathcal{H}om_Y (Lf^\ast F, \mathcal O_Y) \cong Lf^\ast R\mathcal{H}om_V(F,\mathcal O_V) = Lf^\ast (F^\vee).
\]
Taking cohomology gives the last isomorphism in
\begin{equation}\label{eq:AG49-134-1}
  \mathcal{E}xt^i_Y(f^\ast F,\mathcal O_Y) \cong H^i( (f^\ast F)^\vee) \cong H^i ((Lf^\ast F)^\vee) \cong L^if^\ast (F^\vee).
\end{equation}
Plugging in $F^\vee$ in  the spectral sequence \cite[(3.10)]{huybrechts2006fourier}
\[
  E_2^{p,q} = L^pf^\ast (H^q (-)) \Rightarrow L^{p+q}f^\ast (-)
\]
and using \eqref{eq:AG49-134-1}, we obtain
\[
  E_2^{p,q} = L^pf^\ast \mathcal{E}xt^q_V (F,\mathcal O_V) \Rightarrow \mathcal{E}xt^{p+q}_Y(f^\ast F,\mathcal O_Y).
\]
By assumption, $E_2^{p,q}=0$ for $q>2$ or $p>0$, so $E_2^{0,2} = f^\ast \mathcal{E}xt^2(F,\mathcal O_V)$ already equals $E_\infty^{0,2}$ which in turn is isomorphic to $\mathcal{E}xt^2 (f^\ast F,\mathcal O_Y)$.  Since $Y$ is a generic smooth hypersurface of $V$ and $\mathcal{E}xt^2 (F,\mathcal O_V)$ is supported in dimension 0, we can further assume $Y$ is disjoint from the support of $\mathcal{E}xt^2 (F,\mathcal O_V)$, in which case $\mathcal{E}xt^2 (f^\ast F,\mathcal O_Y)$ vanishes, i.e.\ $f^\ast F$ has homological dimension at most 1.  The same argument as in part (i) then shows the desired  cohomology vanishing.
\end{proof}

We introduce some of Maruyama's notations for  families of sheaves: suppose $f : V \to S$ is a smooth, projective, geometrically integral morphism of noetherian schemes over $\mathbb{C}$, and $\mathcal O_V(1)$ is an $f$-very ample line bundle on $V$.  Also suppose that the dimensions of the fibers of $f$ are constant and equal $n$.  Then for any coherent sheaf $E$ on a fiber $V_s$ of $f$, we can write the Hilbert polynomial of $E$ as
\[
 \chi (E(m)) = \sum_{i=0}^n a_i(E) \binom{m+n-1}{n-i}
\]
for some integers $a_i (E)$ for $0 \leq i \leq n$.  Note that $a_0=d\Ch_0(E)$ where $d$ is the degree of $V_s$ with respect to $\mathcal O_{V_s}(1)$.

For any integers $r, a_1, \dots, a_n$, we now set $$\Sigma_{V/S}(n,r,a_1, \dots, a_n)$$ to be the family of isomorphism classes of coherent sheaves $E$ on the fibers $V_s$ of $V$ over $S$ such that
     \begin{enumerate}
     \item $E$ is slope semistable with respect to $\mathcal O_{V_s}(1)$,
     \item $a_0(E)=rd, a_1(E)=a_1, a_i(E) \geq a_i$ for all $2 \leq i \leq n$.
     \end{enumerate}
For any integers $r, a_1, a_2$, we set $$\Sigma_{V/S}'(n,r,a_1, a_2)$$ to be the family of isomorphism classes of coherent sheaves $E$ on the fibers $V_s$ of $V$ over $S$ such that
     \begin{enumerate}
     \item $E$  is slope semistable with respect to $\mathcal O_{V_s}(1)$,
     \item $E$ satisfies Serre's criterion $S_2$ (i.e.\ $E$ is reflexive),
     \item $a_0(E)=rd, a_1(E)=a_1, a_2(E) \geq a_2$.
     \end{enumerate}
  Given a coherent sheaf $E$ of rank $r$ on $V_s$ and a sequence $(\alpha) = (\alpha_1, \dots, \alpha_{r-1})$ of rational numbers, we say $E$ is of type $(\alpha)$ (in the sense of Maruyama \cite[Definition 1.3(1)]{10.1215/kjm/1250521908}) if $E$ is torsion-free, and for every subsheaf $F$ of $E$ with $1 \leq \Ch_0(F)=r_0 \leq r-1$, we have
\[
  \mu (F)\leq \mu (E) + \tfrac{r-r_0}{r_0}\alpha_{r-r_0}
\]
where $\mu$ is the slope computed with respect to $\mathcal O_{V_s}(1)$.     
     For any integers $r, a_1, \cdots, a_n$ and any sequence of rational numbers $(\alpha) = (\alpha_1, \dots, \alpha_{r-1})$, we set $$T_{V/S}(n,r,a_1, \dots, a_n, (\alpha))$$ to be the family of isomorphism classes of coherent sheaves $E$ on the fibers $V_s$ of $V$ over $S$ such that
     \begin{enumerate}
     \item $E$ is of type $(\alpha)$ with respect to $\mathcal O_{V_s}(1)$,
     \item $a_0(E)=rd, a_1(E)=a_1, a_i(E) \geq a_i$ for all $2 \leq i \leq n$.
     \end{enumerate}
In particular, when $(\alpha) = (0,\dots, 0)$, a torsion-free coherent sheaf on $V_s$ is of type $(\alpha)$ if and only if it is slope semistable with respect to $\mathcal O_{V_s}(1)$ \cite[Definition 1.3(2)]{10.1215/kjm/1250521908}, and  $T_{V/S}(n,r,a_1, \cdots, a_n, (\alpha))$ coincides with $\Sigma_{V/S}(n,r,a_1,\dots, a_n)$ in this case.  We also define the following statements for fixed $n, r$:
\begin{itemize}
\item $B_{n,r}$: $\Sigma_{V/S}(n,r,a_1,\dots,a_n)$ is bounded for all $f : V \to S$, $\mathcal O_V(1)$ and $a_1, \dots, a_n$,
\item $B_{n,r}'$: $\Sigma_{V/S}'(n,r,a_1,a_2)$ is bounded for all $f : V \to S$, $\mathcal O_V(1)$ and $a_1, a_2$,
\item $L_{n,r}$: $T_{V/S}(n,r,a_1, \dots, a_n, (\alpha))$ is bounded for all $f : V \to S$, $\mathcal O_V(1)$, $a_1, \dots, a_n$ and $(\alpha)$.
\end{itemize}

In adapting Maruyama's arguments for boundedness of moduli of sheaves to our situation, we need the following variations of the definitions above:

For any integers $r, a_1, a_2, a_3$ and sequence of rational numbers $(\alpha)=(\alpha_0,\dots, \alpha_{r-1})$, we set $$\Sigma_{V/S}^\ast (n,r,a_1, a_2, a_3,(\alpha))$$ to be the family of isomorphism classes of coherent sheaves $E$ on the fibers $V_s$ of $V$ over $S$ such that

     \begin{enumerate}
     \item $E$ is of type $(\alpha)$  with respect to $\mathcal O_{V_s}(1)$,
     \item $H^0(V_s, E(-l))=0=H^1(V_s, E(-l))$ for $l \gg 0$,
     \item $a_0(E)=rd, a_1(E)=a_1, a_2(E)=a_2, a_3(E) \geq a_3$.
     \end{enumerate}
Then we define the statement
\begin{itemize}
\item $B_{n,r}^\ast$: $\Sigma_{V/S}^\ast (n,r,a_1,a_2, a_3,(\alpha))$ is  bounded for all $f : V \to S$, $\mathcal O_V(1)$, $a_1, a_2, a_3$ and $(\alpha)$.
\end{itemize}

The proof of the following  Proposition is modeled after the proof of \cite[Proposition 2.5]{10.1215/kjm/1250521908}, which in turn is modeled after the proof of \cite[Proposition 3.6]{nmj/1118786090}.
\begin{prop}\label{prop:bdd-1}
Suppose $n \geq 4$,  and $\mathscr{F}^\ast$ is a subfamily of $\Sigma^\ast_{V/S} (n,r,a_1, a_2, a_3, (\alpha))$.  Also, suppose there exists a sequence of rational numbers $(\alpha') = (\alpha_1', \cdots, \alpha_{r-1}')$ such that    each member $E \in \mathscr{F}^\ast$ satisfies:
\begin{enumerate}
\item If $E$ is a coherent sheaf on the fiber $V_s$, then $E$ is torsion-free with homological dimension at most 2 on $V_s$ and $H^2(E^\vee) \in \op{Coh}^{= 0}(V_s)$.
\item For a generic smooth hyperplane $Y \in |\sO_{V_s}(1)|$, the restriction $E|_Y$ is of type $(\alpha')$ with respect to $\sO_Y(1)$.
\end{enumerate}
If $B_{n-1,r}^\ast$ holds, then $\mathscr{F}^\ast$ is bounded.
\end{prop}

\begin{proof}
The proof of \cite[Proposition 3.6]{nmj/1118786090} contains a case for a subfamily $\mathscr{F}'$ of $\Sigma_{V/S}'(n, r, a_1, a_2)$.  The argument for that case   applies here, with the family $\Sigma_{V/S}'(n,r,a_1,a_2)$ replaced with $\Sigma^\ast_{V/S} (n,r,a_1, a_2, a_3, (\alpha))$, the statement $B_{n-1,r}'$ replaced with $B_{n-1,r}^\ast$, the subfamily $\mathscr{F}'$ replaced with $\mathscr{F}^\ast$, and \cite[Lemma 3.7]{nmj/1118786090} replaced with Lemma 2.12.  Also, in Maruyama's proof where he uses the fact that $H^0(V_s, E(l))=0$ for a $\mu$-semistable sheaf with $d(E(l), \sO_{V_s}(1))<0$ (where $d(A,\sO_{V_s}(1))$ denotes the degree of a sheaf $A$ with respect to $\sO_{V_s}(1)$), we replace it with  the fact that $H^0(V_s,F)=0$ when $F$ is a sheaf of type $(\beta)$ and $d(F,\sO_{V_s}(1))<-(r-1)r\beta_{r-1}$ \cite[Lemma 1.7]{nmj/1118795443}.
\end{proof}

We can now bound the set of all possible $H^{-1}(E)$ for  $Z_P$-semistable objects $E$ of a fixed Chern character. More generally, we have the following result.

\begin{cor}\label{coro:bdd-1}
 Let $(Z_P, \Ac^p)$ be as in Proposition \ref{lem4}.  Fix $r\ge 1$ and $(D,\gamma,\beta)\in \oplus_{i=1}^3 H^{2i}(X,\bb Q)$.  Then the family of sheaves
 \begin{equation}\label{eq:HleftEset}
  \{ H^{-1}(E) : E  \text{ is $Z_P$-semistable in $\Ac^p$ with }\Ch_{i \le 3}(E)=(-r,-D,\gamma,\beta)\}
 \end{equation}
 is bounded.
\end{cor}

\begin{proof}
Given any $Z_P$-semistable object $E$ in $\Ac^p$ of the prescribed Chern character, we know that $H^{-1}(E)$ is torsion-free and $\mu_\omega$-semistable from the proof of Proposition 2.6.  By Lemmas \ref{lem:bdd-1} and \ref{lem:bdd-2}, we also know that $H^{-1}(E)$ satisfies the conclusions in both parts of Lemma \ref{lem:bdd-2}.  In addition, $$\Ch_3(H^{-1}(E)[1]) = \Ch_3(E)-\Ch_3(H^0(E))$$ and so  $\omega\cdot \Ch_3(H^{-1}(E)) \geq -\omega\cdot \Ch_3(E)$ because  $H^0(E) \in \op{Coh}^{\leq 1}(X)$.  That is, in the Hilbert polynomial of $H^{-1}(E)$ (with respect to $\omega$) the coefficient $a_3(H^{-1}(E))$ is bounded from below by some constant $a_3$ that  depends only on $\omega$ and $\beta$.  If we let $a_1, a_2$ denote the coefficients  $a_1(H^{-1}(E)), a_2(H^{-1}(E))$, then  $H^{-1}(E)$ is a member of $\Sigma_{X/\mathbb{C}}^\ast (4,r,a_1, a_2, a_3, (0))$ (where $(0)=(0, \cdots, 0)$) and satisfies condition (1) in Proposition \ref{prop:bdd-1}.

By \cite[Corollary 4.7]{10.1215/kjm/1250521908}, there exists a sequence of rational numbers $(\alpha')=(\alpha_1', \cdots, \alpha_{r-1}')$ depending only on $r$ and $\omega$, such that $H^{-1}(E)$ also satisfies condition (2) in Proposition 2.13. Therefore,  the boundedness of \eqref{eq:HleftEset} would follow from the boundedness of $B_{3,r}^\ast$ by Proposition \ref{prop:bdd-1}.

Since $B_{3,r}^\ast$ is a special case of $L_{3,r}$, which holds by \cite[Theorem 4.8]{10.1215/kjm/1250521908}, we are done.
\end{proof}

Next, we will bound the set of all possible $H^0(E)$ for $Z_P$-semistable objects of a fixed Chern character.  The strategy is to first show  the  $\mathcal{E}xt^i (H^{-1}(E),\mathcal O_X)$ are bounded, and then show the $\mathcal{E}xt^i (H^0(E),\mathcal O_X)$ are bounded.

\begin{lem}\label{lem:bdd-3}
Let $V$ be a smooth projective variety.  Suppose $\{ E_t \}_{t \in T}$ is a bounded family of complexes in $D^b(V)$.  Then for any $i$, the family of sheaves $\{ H^i (E_t)\}_{t \in T}$ is bounded.
\end{lem}

\begin{proof}
We can assume that there exists a scheme $T$ of finite type over $\mathbb C$, together with a complex of locally free sheaves
\[
 G := [ 0 \to G^m \to G^{m+1} \to \cdots \to G^n \to 0]
\]
on $V \times T$ with $G^j$ sitting at degree $j$ in $D^b(V \times T)$, such that each $E_t$ is isomorphic to $G|^L_t$ for some $t \in T$.  Without loss of generality, let us assume $n=0$.

To show the boundedness of $\{ H^0(G|^L_t)\}_{t \in T}$, note that $H^0(G|^L_t) \cong H^0 (G)|_t$ for all $t$.  Since $T$ is of finite type, we have boundedness of $\{ H^0(G|^L_t)\}_{t \in T}$.  Applying \cite[Lemma 3.16]{TODA20082736}, the family of complexes $\{  \tau^{\leq -1}(G|^L_t) \}_{t \in T}$ (where $\tau^{\leq -1}$ denotes the truncation functor with respect to the standard t-structure) is also bounded.  Repeating this process, we can show that $\{ H^i (G|^L_t)\}_{t \in T}$ is bounded for any  $i$, and the lemma follows.
\end{proof}

\begin{thm}\label{prop:bdd-2}
 Let $(Z_P, \Ac^p)$ be as in Proposition \ref{lem4}.  Fix a Chern character $\Ch' = (-r, -D, \gamma, \beta, \xi)$ on $X$ where $r \geq 1$.  Then the family of complexes
 \begin{equation}\label{eq:bdd-0}
  \{ E \in \Ac^p : \text{$E$  is $Z_P$-semistable in $\Ac^p$ with $\Ch(E)=\Ch'$}\}
 \end{equation}
 is bounded.
\end{thm}

\begin{proof}
For any $Z_P$-semistable object $E$ with the prescribed Chern character, we know from Lemma \ref{lem:bdd-1} that  $(H^{-1}(E))^\vee \in D^{[0,2]}_{\op{Coh}(X)}$.  On the other hand, if we fit $H^0(E)$ in a short exact sequence of sheaves $$0 \to Q \to H^0(E) \to T \to 0$$ where $Q \in \op{Coh}^{= 0}(X)$ and $T \in \op{Coh}^{=1}(X)$,  we obtain the exact triangle $$T^\vee \to (H^0(E))^\vee \to Q^\vee \to T^\vee [1]$$ from which we see $H^0(E)^\vee \in D^{[3,4]}_{\op{Coh}(X)}$ with $\mathcal{E}xt^3 (H^0(E),\mathcal O_X) \cong \mathcal{E}xt^3 (T,\mathcal O_X)$ is pure 1-dimensional, and $\mathcal{E}xt^4 (H^0(E),\mathcal O_X) \cong \mathcal{E}xt^4 (Q, \mathcal O_X)$ is 0-dimensional.  Therefore, dualizing the exact triangle
\[
  H^{-1}(E)[1]\to E \to H^0(E) \to H^{-1}(E)[2],
\]
we obtain the long exact sequence of sheaves
\[
\xymatrix{
  0 \ar[r] & H^1 (E^\vee) \ar[r] & \mathcal{E}xt^0 (H^{-1}(E),\mathcal O_X) \ar `d[ll] `[dll] [dll]  \\
  0 \ar[r] & H^2 (E^\vee) \ar[r] & \mathcal{E}xt^1 (H^{-1}(E),\mathcal O_X) \ar `d[ll] `[dll]_\alpha [dll] \\
  \mathcal{E}xt^3 (H^0(E), \mathcal O_X) \ar[r]^(0.6)\beta & H^3(E^\vee) \ar[r] & \mathcal{E}xt^2 (H^{-1}(E),\mathcal O_X) \ar `d[ll] `[dll] [dll] \\
  \mathcal{E}xt^4 (H^0(E), \mathcal O_X) \ar[r] & H^4(E^\vee ) \ar[r] & 0
}.
\]

From Corollary \ref{coro:bdd-1} and Lemma \ref{lem:bdd-3}, we see that for $i=0, 1, 2$, the family of sheaves
\begin{equation}\label{eq:bdd-2}
\{ \mathcal{E}xt^i (H^{-1}(E),\mathcal O_X) : \text{ $E \in \Ac^p$ is $Z_P$-semistable with $\Ch(E)=\Ch'$} \}
\end{equation}
is bounded.  Recall that Lemma \ref{lem:bdd-1} also tells us $\mathcal{E}xt^2 (H^{-1}(E), \mathcal O_X)$ is a 0-dimensional sheaf, so there exists a constant $l_0$ such that the length of $\mathcal{E}xt^2 (H^{-1}(E),\mathcal O_X)$ is bounded from above by $l_0$ for any $Z_P$-semistable object in $\Ac^p$ with $\Ch (E)=\Ch'$.  On the other hand, $H^4(E^\vee)=0$ from the proof of Proposition 2.8.  Hence
\begin{equation}\label{eq:bdd-3}
  \{ \mathcal{E}xt^4 (H^0(E), \mathcal O_X) : \text{ $E$ is $Z_P$-semistable in  $\Ac^p$ with $\Ch(E)=\Ch'$} \},
\end{equation}
a family of 0-dimensional sheaves on $X$ with length bounded from above by $l_0$, is also bounded.

To bound the possible $\mathcal{E}xt^3 (H^0(E), \mathcal O_X)$, we first note that from the previous paragraph
\[
  \{ \mathcal{E}xt^1 (H^{-1}(E), \mathcal O_X) : \text{ $E$ is $Z_P$-semistable in $\Ac^p$  with $\Ch(E)=\Ch'$} \}
\]
is bounded; by Kleiman's boundedness criterion \cite[Theorem 1.13(iv)]{kleiman1971theoremes}, this means that  there exist integers $\rho, N$ such that every member of this family occurs as a quotient sheaf of $\mathcal O_X(-\rho)^{\oplus N}$.  Composing with $\alpha$, we see that every member of
\begin{equation}\label{eq:bdd-4}
\{ \mathcal{E}xt^3 (H^0(E),\mathcal O_X) : \text{ $E$ is $Z_P$-semistable  in $\Ac^p$  with $\Ch(E)=\Ch'$} \},
\end{equation}
is a pure 1-dimensional sheaf and occurs in a morphism of sheaves of the form
\begin{equation}\label{eq:bdd-1}
  \mathcal O (-\rho)^{\oplus N}  \overset{\alpha'}{\longrightarrow} \mathcal{E}xt^3 (H^0(E),\mathcal O_X)
\end{equation}
where $\op{cokernel} (\alpha') \in \op{Coh}^{= 0}(X)$ (since $H^3(E^\vee)\in\op{Coh}^{= 0}(X)$ from the proof of Proposition 2.8).  Now, since \eqref{eq:bdd-2} is bounded for all $i$ and \eqref{eq:bdd-3} is bounded, there is only a finite number of possibilities for the Hilbert polynomials with respect to $\omega$ for  members of \eqref{eq:bdd-4}.

For each such possibility of Hilbert polynomial $P$, the pair $(\mathcal{E}xt^3(H^0(E),\mathcal O_X), \alpha')$ is then  a stable pair of type $P$ for a `sufficiently large' parameter $\delta$ in the sense of \cite[Definition 2.6]{lin2018moduli} by \cite[Lemma 2.10]{lin2018moduli} (see also \cite[Proposition 1.12]{wandel2015moduli}).  Then by \cite[Proposition 3.3]{lin2018moduli}, the family \eqref{eq:bdd-4} is bounded.  Applying \cite[Lemma 3.16]{TODA20082736} now shows that
 \begin{equation*}
  \{ H^0(E)  : \text{$E$  is $Z_P$-semistable in $\Ac^p$ with $\Ch(E)=\Ch'$}\}
 \end{equation*}
is bounded; applying it once more together with Corollary \ref{coro:bdd-1} then shows that  \eqref{eq:bdd-0} is bounded.
\end{proof}

As a byproduct of the arguments above we prove the following Proposition that will be used in \S \ref{sec_Hall} (cf. Lemma \ref{bdab}).
\begin{prop} \label{prop_bdch4}
Let $(Z_P, \Ac^p)$ be as in Proposition \ref{lem4}.  Fix $r\ge 1$ and $(D,\gamma,\beta)\in \oplus_{i=1}^3 H^{2i}(X,\bb Q)$.  Suppose the subset of $\bb Q$ defined by 
$$\{ \Ch_4(\mathcal E xt^4(H^0(E), \sO_X)) : \text{$E$  is $Z_P$-semistable in $\Ac^p$ with } \Ch_{i\le 3}(E)=(-r,-D,\gamma, \beta)\}$$
 is bounded above. Then, the subset of $\bb Q$ defined by 
 $$\{ \Ch_4(E) : \text{$E$  is $Z_P$-semistable in $\Ac^p$ with } \Ch_{i\le 3}(E)=(-r,-D,\gamma, \beta)\}$$
 is also bounded above.
\end{prop}
\begin{proof}
By Corollary \ref{coro:bdd-1},  the set of possible $\omega \cdot \Ch_3(H^0(E)^\vee)=-\omega \cdot \Ch_3(H^0(E))$ is bounded, so is the set of possible $$\omega\cdot \Ch_3(\mathcal E xt^3(H^0(E), \sO_X)).$$ 
Proceeding as in the proof of Theorem \ref{prop:bdd-2}, we can find constants $p, N$ such that there is a morphism of sheaves
$$ \alpha' : \sO_X(-p)^{\oplus N} \to  \mathcal E xt^3(H^0(E), \sO_X)$$  with 0-dimensional cokernel.
Here, $ \mathcal E xt^3 (H^0(E), \sO_X)$ is a pure 1-dimensional sheaf.  Tensoring with $\sO_X(p)$, we get a morphism of sheaves
$$\alpha'' : \sO_X^{\oplus N} \to A''$$
with 0-dimensional cokernel and in which $A''$ is a pure 1-dimensional sheaf, with the set of possible $ \omega \cdot \Ch_3(A'')$ is still bounded.  We claim that the set of possible  $\Ch_4 (\op{im}(\alpha''))$ is bounded below.
If the claim is proven it will follow that the set of possible  $\Ch_4 (\op{im}(\alpha'))$ is also be bounded below, and hence the set of possible $$\Ch_4(\mathcal E xt^3(H^0(E), \sO_X))$$ is bounded below.  But we have \begin{align*} \Ch_4(E)&=-\Ch_4(H^{-1}(E))+\Ch_4(H^0(E))=-\Ch_4(H^{-1}(E))+\Ch_4(H^0(E)^\vee)\\&=-\Ch_4(H^{-1}(E))-\Ch_4(\mathcal E xt^3(H^0(E), \sO_X))+\Ch_4(\mathcal E xt^4(H^0(E), \sO_X)) ,\end{align*} so by what we said above and our assumption the Proposition follows.

To prove the claim above, it suffices to show that if $\sO_X^{\oplus N} \twoheadrightarrow G$  is a 1-dimensional quotient ($G$ is the image of $\alpha''$) such that $\Ch_3(G)$ is fixed (since $\omega\cdot \Ch_3(G)$ is bounded and $\Ch_3(G)$ is effective, there are only finitely many choices for $\Ch_3(G)$) then the set of possible $\Ch_4(G)$ is bounded below. Let $K$ be the kernel of this quotient. 
Since $K$ is torsion free, it is locally free on an open subset of $X$. So we can find a point $p \in X$  a surjection $K\twoheadrightarrow \sO_p$ with the kernel denoted by $K_1$. We get a short exact sequence
$$0\to K_1\to \sO_X^{\oplus N}\to G_1\to 0,$$
where $\Ch_3(G_1)=\Ch_3(G)$, but $\Ch_4(G_1)=\Ch_4(G)+1$. The rest of argument is as the proof of \cite[Proposition 3.6]{GK} in which the moduli space of stable sheaves is replaced by a suitable Quot scheme of 1-dimensional quotients of $\sO_X^{\oplus N}$.
%
\end{proof}

\subsection{Equivalence of moduli spaces}\label{subsec_deformation}

In this Subsection, we show that the moduli space of stable pairs $P(X,\Ch)$ is identified with a moduli space of $Z_P$-stable objects with trivial determinant.


Given a scheme $T$ of finite type over $\cc$, by a $T$-flat family of stable pairs we mean a complex of sheaves $[\sO_{X \times T} \to F]$ in $D^b(X \times T)$ sitting at degrees $-1, 0$ such that $F$ is flat over $T$, and every restriction $[\sO_X \to F_t]$ to the fiber over a closed point $t \in T$ is a stable pair. For any $J\in D^b(X\times T)$ we denote the derived restriction $J|^L_t\in D^b(X_t)$ by $J_t$, where $X_t:=X\times \{t\}$. 

\begin{prop} \label{prop-family}
Let $\Ch':=(-1,0,\gamma, \beta, \xi)$ be a Chern character vector with $\gamma$ an effective class, and  $T$ be a scheme of finite type over $\cc$.  Let $(Z_P,\sA^p)$ be a polynomial stability condition on $D^b(X)$ as in Proposition \ref{lem4}, and $J \in D^b(X \times T)$ an object with trivial determinant, such that for each closed point $t \in T$, we have $\Ch(J_t) =   \Ch$. Then $J$ is isomorphic in $D^b(X\times T)$ to a $T$-flat family of stable pairs if and only if $J$ is a $T$-flat family of $Z_P$-semistable objects in $\sA^p$.
\end{prop}

\begin{proof}
By Proposition \ref{lem4}, it is clear that every $T$-flat family of stable pairs is a $T$-flat family of $Z_P$-semistable objects in $\sA^p$.

For the converse, suppose $J \in D^b(X \times T)$ is a $T$-flat family of $Z_P$-semistable objects in $\sA^p$.  For each $t \in T$, we have $H^0(J)|_t \cong H^0(J_t) \in \op{Coh}^{\leq 1}(X)$, and so $H^0(J)$ is supported in codimension 3 on $X\times T$.  On the other hand, since $H^{-1}(J_t)$ is torsion-free for each $t \in T$, the sheaf $H^{-1}(J)$ itself is torsion free of rank 1 by \cite[Lemma 5.12]{Lo1}, so it is embedded in its double dual $H^{-1}(J)^{\ast \ast}$, which is also of rank 1.  By the determinant assumption we have $\det{H^{-1}(J)}\cong \det H^{-1}(J)^{\ast \ast}\cong \sO_{X \times T}$. 

Now let $U: = (X \times T) \setminus \mathrm{supp}(H^0(J))$.  Then $(J|_U)|^L_t \cong (H^{-1}(J)|_U)|^L_t[1]$.  We show this is a sheaf for every $t \in T$.  
 We need to prove that $(J|_U)|^L_t$ is a sheaf for each closed point $t\in T$.   First note that $H^0(J_t)\cong H^0(J)|_t$, and so $$\mathrm{supp}(H^0(J_t)) = \mathrm{supp}(H^0(J)|_t) =\mathrm{supp}(H^0(J))\cap X_t,$$ which is disjoint from $U$ by definition.  That is, $H^0(J_t|_{U \cap X_t})=0$.  From the exact triangle $$H^{-1}(J_t) [1] \to J_t \to H^0(J_t)$$ we see $J_t|_{U \cap X_t}$ is a torsion free sheaf (sitting at degree $-1$).  Hence  $$(J|_U)|^L_t \cong (J|^L_t)|_{U \cap X_t}=J_t|_{U \cap X_t}$$ is a sheaf.
 We now know $H^{-1}(J)$ is flat with torsion free rank 1 fibers over $T$ away from the codimension 3 locus $\mathrm{supp}(H^0(J))\subset X \times T$. Therefore, by \cite[Lemma 6.13]{kollar1990}, $H^{-1}(J)^{\ast \ast}$ is locally free away from $\mathrm{supp}(H^0(J))$.
Now, as in the proof of \cite[Lemma 6.13]{kollar1990} (and also proof of \cite[Theorem 2.7]{PT}), any reflexive
rank 1 sheaf that is locally free away from a codimension 3 locus is locally free globally.
Hence $H^{-1}(J)^{\ast \ast}$ is locally free on $X\times T$ and so $H^{-1}(J)^{\ast \ast}\cong \sO_{X \times T}$. As a result, $H^{-1}(J) = I_W$ is the ideal sheaf of some closed subscheme $W \subset X \times T$ of codimension 2. 

Writing $Q = H^0(J)$, from the canonical exact triangle $$I_W [1] \to J \to Q \to I_W [2]$$ we know $J$ is isomorphic to the mapping cone of some morphism $Q [-1] \to I_W [1]$, which represents a class $\alpha$ in $\Ext^2 (Q, I_W)$.  On the other hand, applying $\Hom (Q,-)$ to the short exact sequence $0 \to I_W \to \sO_{X \times T} \to \sO_W \to 0$ gives an isomorphism $\delta : \Ext^1 (Q, \sO_W)\to  \Ext^2 (Q, I_W)$ since we have the vanishings $\Ext^i(Q, \sO_{X \times T})=0$ for $i=1,2$.  This means that there exists a short exact sequence of sheaves on $X \times T$
\[
0 \to \sO_W \to F \to Q \to 0
\]
that represents an element of $\Ext^1 (Q,\sO_W)$ whose image under  $\delta$ is  $\alpha$.  In particular, the composition of the morphisms
\[
  \sO_{X \times T} \to \sO_W \to F
\]
gives a 2-term complex sitting at degrees $-1,0$ that is isomorphic to $J$ in $D^b(X \times T)$.  

It remains to show that $F$ is flat over $T$, which follows from the argument in \cite[Theorem 2.7]{PT}.  For the convenience of readers, we provide a proof here. 
Let $t \in T$ be an arbitrary  closed point, and let $\iota : X_t \hookrightarrow X \times T$ denote the corresponding closed immersion.  By \cite[Lemma 3.31]{huybrechts2006fourier}, in order to show $F$ is flat over $T$, it suffices to show that $L\iota^\ast F$ is a sheaf sitting at degree $0$ on $X_t$.  From the construction of $F$ we have an exact triangle in $D^b(X \times T)$
\[
J[-1] \overset{a}{\to} \sO_{X \times T} {\to} F \to J.
\]
On the other hand, since $J$ is a flat family of $Z_P$-semistable objects in $\sA^p$, by Proposition \ref{lem4} we also have the exact triangle in $D^b(X_t)$
\begin{equation}\label{eq12}
J_t[-1] \overset{b}{\to} \sO_{X_t} {\to} G \to J_t
\end{equation}
for some 2-dimensional coherent sheaf $G$ on $X_t$.  Applying $\Hom (-,\sO_{X_t})$ to \eqref{eq12} gives the exact sequence of $\Gamma (X_t,\sO_{X_t})$-modules
\[
\Hom (G,\sO_{X_t}) \to \Hom (\sO_{X_t},\sO_{X_t}) \to \Hom (J_t[-1], \sO_{X_t}) \to \Hom (G[-1], \sO_{X_t})
\]
where the first and last terms vanish by Serre duality and dimension reasons.  Hence $L\iota^\ast a$ and $b$ are both nonzero elements of $$\Hom (J_t[-1], \sO_{X_t})\cong \Hom (\sO_{X_t}, \sO_{X_t}) \cong \Gamma (X_t, \sO_{X_t}).$$  If we let $V := X_t \setminus ((\mathrm{supp} \, L\iota^\ast F) \cup (\mathrm{supp}\, G))$, then  both $L\iota^\ast a$ and $b$ restrict to the identity on $V$.  Since $L\iota^\ast F$ and $G$ are both supported in codimension 2, this means that $L\iota^\ast a$ and $b$ coincide outside a codimension-2 locus on $X_t$.  Hence $L\iota^\ast a$ and $b$  must coincide by Hartog's Theorem, meaning $L\iota^\ast F$ is a sheaf.
\end{proof}

By Proposition \ref{prop-family} and Corollary \ref{cor_mod} we have

\begin{thm} \label{thm_moduli} Let $X$ be a nonsingular projective 4-fold, and
 $$\Ch=(0,0,\gamma,\beta, \xi), \qquad \Ch'=(-1,0,\gamma,\beta, \xi)$$ be two Chern character vectors. Then,  $P(X,\Ch) \cong \mathcal M_{\sO_X}^{Z_P}(X,\Ch')$.
\end{thm}  \qed

\section{Stable pair/ideal correspondences}
As in \S \ref{sec_stability_moduli},  we let $X$ be a nonsingular projective 4-fold with a choice of a polarization  $\omega$.

\subsection{Polynomial stability and ideal sheaves}\label{subsec_polynomial_ideal_sheaves}

In this Subsection, we provide a polynomial stability that in particular recovers the ideal sheaves of 2-dimensional subschemes in the 4-fold $X$.

\begin{prop}\label{lem5}
 Suppose $p$ is the perversity function in Proposition \ref{lem4}, i.e.\
\begin{equation*}
p(0)=p(1)=0, \text{\quad} p(2)=p(3)=p(4)=-1,
\end{equation*}
 and that  $p$ is associated to a stability vector $\rho$ with
 \begin{equation}\label{eq14}
 \phi (-\rho_2)> \phi(-\rho_3) > \phi (-\rho_4) > \phi (\rho_0) > \phi (\rho_1).
 \end{equation}
 Let $Z_I$ be the resulting polynomial stability condition as in \cite[Theorem 3.2.2]{BayerPBSC}.  Then for any object $E \in D^b(X)$ with $\Ch(E) = (-r,-D,\gamma, \beta, n)$ where $r\neq 0$ and $r, \omega^3 D$ are coprime, the following are equivalent:
 \begin{itemize}
 \item[(i)]  $E[-1]$ is a $\mu_\omega$-stable torsion-free sheaf.
 \item[(ii)] $E$  is  a $Z_I$-stable object in  $\Ac^p$.
 \item[(iii)] $E$  is  a $Z_I$-semistable object in  $\Ac^p$.
 \end{itemize}
 In particular, when $\Ch(E)=(-1,0,\gamma,\beta, n)$ with $\gamma$ effective and $\det{E} =\sO_X$, the above conditions are equivalent to:
 \begin{itemize}
 \item[(iv)] $E[-1]$ is the ideal sheaf of a 2-dimensional closed subscheme of $X$.
 \end{itemize}
\end{prop}

\begin{proof}
Note that $\Ac^p = \langle \op{Coh}^{\geq 2}(X)[1], \op{Coh}^{\leq 1}(X) \rangle$ as in Proposition \ref{lem4}.  Suppose $E$ is an object of $D^b(X)$ with Chern character  as above.

(iii) $\Rightarrow$ (i): Suppose $E$ is a $Z_I$-semistable object in $\Ac^p$.  Then from \eqref{eq14}, $H^0(E)$ must vanish or else it would destabilize $E$.  Also, the $Z_I$-semistability of $E=H^{-1}(E)[1]$ and \eqref{eq14} together imply that $H^{-1}(E)$ is a torsion-free sheaf.  Given any short exact sequence of torsion-free sheaves $$0 \to F' \to H^{-1}(E) \to F'' \to 0,$$ we have an induced injection $F'[1] \hookrightarrow H^{-1}(E) [1]$ in $\Ac^p$.  The $Z_I$-semistability of $E$ then ensures $\mu_\omega (F') \leq \mu_\omega (H^{-1}(E))$, implying $H^{-1}(E)$ is $\mu_\omega$-semistable, which is equivalent to $\mu_\omega$-stable due to the coprime assumption.

(i) $\Rightarrow$ (ii):  Suppose  $E=H^{-1}(E)[1]$ where $H^{-1}(E)$ is a $\mu_\omega$-stable torsion-free sheaf, and take any $\Ac^p$-short exact sequence of the form $0 \to M \to E \to N \to 0$ where $M, N \neq 0$.  From the associated long exact sequence of sheaves
\[
  0 \to H^{-1}(M) \to H^{-1}(E) \to H^{-1}(N) \to H^0(M) \to 0
\]
we know $H^0(N)=0$. If $H^{-1}(M)=0$, then $M=H^0(M)$ is supported in dimension at most $1$, implying $\phi (M) \prec \phi (E)$.  If $$0<\Ch_0(H^{-1}(M))< \Ch_0(H^{-1}(E))$$ then the coprime assumption on $r, \omega^3 D$ implies $\mu_\omega (H^{-1}(M)) < \mu_\omega (H^{-1}(E))$ which forces $\phi (M) \prec \phi (E)$; moreover,  $H^{-1}(N)$ must be a torsion sheaf  supported in dimension 3 or 2 from which it follows that  $\phi (E) \prec \phi (N)$.  Hence $E$ is $Z_I$-stable in $\Ac^p$.

(ii) $\Rightarrow$ (iii): clear.

For a coherent sheaf $F$ with  $\Ch(F[1])=(-1,0,\gamma,\beta, n)$ where $\gamma$ is effective and $\det{F[1]} =\sO_X$, $F$ is torsion-free if and only if it is the ideal sheaf $I_W$ of a 2-dimensional closed subscheme $W$ of $X$ with $\Ch_2(\sO_W)=\gamma$.  This gives the equivalence with (iv).
\end{proof}

\begin{prop} \label{prop-family2}
Let $T$ be a scheme of finite type over $\cc$.  Let $(Z_I,\Ac^p)$ be a polynomial stability condition on $D^b(X)$ as in Proposition \ref{lem5}, and $I \in D^b(X \times T)$ an object such that, for each closed point $t \in T$,  $\Ch (I_t) = (-r,-D, \gamma,\beta, n)$ is fixed and in which $r\neq 0$ and $r, \omega^3 D$ are coprime.  Then $I[-1]$ is isomorphic in $D^b(X)$ to a $T$-flat family of $\mu_\omega$-stable torsion-free sheaves on $X$ if and only if $I$ is a $T$-flat family of $Z_I$-semistable objects in $\Ac^p$.  In particular, when $\Ch(I_t) = (-1,0,\gamma, \beta, n)$ where $\gamma$ is effective and $\det(I)=\sO_{X\times T}$,  $I[-1]$ is isomorphic in $D^b(X)$ to a $T$-flat family of ideal sheaves of 2-dimensional subschemes of $X$ if and only if $I$ is a $T$-flat family of $Z_I$-semistable objects in $\Ac^p$.
\end{prop}

\begin{proof}
Suppose $I=F[1]$, where $F$ is a $T$-flat family of  $\mu_\omega$-stable torsion-free sheaves on $X$.  By Proposition \ref{lem5}, for every closed point $t \in T$, the restriction  $I_t=F|^L_t[1] = F|_t[1]$ is a $Z_I$-semistable object  in $\Ac^p$.


For the converse, suppose $I$ is a $T$-flat family of $Z_I$-semistable objects in $\Ac^p$.  Then by Propostion \ref{lem5}, each $I_t[-1]$ is a $\mu_\omega$-stable torsion-free sheaf on $X$.  By  \cite[Lemma 3.31]{huybrechts2006fourier}, $I[-1]$ must be a $T$-flat family of sheaves on $X$.  This proves the first half of the proposition.

The second half of the proposition follows from the equivalence with (iv) in Proposition \ref{lem5}.
\end{proof}

By Proposition \ref{prop-family2} we have

\begin{thm} \label{thm_moduli2} Let $X$ be a nonsingular projective 4-fold, and
 $$\Ch'=(-r,0,\gamma,\beta, \xi),\quad \Ch=(0,0,\gamma,\beta, \xi)$$ be Chern character vectors, where $r\neq 0$ and $r, \omega^3 D$ are coprime. Then, the moduli stack of $Z_I$-semistable objects with Chern character $\Ch'$ is identified the moduli stack of $\mu_\omega$-stable torsion free sheaves with Chern character $-\Ch'$. In particular, when $r=1$, the moduli stack of $Z_I$-semistable objects with Chern character $\Ch'$ and trivial determinant is identified with the Hilbert scheme $\op{Hilb}(X,\Ch)$.
  \end{thm}  \qed

Fix a polynomial stability $(Z_{I},\Ac^p)$ with $p$ and $\rho$ as in Proposition \ref{lem5}, and a polynomial stability $(Z_P, \Ac^p)$ with $p$ and $\rho$ as in Proposition \ref{lem4}.  The heart $\Ac^p$ is the same in both instances.  Under the conditions
\begin{align}
\Ch(E)&=(-1,0,\gamma, \beta, n) \text{\quad where $\gamma$ is effective}, \notag\\
\det{E} &= \sO_X, \label{eq:chdetpairs}
\end{align}
 a $Z_{I}$-semistable object $E$ in $\Ac^p$ is up to a shift the ideal sheaf of a 2-dimensional subscheme of $X$ by Proposition \ref{lem5}, while a $Z_P$-semistable object $E$ in $\Ac^p$ is  a stable pair by Proposition \ref{lem4}.
An example of stability vectors $\rho_i$ that satisfies \eqref{eq10} and \eqref{eq11} is shown in Figure \ref{fig:DTPTstabvec}. In terms of stability vectors for polynomial stability, we can interpret  the $Z_I |Z_P$ wall-crossing as $-\rho_4$ rotating clockwise from its location in the left figure arrangement past $\rho_0$ and then past $\rho_1$.

\begin{figure} \label{fig1}
  \caption{Stability vectors $\rho_i$ for $(Z_{I},\Ac^p)$- and $(Z_{P},\Ac^p)$- polynomial stability conditions on a 4-fold.} \label{fig:DTPTstabvec}
  \centering
  \vspace*{-0.0in}
  \scalebox{1}{
\begin{tikzpicture}
    \draw[->] (-2,0) -- (2,0);
    \draw[->] (0,-2) -- (0,2);

     \draw[->,thick] (0,0) -- (1.8478,0.7654);
     \draw[->,thick] (0,0) -- (1.4142,1.4142);
     \draw[->,thick] (0,0) -- (0.7654,1.8478);
     \draw[->,thick] (0,0) -- (-1, 1.7321);
     \draw[->,thick] (0,0) -- (-1.7321,1);

     \filldraw (1.8478,0.7654)  node[right] {$\rho_1$};
     \filldraw (1.4142,1.4142)  node[right] {$\rho_0$};
     \filldraw (0.7654,1.8478)  node[right] {$-\rho_4$};
     \filldraw (-1, 1.7321)  node[left] {$-\rho_3$};
     \filldraw (-1.7321,1) node[left] {$-\rho_2$};

     \filldraw (0,-2.5) node[below] {$Z_{I}$};

    \draw[->] (3,0) -- (7,0);
    \draw[->] (5,-2) -- (5,2);

     \draw[->,thick] (5,0) -- (6.8478,0.7654);
     \draw[->,thick] (5,0) -- (6.4142,1.4142);
     \draw[->,thick] (5,0) -- (5.7654,1.8478);
     \draw[->,thick] (5,0) -- (4,1.7321);
     \draw[->,thick] (5,0) -- (3.2679,1);

     \filldraw (6.8478,0.7654)  node[right] {$-\rho_4$};
     \filldraw (6.4142,1.4142)  node[right] {$\rho_1$};
     \filldraw (5.7654,1.8478)  node[right] {$\rho_0$};
     \filldraw (4,1.7321)  node[left] {$-\rho_3$};
     \filldraw (3.2679,1)  node[left] {$-\rho_2$};

     \filldraw (5,-2.5) node[below] {$Z_P$};

\end{tikzpicture}
}
 \vspace*{-0.15in}
\end{figure}
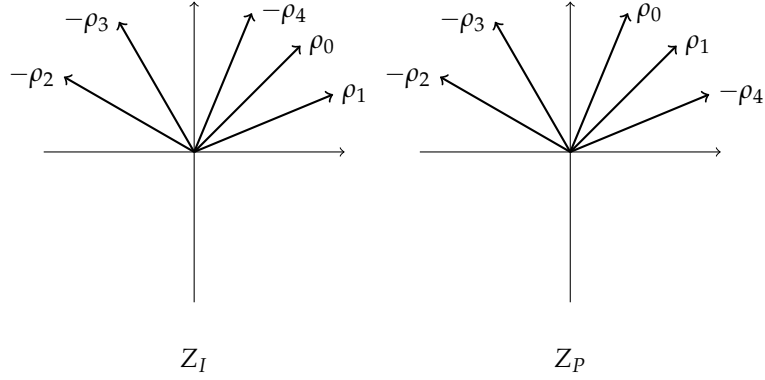

\subsection{Categorical correspondences} \label{sec_cat}

We start by reviewing some homological algebra facts. Let $\mathcal D$ be a triangulated category and  $\Cc_1, \dots, \Cc_n$ be full subcategories of $\mathcal D$ such that  all of $\Cc_1, \dots, \Cc_n$ contain the zero object.  Following the notation in \cite[2.2]{Pol2}, we  write $$[\Cc_1, \dots, \Cc_n]$$ to denote the full subcategory of $\mathcal D$ consisting of objects $E$ such that $E$ admits a filtration by exact triangles of the form
   \[
    \scalebox{0.9}{
    \xymatrix{
    0=E_0 \ar[rr] && E_1 \ar[dl] \ar[rr] & & E_2 \ar[r] \ar[dl] & \cdots \ar[r] & E_{n-1} \ar[rr] & & E_n=E \ar[dl] \\
    & M_1 \ar[ul]^{[1]} & & M_2 \ar[ul]^{[1]} & & & & M_n \ar[ul]^{[1]} &
    }}
    \]
where $M_i \in \Cc_i$ for $1 \leq i \leq n$.  If we also  have $\Hom_{\mathcal D} (\Cc_i, \Cc_j)=0$ for all $1 \leq i < j \leq n$, then we write $$\ldbrack \Cc_1, \dots, \Cc_n \rdbrack$$ instead of $[\Cc_1, \dots, \Cc_n]$. As usual, $$\langle \Cc_1, \dots, \Cc_n \rangle$$ denotes the extension closure of $\Cc_1, \dots, \Cc_n$.

The following Lemma is known to experts. We include a proof for completeness.
\begin{lem}\label{lem:sqbracketassoc}
Given full subcategories $\Cc_1, \Cc_2, \Cc_3$ in a triangulated category $\mathcal D$, we have equalities of subcategories of $\mathcal D$
\begin{equation*}
[[\Cc_1, \Cc_2], \Cc_3] = [\Cc_1, \Cc_2, \Cc_3] = [ \Cc_1, [\Cc_2, \Cc_3]].
\end{equation*}
\end{lem}


\begin{proof}
The first equality easily follows from the definition of the notation $[\cdots]$.

Now take any $E \in [\Cc_1, \Cc_2, \Cc_3]$.  Then there is a sequence of exact triangles
   \[
    \scalebox{0.9}{
    \xymatrix{
    0 \ar[rr] && E_1 \ar[dl] \ar[rr] & & E_2 \ar[rr] \ar[dl] && E_3=E \ar[dl]  \\
    & M_1 \ar[ul]^{[1]} & & M_2 \ar[ul]^{[1]} &&  M_3 \ar[ul]^{[1]} &
    }
    }
    \]
    where $M_i \in \Cc_i$ for each $i$.  Applying the octahedral axiom to the composition of morphisms $E_1 \to E_2 \to E_3$ yields the diagram
    \[
\scalebox{0.8}
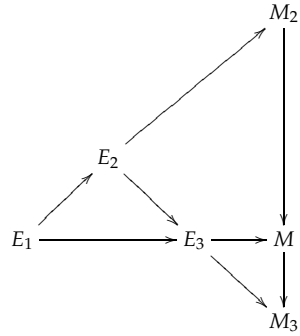
{\xymatrix{
   & & & M_2 \ar[ddd] \\
  & & & \\
  & E_2 \ar[dr] \ar[uurr] & & \\
  E_1 \ar[ur] \ar[rr] & & E_3 \ar[r] \ar[dr] & M  \ar[d] \\
  & & & M_3
}
}
\]
for some object $M$ that fits in an exact triangle $M_2 \to M \to M_3 \to M_2[1]$.  Now we can form the sequence of exact triangles
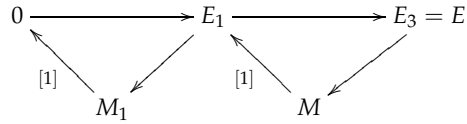
\[
    \scalebox{0.9}{
    \xymatrix{
    0 \ar[rr] && E_1 \ar[dl] \ar[rr] & & E_3 =E \ar[dl]   \\
    & M_1 \ar[ul]^{[1]} & & M \ar[ul]^{[1]} &
    }
    }
    \]
which shows $E \in [\Cc_1, [\Cc_2, \Cc_3]]$.

For the other inclusion, take any $E \in [\Cc_1, [\Cc_2, \Cc_3]]$.  This means that we have a sequence of exact triangles
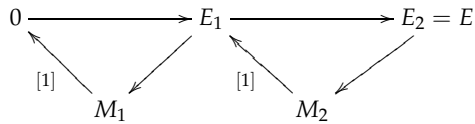
\[
    \scalebox{0.9}{
    \xymatrix{
    0 \ar[rr] && E_1 \ar[dl] \ar[rr] & & E_2=E  \ar[dl]   \\
    & M_1 \ar[ul]^{[1]} & & M_2 \ar[ul]^{[1]} &
    }
    }
\]
where $M_1 \in \Cc_1$ and $M_2 \in [\Cc_2, \Cc_3]$.  Then $M_2$ itself is an extension $N_2 \to M_2 \to N_3 \to N_2[1]$ where $N_2\in \Cc_2$ and $N_3 \in \Cc_3$.  Now  apply the octahedral axiom to the composition $N_2 \to M_2 \to E_1[1]$ to obtain
  \[
\scalebox{0.8}
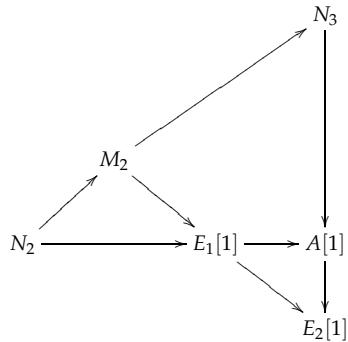
{\xymatrix{
   & & & N_3 \ar[ddd] \\
  & & & \\
  & M_2 \ar[dr] \ar[uurr] & & \\
  N_2 \ar[ur] \ar[rr] & & E_1[1] \ar[r] \ar[dr] & A [1]  \ar[d] \\
  & & & E_2[1]
}
}
\]
for some object $A$.  From the exact triangles $E_1 \to A \to N_2 \to E_1[1]$ and  $A \to E_2 \to N_3 \to A[1]$ we obtain
\[
    \scalebox{0.9}{
    \xymatrix{
    0 \ar[rr] && E_1 \ar[dl] \ar[rr] & & A \ar[rr] \ar[dl] && E_2=E \ar[dl]  \\
    & E_1 \ar[ul]^{[1]} & & N_2 \ar[ul]^{[1]} &&  N_3 \ar[ul]^{[1]} &
    }
    }
\]
which shows $E \in [\Cc_1, \Cc_2, \Cc_3]$.
\end{proof}

From the definition of  $[\cdots]$, it is  easy to see that
\begin{equation}\label{eq:bracket2}
  [\dots[[\Cc_1, \Cc_2], \Cc_3], \dots ], \Cc_{n-1}], \Cc_n ] = [\Cc_1, \dots, \Cc_n]
\end{equation}
for any full subcategories $\Cc_1, \dots, \Cc_n$ of a triangulated category $\mathcal D$.

\begin{cor}\label{coro:P30A3-1}
Let $\mathcal D$ be a triangulated category and $\Cc_1, \dots, \Cc_n$ full subcategories of $\mathcal D$.  Then  any bracketing of $\Cc_1, \dots, \Cc_n$  in this order using $[\cdots]$ is equal to $[\Cc_1, \dots, \Cc_n]$.  The same statement holds for $\ldbrack \cdots \rdbrack$ in place of $[ \cdots ]$.
\end{cor}

\begin{proof}
The first claim follows from  Lemma \ref{lem:sqbracketassoc}  and \eqref{eq:bracket2}.  The second claim follows easily.
\end{proof}

  \begin{lem}\label{lem9}
     Let $\mathcal D$ be a triangulated category, and $\Cc_1, \dots, \Cc_n$ full subcategories of $\mathcal D$ that all contain the zero object.  If $[\Cc_1, \dots, \Cc_n]$ is extension-closed in $\mathcal D$, then $[\Cc_1, \dots, \Cc_n]=\langle \Cc_1, \dots, \Cc_n \rangle$.
     \end{lem}

     \begin{proof}
     We have $\Cc_i  \subseteq [\Cc_1, \dots, \Cc_n] \subseteq \langle \Cc_1, \dots, \Cc_n \rangle$ for all $1 \leq i \leq n$.  Consequently, if  $[\Cc_1, \dots, \Cc_n]$ is extension-closed,  then $\langle \Cc_1, \dots, \Cc_n \rangle \subseteq [\Cc_1, \dots, \Cc_n]$ and hence equality holds.
     \end{proof}

Given an abelian category $\Ac$ and a torsion pair $(\mathcal T, \mathcal F)$ in $\Ac$, we will refer to $\mathcal T$ (resp.\ $\mathcal F$) as a torsion class  (resp.\ torsion-free class) in $\Ac$.  If $\Cc_1, \dots, \Cc_n$ are full subcategories of $\Ac$ such that $\Ac = \ldbrack \Cc_1, \dots, \Cc_n\rdbrack$, then for any $1 \leq i \leq n-1$, 
\[
 ( \ldbrack \Cc_1, \dots, \Cc_i \rdbrack, \ldbrack \Cc_{i+1}, \dots, \Cc_n \rdbrack )
\]
is a torsion pair in $\Ac$.

\begin{rmk}\label{rem:Pol1}
When $\Ac$ is a noetherian abelian category, a torsion class in $\Ac$ is precisely an extension-closed full subcategory of $\Ac$ that is closed under quotient in $\Ac$ \cite[Lemma 1.1.3]{Pol}.
\end{rmk}

Let $V$ be a smooth projective variety of dimension $d \geq 3$ and $\omega$ an ample class on $V$.  Given a fixed positive integer  $r$ and a fixed divisor  $D$ on $V$, let $\mu := \tfrac{\omega^{d-1}D}{r}$.  By truncating $\mu_\omega$-HN filtrations of coherent sheaves at $\mu_\omega = \mu$, we obtain a torsion pair
\begin{align*}
  \op{Coh}_{> \mu}(V) := \langle E \in \op{Coh} (V) : E \text{ is $\mu_\omega$-semistable},  \mu_\omega (E)>\mu \rangle \\
  \op{Coh}_{\leq \mu}(V) := \langle E \in \op{Coh} (V) : E  \text{ is $\mu_\omega$-semistable},  \mu_\omega (E)\leq \mu \rangle
\end{align*}
in $\op{Coh} (V)$, and hence the heart of a bounded t-structure on $D^b(V)$,
\begin{equation} \label{equ_heart}
  \Ac_\mu(V) := \langle \op{Coh}_{\leq \mu}(V)[1],\op{Coh}_{>\mu}(V)\rangle.
\end{equation}
Since the pair $(Z,\op{Coh} (V))$ with $Z(E) = -\omega^{d-1} \cdot \Ch_1 (E) + i\Ch_0(E)$ is a very weak stability condition that is good in the sense of  \cite[Lemma 3.3]{piyaratne2019moduli}, it follows from \cite[Lemma 2.17, Remark 2.10]{piyaratne2019moduli} that $\Ac_\mu(V)$ is a noetherian abelian category.

Next, since the function $$r\omega^{d-1}\cdot \Ch_1(-)-\Ch_0(-)\omega^{d-1}\cdot D$$ is nonnegative on $\Ac_\mu(V)$, the full subcategory
\[
  \Bc_\mu(V):= \{ E \in \Ac_\mu : r\omega^{d-1}\cdot \Ch_1(E)-\Ch_0(E) \omega^{d-1}\cdot D=0\}
\]
is a Serre subcategory of $\Ac_\mu(V)$ and is itself a noetherian abelian category.  We also set
$$\op{Coh}_\mu (V): =   \langle E \in \op{Coh} (V) : E \text{ is $\mu_\omega$-semistable},  \mu_\omega (E)=\mu \rangle.$$

\begin{lem}\label{lem10}
We have the following equality of subcategories of $D^b(V)$:
\begin{itemize}
\item[(a)] $\mathcal B_\mu(V) = \ldbrack \op{Coh}_\mu (V)[1], \op{Coh}^{\leq d-2}(V)  \rdbrack = \langle \op{Coh}_\mu (V)[1], \op{Coh}^{\leq d-2}(V) \rangle$.
\item[(b)] For any torsion class $\mathcal T$ in $\op{Coh} (V)$ with $\mathcal T \subseteq \op{Coh}^{\leq 1}(V)$, 
    \[
    \ldbrack \op{Coh}_\mu (V)[1], \mathcal T \rdbrack=\langle \op{Coh}_\mu (V)[1], \mathcal T\rangle.
    \]
\end{itemize}
\end{lem}

\begin{proof}
(a) Since $\Bc_\mu(V)$ is extension closed, it suffices to prove the first equality by Lemma \ref{lem9}.  Given any  object $E \in \Bc_\mu(V)$, if it does not satisfy both $H^{-1}(E) \in \op{Coh}_\mu (V)$ and $H^0(E)\in\op{Coh}^{\leq d-2}(V)$ the condition $$r\omega^{d-1}\cdot \Ch_1(E)-\Ch_0(E) \omega^{d-1} \cdot D=0$$ is violated. Hence $\Bc_\mu(V)$ is contained in $\ldbrack \op{Coh}_\mu (V)[1], \op{Coh}^{\leq d-2}(V)  \rdbrack$, while the other inclusion is clear.

(b) Since $[ \op{Coh}_\mu (V)[1], \mathcal T]$ is contained in $\Bc_\mu(V)$, which is an abelian subcategory of $D^b(V)$,  it suffices to show that the category $[ \op{Coh}_\mu (V)[1], \mathcal T]$ is  closed under extension in $\Bc_\mu(V)$ by Lemma \ref{lem9}.

Take any short  exact sequence $$0 \to E' \to E \to E'' \to 0$$ in $\Bc_\mu$ where $H^{-1}(E'), H^{-1}(E'') \in \op{Coh}_\mu(V)$ and $H^0(E'), H^0(E'') \in \mathcal T$.  From the long exact sequence of cohomology, $H^0(E)$ must lie in $\mathcal T$ since  $\mathcal T$ is  a torsion class in $\op{Coh} (V)$.  That $H^{-1}(E)$ lies in $\op{Coh}_\mu (V)$ follows from part (a).
\end{proof}

\begin{lem}\label{lem6}
Let $\mathcal T$ be any torsion class in $\op{Coh} (V)$ with $\mathcal T \subseteq \op{Coh}^{\leq 1}(V)$.  Then  $\mathcal T$ and $\langle \op{Coh}_\mu (V)[1], \mathcal T \rangle$   are both torsion classes in $\Bc_\mu(V)$.
\end{lem}

\begin{proof}
By Remark \ref{rem:Pol1}, it suffices to show that each of these two categories is closed under quotient in $\Bc_\mu(V)$.

Consider a $\Bc_\mu(V)$-short exact sequence $$0 \to M \to E \to N\to 0.$$  If $E \in \mathcal T$, then from the long exact sequence $$0 \to H^{-1}(N) \to H^0(M) \to H^0(E)\to H^0(N) \to 0$$ we see that $\Ch_i(H^{-1}(N))=\Ch_i(H^0(M))$ for $i=0,1$, which is impossible if $\Ch_0(H^{-1}(N))>0$.  Hence $H^{-1}(N)$ must vanish, and so $H^0(N)=N$ lies in $\mathcal T$ since $H^0(E)$ lies in $\mathcal T$.

If $E \in \langle \op{Coh}_\mu (V)[1], \mathcal T \rangle$, then we must have $H^0(N)\in \mathcal T$.  That  $H^{-1}(N) \in \op{Coh}_\mu (V)$ follows from Lemma \ref{lem10}(a), and so $N \in \langle \op{Coh}_\mu (V)[1], \mathcal T \rangle$.
\end{proof}

\begin{defn} For any torsion class $\mathcal T$ in $\op{Coh} (V)$ with $\mathcal T \subseteq \op{Coh}^{\leq 1}(V)$, we define $\op{Coh}^{\mathcal T}_\mu(V)$ to be the full subcategory of objects $E \in D^b(V)$  satisfying
\begin{itemize}
\item[(i)] $E \in D^{[-1,0]}_{\op{Coh}(V)}$.
\item[(ii)] $H^{-1}(E)$ is a $\mu_\omega$-semistable torsion-free sheaf of slope $\mu$.
\item[(iii)] $H^0(E) \in \mathcal T$.
\item[(iv)] $\Hom_{D^b(V)}(\mathcal T,E)=0$.
\end{itemize}
\end{defn}

\begin{prop}\label{lem:AG50-17-2}
Let $\mathcal T$ be a torsion class in $\op{Coh} (V)$ with $\mathcal T \subseteq \op{Coh}^{\leq 1}(V)$.  Then 
\begin{equation}\label{eq:AG50-17-2}
  \ldbrack \op{Coh}_\mu (V)[1], \mathcal T \rdbrack = \ldbrack \mathcal T, \op{Coh}_\mu^{\mathcal T}(V)\rdbrack.
\end{equation}
Furthermore, this  category is extension-closed, and so can  be written as
\begin{equation*}
  \langle \op{Coh}_\mu (V)[1], \mathcal T \rangle = \langle \mathcal T, \op{Coh}_\mu^{\mathcal T}(V)\rangle.
\end{equation*}
\end{prop}

\begin{proof}

We have proven in Lemma \ref{lem10} that $$\ldbrack \op{Coh}_\mu(V)[1], \mathcal T \rdbrack =
\langle \op{Coh}_\mu(V )[1],\mathcal T \rangle.$$ Also, this category clearly contains the right-hand side of \eqref{eq:AG50-17-2}.


For the other inclusion in \eqref{eq:AG50-17-2}, take any object $E$ in $\langle \op{Coh}_\mu(V) [1], \mathcal T\rangle$.   We proceed as in the proof of \cite[Lemma 3.17]{Toda2}: Suppose $E$ does not lie in $\op{Coh}^\mathcal T_\mu(V)$.  Then  there must be a nonzero morphism $\alpha : T \to E$ in $D^b(X)$ for some $T \in \mathcal T$.  Since $T, E$ both lie in  $\Bc_\mu(V)$, the morphism $\alpha$ is  in the abelian category $\Bc_\mu(V)$.  Since $\mathcal T$ is closed under quotient in $\Bc_\mu(V)$ by Lemma \ref{lem6}, we can replace $T$ by $\image \alpha$ and assume that $\alpha$ is an injection in $\Bc_\mu(V)$.  Then, since $\langle \op{Coh}_\mu(V) [1], \mathcal T \rangle$ is also closed under quotient in $\Bc_\mu(V)$ by Lemma \ref{lem6}, the quotient $E/\image \alpha$ again lies in $\langle \op{Coh}_\mu(V) [1], \mathcal T\rangle$.  If $E/\image \alpha \notin \op{Coh}^{\mathcal T}_\mu(V)$, there must be a nonzero morphism $T' \to E/\image \alpha$ for some $T' \in \mathcal T$, and we can repeat the process.  Since $\Bc_\mu(V)$ is a noetherian abelian category, this process must terminate, in which case we can write $E$ as an extension of an object in $\op{Coh}^{\mathcal T}_\mu(V)$ by an object in $\mathcal T$, i.e.\ $E$ lies in the right-hand side of \eqref{eq:AG50-17-2}.

The last assertion follows from \eqref{eq:AG50-17-2} and Lemma \ref{lem9}.
\end{proof}

Given a torsion class $\mathcal T$ in $\op{Coh} (V)$, we write $\mathcal T^\circ$ to denote the corresponding torsion-free class in $\op{Coh} (V)$ i.e.  the full subcategory of $\op{Coh} (V)$ given by
\[
  \mathcal T^\circ := \{ E \in \op{Coh} (V) : \Hom_{\op{Coh}(V)} (\mathcal T, E)=0\}.
\]

We now have a more general form of Proposition \ref{lem:AG50-17-2}:

\begin{thm}\label{thm:catwc}
Suppose $\mathcal T_1, \mathcal T_2$ are torsion classes in $\op{Coh} (V)$ such that 
\[
  \mathcal T_1 \subseteq \mathcal T_2 \subseteq \op{Coh}^{\leq 1}(V).
\]
Then
\begin{equation}\label{eq:AG50-17-5}
  \ldbrack \mathcal T_2 \cap \mathcal T_1^\circ, \, \op{Coh}^{\mathcal T_2}_\mu(V) \rdbrack = \ldbrack \op{Coh}^{\mathcal T_1}_\mu(V), \,  \mathcal T_2 \cap \mathcal T_1^\circ \rdbrack.
\end{equation}
Furthermore, this  category is extension-closed, and so can  be written as
\begin{equation*}
  \langle \mathcal T_2 \cap \mathcal T_1^\circ, \, \op{Coh}^{\mathcal T_2}_\mu(V) \rangle = \langle \op{Coh}^{\mathcal T_1}_\mu(V), \,  \mathcal T_2 \cap \mathcal T_1^\circ \rangle.
\end{equation*}
\end{thm}

\begin{proof}
Since $\mathcal T_1, \mathcal T_2$ are both torsion classes in $\op{Coh} (V)$, we have
\[
  \mathcal T_2 = \ldbrack \mathcal T_1, \mathcal T_2 \cap \mathcal T_1^\circ \rdbrack.
\]
By Lemma \ref{lem6} and Proposition  \ref{lem:AG50-17-2}, $\ldbrack \mathcal T_2, \op{Coh}^{\mathcal T_2}_\mu(V) \rdbrack$ is a torsion class in $\Bc_\mu(V)$, and so
\begin{align*}
  \Bc_\mu(V) &= \ldbrack \ldbrack \mathcal T_2, \op{Coh}^{\mathcal T_2}_\mu(V) \rdbrack, \mathcal F \rdbrack \text{ for some torsion-free class $\mathcal F$ in $\Bc_\mu(V)$} \\
  &= \ldbrack \ldbrack \ldbrack \mathcal T_1, \mathcal T_2 \cap \mathcal T_1^\circ \rdbrack, \op{Coh}^{\mathcal T_2}_\mu(V) \rdbrack, \mathcal F \rdbrack \\
  &=  \ldbrack \mathcal T_1, \mathcal T_2 \cap \mathcal T_1^\circ, \op{Coh}^{\mathcal T_2}_\mu(V), \mathcal F \rdbrack \text{ by Corollary \ref{coro:P30A3-1}}.
\end{align*}
Hence $\ldbrack \mathcal T_1, \mathcal T_2 \cap \mathcal T_1^\circ, \op{Coh}^{\mathcal T_2}_\mu(V)\rdbrack$ is a torsion class in $\Bc_\mu(V)$, while $$\ldbrack \mathcal T_2 \cap \mathcal T_1^\circ, \op{Coh}^{\mathcal T_2}_\mu(V), \mathcal F \rdbrack$$ is a torsion-free class in $\Bc_\mu(V)$.  Since their intersection equals $\ldbrack\mathcal T_2 \cap \mathcal T_1^\circ, \op{Coh}_\mu^{\mathcal T_2}(V)\rdbrack$ by the following Lemma, it is extension-closed in $\Bc_\mu(V)$.

\begin{lem}\label{lem:AG50-23-1}
Suppose $\Ac$ is the heart of a t-structure on a triangulated category $\mathcal D$, and $\Ac= \ldbrack \Cc_1, \dots, \Cc_n \rdbrack$ for some full subcategories $\Cc_1, \dots, \Cc_n$.  Then for any $1 \leq i  \leq j \leq n$ we have
    \[
      \ldbrack\Cc_i, \dots, \Cc_j \rdbrack = \ldbrack \Cc_1, \dots, \Cc_j \rdbrack \cap \ldbrack \Cc_i , \dots, \Cc_n \rdbrack.
    \]
\end{lem}

\begin{proof}
The left-to-right inclusion is clear, while the right-to-left inclusion follows from the uniqueness of decomposition (e.g.\ see \cite[Remark 3.7]{Toda2}).
\end{proof}

Next, we show that $\ldbrack \op{Coh}^{\mathcal T_1}_\mu(V), \mathcal T_2 \cap \mathcal T_1^\circ \rdbrack$ is also extension-closed in the abelian category $\Bc_\mu(V)$.  To see this, take any $E', E''$ from this category and consider an extension $E$ of $E''$ by $E'$, i.e.\ we have the $\Bc_\mu(V)$-short exact sequence
\begin{equation}\label{eq:AG50-17-3}
0 \to E' \to E \to E'' \to 0.
\end{equation}
Since $E$ lies in $\Bc_\mu(V)$, we have $H^{-1}(E) \in \op{Coh}_\mu(V)$.  Also, since $$\Hom (\mathcal T_1, E')=0=\Hom (\mathcal T_1, E''),$$ we have $\Hom (\mathcal T_1, E)=0$.  From the long exact sequence of \eqref{eq:AG50-17-3}, we know $H^0(E) \in \mathcal T_2$ so it fits in a short exact sequence of sheaves
\[
0 \to A \to H^0(E) \to H^0(E)/A \to 0
\]
where $A$ is the maximal subsheaf of $H^0(E)$ lying in $\mathcal T_1$, which in turn implies $H^0(E)/A \in \mathcal T_2 \cap \mathcal T_1^\circ$.  Note that both the canonical map $E \to H^0(E)$ and the sheaf surjection $H^0(E) \to H^0(E)/A$ are $\Bc_\mu(V)$-surjections.  Writing $G$ to denote the $\Bc_\mu(V)$-kernel of the composition
\[
  E \to H^0(E) \to H^0(E)/A
\]
and applying the octahedral axiom to this composition, we obtain the $\Bc_\mu(V)$-short exact sequence
\[
0 \to H^{-1}(E)[1] \to G \to A \to 0
\]
from which we see $H^{-1}(G) \cong H^{-1}(E) \in \op{Coh}_\mu(V)$ and $H^0(G) \cong A \in \mathcal T_1$.  On the other hand, applying $\Hom (\mathcal T_1,-)$ to the $\Bc_\mu(V)$-short exact sequence
\[
0 \to G \to E \to H^0(E)/A \to 0
\]
gives $\Hom (\mathcal T_1, G)=0$.  Hence $G \in \op{Coh}^{\mathcal T_1}_\mu(V)$ while $H^0(E)/A \in \mathcal T_2 \cap \mathcal T_1^\circ$.  This shows that $\ldbrack \op{Coh}^{\mathcal T_1}_\mu(V), \mathcal T_2 \cap \mathcal T_1^\circ \rdbrack$ is extension-closed.

We have now shown that both sides of \eqref{eq:AG50-17-5} are extension-closed.  To finish the proof, we will now show
\[
  \op{Coh}^{\mathcal T_2}_\mu (V)\subset \ldbrack \op{Coh}^{\mathcal T_1}_\mu(V), \mathcal T_2 \cap \mathcal T_1^\circ \rdbrack \text{\quad and \quad}  \ldbrack \mathcal T_2 \cap \mathcal T_1^\circ, \op{Coh}^{\mathcal T_2}_\mu(V) \rdbrack \supset  \op{Coh}^{\mathcal T_1}_\mu(V).
\]
Given any $E \in \op{Coh}^{\mathcal T_2}_\mu(V)$, we can follow the same argument as in the previous paragraph to write $E$ as an extension in $\Bc_\mu(V)$
\[
0 \to G' \to E \to A' \to 0
\]
where $H^{-1}(G') \in \op{Coh}_\mu(V)$, $H^0(G') \in \mathcal T_1$ and $A' \in \mathcal T_2 \cap \mathcal T_1^\circ$.  Applying $\Hom (\mathcal T_1, -)$ to this short exact sequence (and noting $\mathcal T_1 \subseteq \mathcal T_2$) gives $\Hom (\mathcal T_1, G')=0$, i.e.\ $G' \in \op{Coh}^{\mathcal T_1}_\mu(V)$.  This proves the first inclusion.

To show the other inclusion, take any $E \in \op{Coh}^{\mathcal T_1}_\mu(V)$.  Since $\mathcal T_2$ is a torsion class in $\Bc_\mu(V)$ by Lemma \ref{lem6}, there is a $\Bc_\mu(V)$-short exact sequence
\begin{equation}\label{eq:AG50-17-6}
0 \to L \to E \to E' \to 0
\end{equation}
where $L$ is the maximal subobject of $E$ lying in $\mathcal T_2$.  Applying the functor $\Hom (\mathcal T_1,-)$ shows $0=\Hom_{D^b(V)}(\mathcal T_1, L) = \Hom_{\op{Coh} (V)} (\mathcal T_1, L)$, i.e.\ $L \in \mathcal T_2\cap\mathcal T_1^\circ$.  We claim that $E' \in \op{Coh}^{\mathcal T_2}_\mu(V)$.  Since $E' \in \Bc_\mu(V)$, we have $H^{-1}(E') \in \op{Coh}_\mu(V) $.  From the long exact sequence of sheaves for \eqref{eq:AG50-17-6}, we see that $H^0(E')$ is a quotient sheaf of $H^0(E)$ which lies in $\mathcal T_1$ and hence $\mathcal T_2$.  Hence $H^0(E')$ itself lies in $\mathcal T_2$.  Finally, we must have $$0 = \Hom_{D^b(V)}(\mathcal T_2, E') = \Hom_{\Bc_\mu(V)}(\mathcal T_2, E'),$$ or else the maximality of $L$ would be violated.  Hence $E' \in \op{Coh}^{\mathcal T_2}_\mu(V)$.
\end{proof}

We have the slope function
\[
\overline{\mu}_\omega (F) := \frac{\Ch_d(F)}{\omega \cdot \Ch_{d-1}(F)}
\]
on $\op{Coh}^{\leq 1}(V)$.  For any interval $I$, we  set
\[
  \Cc_I(V):= \langle F \in \op{Coh}^{\leq 1}(V) : F \text{ is $\overline{\mu}_\omega$-semistable}, \, \overline{\mu}_\omega (F) \in I \rangle.
\]

\begin{cor} \label{cor:AG50-17-7}
Let $b$ be any real number, and $\mathcal U:= \Cc_{(-\infty, b)}(V)$.
Then 
\begin{equation*}
\ldbrack \mathcal U, \op{Coh}^{\op{Coh}^{\leq 1}(V)}_\mu(V) \rdbrack = \ldbrack \op{Coh}^{\Cc_{[b,\infty]}(V)}_\mu(V), \mathcal U\rdbrack.
\end{equation*}
\end{cor}

\begin{proof}
This follows from Theorem \ref{thm:catwc} by choosing $\mathcal T_1 = \Cc_{[b,\infty]}(V)$  and $\mathcal T_2 = \op{Coh}^{\leq 1}(V)$.
\end{proof}

Suppose $(Z_P, \Ac^p)$ is the polynomial stability condition in Proposition \ref{lem4} and $E \in \Ac^p$ is an object of nonzero rank such that $\Ch_0 (E)$, $\omega^3\Ch_1(E)$ are coprime; let $$\mu:=\frac{\omega^3\cdot \Ch_1(E)}{\Ch_0(E)}.$$  Then, by Proposition \ref{prop1},  $E$ is an object of $\op{Coh}^{\op{Coh}^{\leq 1}(X)}_\mu(X)$ if and only if $E$  is a $Z_P$-stable object of slope $\mu_\omega = \mu$.  

For a fixed triple $(r,D,\gamma)$ such that $\frac{\omega^3D}{r}=\mu$ and a fixed line bundle $L$ on $X$ such that $c_1(L)=D$, define
\begin{align*}
  \op{Coh}_\mu(X; r,L,\gamma) &:= \{E \in \op{Coh}_\mu (X) : (\Ch_0(E),\Ch_2(E))=(r,-\gamma), \det{E}=L\}, \\
  \op{Coh}^P_\mu (X;r,L,\gamma) &:= \{ E\in \op{Coh}^{\op{Coh}^{\leq 1}(X)}_\mu(X) : (\Ch_0(E),\Ch_2(E))=(-r,\gamma), \det{E}=L^*\}.
\end{align*}
Therefore, applying Proposition  \ref{lem:AG50-17-2} with $\mathcal T =\op{Coh}^{\leq 1}(X)$, we get

\begin{cor}\label{lem7}
Suppose that $r, \omega^3D$ are coprime. We have the  equalities of subcategories of $\Bc_\mu(X)$
\begin{equation}\label{eq17}
\ldbrack \op{Coh}_\mu(X;r,L,\gamma)[1], \op{Coh}^{\leq 1}(X) \rdbrack= \ldbrack\op{Coh}^{\leq 1}(X), \op{Coh}^P_\mu(X; r,L,\gamma) \rdbrack.
\end{equation}
\end{cor} \qed

\bigskip

Categories of the form $\op{Coh}^{\mathcal T}_\mu(V)$  appear to be useful in several other situations. Here are some instances:

\begin{enumerate}

\item Suppose $\dim V=3$ and $(Z, \Ac^p)$ is the polynomial stability condition in \cite[1.3]{Lo2}, where $\Ac^p = \langle \op{Coh}^{\geq 2}(V)[1], \op{Coh}^{\leq 1}(V)\rangle$.  Following Bayer  \cite{BayerPBSC}, a $Z$-stable object in $\Ac^p$ is referred to as a PT-stable object in \cite{Lo2}.  Suppose  $E \in \Ac^p$ is an object of nonzero rank such that $\Ch_0 (E)$, $\omega^2\cdot \Ch_1(E)$ are coprime; let $\mu: = \omega^2\cdot \Ch_1(E)/\Ch_0(E)$  and   $\mathcal T = \op{Coh}^{= 0}(V)$.  Then by \cite[Proposition 2.24]{Lo2}, $E$ is an object of  $\op{Coh}^{\mathcal T}_\mu(V)$ if and only if $E$ is a PT-stable object in $\Ac^p$.  In particular, when $\Ch_0(E)=-1$ and $\det(E)=\sO_V$, the object $E$ lies in $\op{Coh}^{\mathcal T}_0(V)$ if and only if $E$ is isomorphic to a stable pair in the sense of Pandharipande-Thomas \cite{PT}.
\item Let $V$ be a smooth projective threefold, and let $\mathcal T = \op{Coh}^{= 0}(V)$.  Then, Proposition  \ref{lem:AG50-17-2} recovers Toda's identity \cite[Lemma 3.17]{Toda2}
\[
  \langle \op{Coh}_\mu (V), \Cc_\infty(V) \rangle = \langle \Cc_\infty(V), \op{Coh}^P_\mu (V)\rangle
\]
where $\op{Coh}^P_\mu(V)$ corresponds to $\op{Coh}^{\op{Coh}^{= 0}(V)}_\mu(V)$ in our notation.
\item Let $V$ be a smooth projective threefold.  If we choose $\mathcal T_1 = \Cc_\infty(V)$ and $\mathcal T_2 = \Cc_{[0,\infty]}(V)$ in Theorem \ref{thm:catwc}, we recover the relation
    \[
      \langle \op{Coh}^P_\mu(V), \Cc_{[0,\infty)(V)}\rangle = \langle \Cc_{[0,\infty)}(V), \op{Coh}^L_\mu (V)\rangle
    \]
    in the proof of Toda's \cite[Theorem 4.8]{Toda2}.  Here, Toda's notation  $\op{Coh}^L_\mu (V)$ corresponds to $\op{Coh}^{\Cc_{[0,\infty]}(V)}_\mu(V)$ in our notation, and the objects in this category define his $L$-invariants in \cite[4.1]{Toda2}.  
\item Let $\pi : V \to S$ be a Weierstra{\ss} elliptic Calabi-Yau threefold.  Let $\mathcal T$ be the extension closure of semistable fiber sheaves on $V$ with slope in  $(-1,\infty]$.  Then  an object $E \in D^b(V)$ with $\Ch_0(E)=1$ and $\Ch_1(E)=0$ lies in $\op{Coh}^{\mathcal T}_0(V)[-1]$ if and only if it is a $\pi$-stable pair in the sense of Oberdieck-Shen \cite{OS1}.
    
    If we choose  $\mathcal T_1 = \op{Coh}^{= 0}(V)$ and $\mathcal T_2=\mathcal T$, and restrict both sides of \eqref{eq:AG50-17-5} to subcategories of objects with $\Ch_0 =-1$, $\Ch_1=0$ followed by a shift in the derived category, we recover Oberdieck-Shen's relation \cite[Lemma 5]{OS1}
    \[
      \langle \op{Coh}_P(V), \Cc_{(-1,\infty)}(V)[-1]  \rangle = \langle \Cc_{(-1,\infty)}(V)[-1],  \op{Coh}_\pi (V) \rangle,
    \]
    where by $\Cc_{(-1,\infty)}(V)$ they mean the extension closure of all semistable fiber sheaves with slope in $(-1,\infty)$, and by $\op{Coh}_P(V)$ and $\op{Coh}_\pi (V)$ they mean, respectively, the subcategories of $\op{Coh}^{\mathcal T_1}_0(V)[-1]$ and $\op{Coh}^{\mathcal T_2}_0(V)[-1]$ with $\Ch_0=1, \Ch_1=0$, or equivalently, the subcategories of Pandharipande-Thomas stable pairs and $\pi$-stable pairs.
\end{enumerate}

\subsection{ Hall algebra correspondences} \label{sec_Hall}
In Proposition  \ref{lem:AG50-17-2}, if we set $\mathcal T=\Cc_{[0,\infty]}(X)$, and restrict to objects of rank $-1$ with trivial determinant and fixed $\Ch_2=\gamma$, we obtain 
\begin{equation} \label{1strel}
  \ldbrack \op{Coh}_0 (X;1,\sO_X,-\gamma)[1], \Cc_{[0,\infty]}(X)\rdbrack = \ldbrack \Cc_{[0,\infty]}(X), \op{Coh}_0^{\Cc_{[0,\infty]}(X)}(X;-1,\sO_X,\gamma)\rdbrack.
\end{equation} 

Similarly in Corollary \ref{cor:AG50-17-7}, if we set $b=0$ and make the same restriction, we obtain 

\begin{equation}\label{2ndrel}
\ldbrack \Cc_{(-\infty, 0)}(X), \op{Coh}^{P}_0(X;-1,\sO_X,\gamma) \rdbrack = \ldbrack \op{Coh}^{\Cc_{[0,\infty]}(X)}_0(X;-1,\sO_X,\gamma), \Cc_{(-\infty, 0)}(X)\rdbrack.
\end{equation}

\begin{lem} \label{bda1a2}
For a fixed $\beta \in H^6(X,\bb Q)$, the following subsets of $\bb Q$ \begin{align*} &A^{\ge 0}_{\Cc}(\beta):=\{\Ch_4(Q) \colon Q \in \Cc_{[0,\infty]}(X) \;\&\; \Ch_3(Q)=\beta\},\\ &A_I(\beta):=\{\Ch_4(E) \colon E \in \op{Coh}_0 (X;1,\sO_X,-\gamma)[1] \;\&\; \Ch_3(E)=\beta\},\end{align*}
are bounded  below. 
\end{lem} 
\begin{proof}
The claim for $A^{\ge 0}_{\Cc}(\beta)$ follows from the definition of $\Cc_{[0,\infty]}(X)$. Any object of $\op{Coh}_0 (X;1,\sO_X,-\gamma)$ is an ideal sheaf and hence stable. Thus, the claim for $A_I(\beta)$ follows from \cite[Proposition 3.6]{GK}. 
\end{proof}

Consider the following conditions on the class $\gamma$
\begin{align} \label{cond1} \big\{ \omega \cdot \beta \mid & \; \beta=\Ch_3(\sO_S) \text{ for some 2-dimensional subscheme } S \subset X,\\  \notag &\;\text{ such that } \sO_S \text{ is pure and } \Ch_2(\sO_S)=\gamma \big\} \subset \bb Q \text{  is bounded below.} \end{align}
\begin{align} \label{cond2} &\text{For any 2-dimensional subscheme $S\subset X$ in class $\gamma$}\\ \notag &\text{purity implies Cohen-Macaulayness.} \end{align}

Condition \eqref{cond2} can be equivalently rephrased as for any 2-dimensional subscheme $S\subset X$ in class $\gamma$, $\sO_S$ satisfies Serre's condition $S_{1,2}$ implies that $\sO_S$ satisfies $S_{2,2}$.
If $\gamma$ is a reduced class\footnote{This means that in every decomposition $\gamma=\sum \gamma_i$ into effective classes, none of $\gamma_i$ is an integer multiple of an effective class.} then condition \eqref{cond1} is satisfied; In fact more strongly, it is known that in this case there are finitely many possible choices for $(\omega \cdot \Ch_3(\sO_S), \Ch_4(\sO_S))\in \bb Q^2$ (This fact is attributed to Chow in \cite[Lemma 2.4]{Gr}. See the link in the footnote \footnote{https://mathoverflow.net/questions/264423/counting-hilbert-polynomials-of-projective-varieties} for a nice argument by Jason Starr.). 

\begin{lem} \label{lem_b1b2}
 Suppose $\gamma$ satisfies \eqref{cond1}. The following subsets of $\bb Q$ 
 $$B_I:=\{\omega\cdot \Ch_3(E) \colon E \in \op{Coh}_0 (X;1,\sO_X,-\gamma)[1]\},$$ 
$$B_P:=\{\omega\cdot \Ch_3(E) \colon E \in \op{Coh}_0^P (X;-1,\sO_X,\gamma)\}$$ are bounded below.
\end{lem}
\begin{proof}

Any object of $\op{Coh}_0 (X;1,\sO_X,-\gamma)$ is the ideal sheaf $I_Z$ of a 2-dimensional subscheme $Z\subset X$ in class $\gamma$. There exists a short exact sequence $$0\to I_Z\to I_S\to Q\to 0,$$ such that $Q\in \op{Coh}^{\le 1}(X)$ and $I_S\in \op{Coh}_0 (X;1,\sO_X,-\gamma)$ is the ideal of a subscheme $S\subset X$ such that $\sO_S$ is a pure 2-dimensional sheaf. This shows that $\Ch_3(I_Z)=\Ch_3(I_S)-\Ch_3(Q)$, and since $\Ch_3(Q)$ is effective, the claim for $B_I$ follows from \eqref{cond1}.
  
Any object of $\op{Coh}_0^P (X;-1,\sO_X,\gamma)$ is a stable pair $[\sO_X\to F]$, so there is a short exact sequence $$0\to \sO_S\to F\to Q\to 0,$$ such that $Q\in \op{Coh}^{\le 1}(X)$ and $\sO_S$ is a pure 2-dimensional sheaf. A similar reasoning as above proves the claim for $B_P$.
\end{proof}

\begin{cor} \label{cor_a3} Suppose $\gamma$ satisfies \eqref{cond1}. Then for a fixed $\beta \in H^6(X,\bb Q)$, the following subset of $\bb Q$
$$A_{\mathcal T}(\beta):=\{\Ch_4(E) \colon E \in \op{Coh}^{\Cc_{[0,\infty]}(X)}_0(X;-1,\sO_X,\gamma)\; \&\; \Ch_3(E)=\beta\}$$ is bounded below.
\end{cor}
\begin{proof}
Suppose  $E \in  \op{Coh}^{\Cc_{[0,\infty]}(X)}_0(X;-1,\sO_X,\gamma)$ with  $\Ch_3(E)=\beta$. Then $E$  belongs to LHS of \eqref{1strel}. Since $\Ch_3(Q)$ is effective for any $Q \in \Cc_{[0,\infty]}(X)$, by Lemma \ref{lem_b1b2} there are only finitely many ways of decomposing $\beta=\beta_1+\beta_2$, where $\beta_1$  and $\beta_2$ are the $\Ch_3$ of some objects in $\op{Coh}_0 (X;1,\sO_X,-\gamma)[1]$ and $\Cc_{[0,\infty]}(X)$, respectively. The claim now follows from Lemma \ref{bda1a2}.
\end{proof}

\begin{lem} \label{bdab}
Suppose $\gamma$ satisfies \eqref{cond2}. Then for a fixed $\beta \in H^6(X,\bb Q)$, the following subset of $\bb Q$ 
$$A_P(\beta):=\{\Ch_4(E) \colon E \in \op{Coh}_0^P (X;-1,\sO_X,\gamma)\; \&\; \Ch_3(E)=\beta\}$$
is bounded above.
\end{lem}
\begin{proof}
This Lemma is a special case of Proposition \ref{prop_bdch4} by noting that Condition \eqref{cond2} implies $\mathcal E xt^4(H^0(E),\sO_X)=0$. Here, we give a simpler alternative argument that only works in rank 1. We equivalently prove that for $\xi \gg 0$ and fixed $\gamma$ and $\beta$ the moduli space $P(X,\Ch)$ is empty, where $$\Ch=(0,0,\gamma,\beta, \xi).$$  
If not, there exists an infinite sequence  of stable pairs $(F_i,s_i)$ with fixed $\Ch_2(F_i)=\gamma$, $\Ch_3(F_i)=\beta$ and strictly increasing $\xi_i:=\Ch_4(F_i)$. The scheme theoretic support of $F_i$ is a Cohen-Macaulay surface $S_i$ by assumption. Choose an $i\gg 0$ and let $Q_i:=\op{coker}(s_i)$. It is a pure 1-dimensional sheaf by our assumption and the discussion in \S \ref{sec_description}. Let $C_i$ be the reduced support of $Q_i$, then $Q_i|_{C_i}$ is locally free away from finitely many singularities. Pick a point $p \in C_i$ away from the singularities and a surjection $Q_i|_{C_i}\twoheadrightarrow \sO_p$. We can form a commutative diagram of exact sequences
$$
 \xymatrix{\sO_{S_i} \ar@{=}[d] \ar@{^(->}[r] & F'_i \ar@{^(->}[d] \ar@{->>}[r] & Q'_i \ar@{^(->}[d]\\
   \sO_{S_i}  \ar@{^(->}[r]  & F_i \ar@{->>}[r] \ar@{->>}[d] & Q_i  \ar@{->>}[d] \\
  & \sO_p \ar@{=}[r] & \sO_p.}
 $$
The top row gives a stable pair with $Q'_i$ pure 1-dimensional and  $\Ch_2(F'_i)=\gamma$, $\Ch_3(F'_i)=\beta$ but $\Ch_3(F'_i)=\xi_i-1$. We can carry out the same process with this stable pair by picking another point away from the singularities. After repeating $[\xi_i]$-times, we end up with a stable pair in class $$\Ch':=(0,0,\gamma,\beta,\xi_i-[\xi_i]).$$ We could move the points on $C_i$ and this would imply that the dimension of the tangent spaces of $P(X,\Ch')$ at some points would be bigger than any positive integer, which is impossible.

\end{proof}

\begin{cor} \label{cor_a3bd}
Suppose $\gamma$ satisfies \eqref{cond1} and \eqref{cond2}. Then the set $A_{\mathcal T}(\beta)$ in Corollary \ref{cor_a3} is bounded. 
\end{cor}
\begin{proof}
By Corollary \ref{cor_a3} we need to show $A_{\mathcal T}(\beta)$  is bounded above. Suppose  $$E \in  \op{Coh}^{\Cc_{[0,\infty]}(X)}_0(X;-1,\sO_X,\gamma)$$ with  $\Ch_3(E)=\beta$. Then $E$  belongs to LHS of \eqref{2ndrel}. Since $\Ch_3(Q)$ is effective for any $Q \in \Cc_{(-\infty,0)}(X)$, by Lemma \ref{lem_b1b2} there are only finitely many ways of decomposing $\beta=\beta_1+\beta_2$, where $\beta_1$  and $\beta_2$ are the $\Ch_3$ of some objects in $\op{Coh}_0^P (X;1,\sO_X,-\gamma)$ and $\Cc_{(-\infty,0)}(X)$, respectively. The claim now follows from Lemma \ref{bdab} and Lemma \ref{bda1a2} by the fact that the derived dual of any $Q\in \Cc_{(-\infty,0)}(X)$ lies in  $\Cc_{(0,\infty)}(X)[-3]$.
\end{proof}

Let $\mathcal{O}bj(\Ac_{0})$ be the stack of the objects in $\Ac_{0}(X)$ (introduced in \eqref{equ_heart}), and $\mathcal{O}bj_{\Ch}(\Ac_{0})\subset \mathcal{O}bj(\Ac_{0})$ be the substack of the objects with Chern character $\Ch$. Denote by $$H(\Ac_{0}):=K(\text{St}/\mathcal{O} bj(\Ac_{0}))$$ the motivic
Hall algebra of  $\Ac_{0}$ \cite{Bridgeland10, Toda2}. Let $H_{\Ch}(\Ac_{0})$  be the span of $$[\mathcal X\to \mathcal{O}bj_{\Ch}(\Ac_{0})\subset \mathcal{O}bj(\Ac_{0})].$$ 

Suppose from now on that  $\gamma$ satisfies \eqref{cond1} and \eqref{cond2}.
Let \begin{align*}C&:=\min(0,\op{inf} B_I, \op{inf} B_P). \\L(\beta)&:=\min(0, \op{inf} A_I(\beta), \op{inf} A_{\mathcal T}(\beta)),\\  U(\beta)&:=\max(0,\op{sup} A_P(\beta), \op{sup} A_{\mathcal T}(\beta)).\end{align*}
By Lemmas \ref{bda1a2}, \ref{lem_b1b2}, \ref{bdab} we know $C, L(\beta), U(\beta) \in \bb Q$. 

For $\Ch=(-1,0,\gamma,\beta,\xi)$ or $\Ch=(0,0,0,\beta,\xi)$ define  $$\widehat{H}(\Ac_{0}):=\prod_{\tiny \begin{array}{c}\omega \cdot \beta \ge C\\ \xi \ge L(\beta)\end{array}} H_{\Ch}(\Ac_{0}),\qquad \widetilde{H}(\Ac_{0}):=\prod_{\tiny \begin{array}{c}\omega \cdot \beta \ge C\\ \xi \le U(\beta)\end{array}} H_{\Ch}(\Ac_{0}).$$
We want to define the following elements of the completed Hall algebras 
\begin{align} \label{member}
&\delta_{\Cc}^{\ge 0} \in \widehat{H}(\Ac_{0}), \quad  \delta_{\Cc}^{<0} \in \widetilde{H}(\Ac_{0}), \quad \delta_{I}(\gamma) \in \widehat{H}(\Ac_{0})\\ \notag
& \delta_{P}(\gamma) \in \widetilde{H}(\Ac_{0}) ,\quad \widetilde{H}(\Ac_{0}) \ni \delta_{\mathcal T}(\gamma) \in \widehat{H}(\Ac_{0}),
\end{align}
which respectively correspond to the moduli stacks of objects in the categories 
\begin{align} \label{cats} &\Cc_{[0,\infty]}(X),\quad \Cc_{(-\infty,0)}(X), \quad \op{Coh}_0(X;-1,\sO_X,\gamma)[1],\\ \notag &\op{Coh}_0^{P}(X;-1,\sO_X,\gamma),\quad \op{Coh}_0^{\Cc_{[0,\infty]}(X)}(X;-1,\sO_X,\gamma).\end{align} For these to make sense we need to know that the $\bb C$-valued points of those moduli stacks, once the Chern character vector $\Ch$ is fixed, form constructible subsets of $\mathcal{O}bj_{\Ch}(\Ac_{0})$. For each of  the categories in the first line of \eqref{cats}, this follows from the fact that the set of semistable sheaves with a fixed Chern character is bounded; for the first category in the second line of \eqref{cats} this follows from Theorem \ref{prop:bdd-2} (or alternatively from the fact that $P(X,\Ch)$ is of finite type over $\bb C$); for the last category in \eqref{cats}, this follows from \eqref{1strel} in combination with Condition \eqref{cond1} and Lemma \ref{lem_b1b2}. 
Finally, the memberships claimed in \eqref{member} follow from Lemmas \ref{bda1a2}, \ref{lem_b1b2}, \ref{bdab} and Corollary \ref{cor_a3bd}.

The categorical relations \eqref{1strel}, \eqref{2ndrel} imply 
\begin{thm} \label{thm_hall} Suppose $\gamma\in H^4(X,\bb Z)$ is an effective class satisfying \eqref{cond1}, \eqref{cond2}. We have the following identities 
\begin{align*}
&\delta_{I}(\gamma)\star \delta^{\ge 0}_{\Cc}= \delta^{\ge 0}_{\Cc} \star \delta_{\mathcal T}(\gamma) \quad \text{ in } \widehat{H}(\Ac_{0}),\\
&\delta_{\Cc}^{< 0} \star \delta_{P}(\gamma) =  \delta_{\mathcal T}(\gamma) \star \delta^{<0}_{\Cc} \quad \text{ in } \widetilde{H}(\Ac_{0}).
\end{align*}
\end{thm} \qed

\section{Stable pair invariants}\label{sec_stable_pair_inv} Let $X$ be a nonsingular projective Calabi-Yau 4-fold.
For the Chern character vector $\Ch$  given as in  \eqref{chern_char} we showed that the moduli space of stable pairs $P:=P(X,\Ch)$ is identified with a moduli space of polynomial stable objects $J:=[\sO_X\to F]$ in the derived category (Theorem \ref{thm_moduli}). In this Section, we use this fact to equip this moduli space with a virtual cycle.    

\subsection{Oh-Thomas construction}\label{subsec_virtual_cycle_Oh-Thomas}
We briefly review the construction of virtual fundamental class of $P$ in the sense of \cite{OT}. 
Denote by 
$$\pi_P: X\times P\to P; \quad \pi_X: X\times P\to X$$
to be the natural projections.  Let 
$$\j:=[\sO_{X\times P}\to\ff]$$ be the universal stable pair  on $X\times P$. For a closed point $t\in P$ let $J:=\bb J_t$. \begin{lem} \label{lem_simple} Any stable pair $J$ as above is simple i.e. $\Hom(J,J)\cong \bb C\cong \Ext^4(J,J)$.
\end{lem}
\begin{proof}
The second isomorphism follows from the first one by Serre duality. The first isomorphism can either be proven by showing that $\mathcal{H} om(J,J)\cong \sO_X$ by the same argument as in \cite[Lemma 1.15]{PT}. Alternatively, by Proposition \ref{prop1}, we know that $J$ is a $Z_P$-stable object of $\sA^p$, and so the same argument as in \cite[Corollary 1.2.8]{HL} works in which the reduced Hilbert polynomial is replaced by  $\phi_{Z_P}$. 
\end{proof}

Since $J$ is an object of the heart $\Ac^p$ in Proposition \ref{lem4}, we have $\Ext^{i<0}(J,J)=0$. By Lemma \ref{lem_simple} $\Hom(J,J)_0= 0= \Ext^4(J,J)_0$, so by basechange the perfect complex $R\pi_{P*}R\rH om(\j, \j )_0$ has amplitude in $[1,3]$. Here, the subscript $0$ indicates the trace-free part and using the fact that $\rk \j\neq 0$ we have the decomposition $$R\pi_{P*}R\rH om(\j, \j ) \cong R\pi_{P*}R\rH om(\j, \j )_0\oplus R\pi_{P*}\sO_X$$ splitting the trace map $R\pi_{P*}R\rH om(\j, \j )\to R\pi_{P*}\sO_X$.

Let $\mathbb L_P\in D^b(P)$ denote the truncated cotangent complex of $P$.
Then \cite[Theorem 4.1]{HT} defines an obstruction theory
\begin{equation}\label{eqn_obstruction_theory}
\bb E:=R\pi_{P*}R\rH om(\j, \j )_0[3]\stackrel{\phi}{\longrightarrow} \bb L_P
\end{equation}
(\cite[Definition 4.4]{BF})
by using truncated Atiyah class and taking the trace-free part.  The virtual dimension of this  obstruction theory is given by
\begin{align}\label{eqn_virtual_dimension}
\vd_P:=\rk(\bb E)&=\chi(\sO_X)-\chi(J,J)=2\chi(F)-\chi(F,F)\\ \notag &=2(\xi+\gamma\cdot \operatorname{td}_2(X))-\gamma^2. 
\end{align}

The relative Serre duality  for the morphism $\pi_P$ implies an isomorphism 
\begin{equation}\label{eqn_Serre_duality}
\theta: \bb E\stackrel{\sim}{\longrightarrow} \bb E^{\vee}[2]
\end{equation}
such that $\theta=\theta^{\vee}[2]$. The obstruction theory $\phi$ is not perfect, and so \cite{BF} is not directly applied to equip $P$ with a well-behaved virtual cycle. The idea of \cite{OT} is to find an isotropic cone in an oriented orthogonal bundle such that an analog of square-root intersection \cite{EG} gives a virtual fundamental class of $P$.  In Theorem \ref{thm_moduli} since we fix the determinant, one indeed needs to follow an enhancement of this construction as sketched in \cite[\S 4.3]{OT} that allows taking the trace-free part of the obstruction theory. 

It is possible to take a 3-term locally free resolution $E^\bu$ of $\bb E$ of the form
\begin{equation}\label{eqn_obstruction_complex}
 E^{\bullet}:=[T\stackrel{a}{\longrightarrow}E\stackrel{a^*}{\longrightarrow}T^*].
\end{equation}
Thus $E$ is self-dual and there is a non-degenerate quadratic form 
$q$ on it. Then \eqref{eqn_Serre_duality} induces an isomorphism 
\begin{equation}\label{eqn_det_theta}
\det \theta: (\det E^{\bullet})^{\otimes 2}\stackrel{\sim}{\longrightarrow}\sO_P.
\end{equation}
and by \cite{CGJ} the complex $E^{\bullet}$ admits an \emph{orientation}, i.e., an isomorphism 
\begin{equation}\label{eqn_orientation}
o: \sO_P\stackrel{\sim}{\longrightarrow} \det E^{\bullet}
\end{equation}
such that $o^2=(\det\theta)^{-1}$.
The orientation $o$ defines an orientation on the quadratic bundle $E$ in  (\ref{eqn_obstruction_complex}), which means a trivialization
$$o: \sO_P\stackrel{\sim}{\longrightarrow} \det E$$
such that $(\det E)^{\otimes 2}\stackrel{\sim}{\rightarrow}\sO_P$. 
Therefore, $E$ acquires a structure of $\SO(2r, \cc)$-bundle.

Consider the truncation 
\begin{equation}\label{eqn_truncation_EBullet}
\tau E^{\bullet}:= [E\to T^*]\to E^{\bullet}\stackrel{\phi}{\longrightarrow}\bb L_P,
\end{equation} 
which clearly defines a perfect obstruction theory.  
Let $h^1/h^0((\tau E^{\bullet})^{\vee})$ be the corresponding stack bundle  and  $\mathfrak{C}_{P}\subset h^1/h^0((\tau E^{\bullet})^{\vee})$ be the intrinsic normal cone of $P$ \cite{BF}. Then $\mathfrak{C}_P$ lifts to a subcone 
$$C:=C_{E^\bullet}\subset E^*\cong E.$$
One of the main results of \cite{OT} is that the subcone $C \subset E$ is isotropic and defines a \emph{square root Gysin map}:
\begin{equation}\label{eqn_square_root_Gysin}
\sqrt{0^!_E}: A_*(C,\zz\Big[\frac{1}{2}\Big])\to  A_{*-\frac12\op{rk}E}(P, \zz\Big[\frac{1}{2}\Big]).
\end{equation}
In the special case where the isotropic cone $C$ factors through a maximal isotropic subbundle $\Lambda\subset (E,q)$ the square root Gysin map is given by $(-1)^{|\Lambda|} 0^!_\Lambda$, where $(-1)^{|\Lambda|}$ is the sign of $\Lambda$ defined as in \cite[Definition 2.2]{OT}. 

If $i\colon P\subset Y$ is a closed subscheme, and $(E,q,o)$ is the restriction of some $SO(2r, \bb C)$-bundle $(E_Y,q_Y,o_Y)$ on $Y$ with an isotropic section $s$ such that $P=Z(s)$ and $C=C_{P/Y}$ then \begin{equation} \label{equ_pushfwd} i_*\circ \sqrt{0^!_E}=\sqrt e (E_Y) \cap -,\end{equation}
where $\sqrt e (E_Y)$ is the Edidin-Graham square-root Euler class \cite{EG}. Following \cite{OT} we have the following Definition.

\begin{defn}\label{defn_stable_vir}
If the virtual dimension $\vd$ is an odd number define  $[P]^{\vir}=0$.
If $\vd_P$ is an even number the virtual fundamental class of $P$ is defined as:
\begin{equation}\label{eqn_virtual_class}
[P]^{\vir}=[P(X,\Ch)]^{\vir}=\sqrt{0^!_E}[C_{E^\bullet}]\in A_{\frac{1}{2}\vd_P}(P, \zz\Big[\frac{1}{2}\Big])
\end{equation} independent of the choice of the resolution $E^\bu$.
\end{defn}

A special case that arises in some of our examples is when $P$ is smooth and the obstruction sheaf $\op{Ob}_P:=h^1(\bb E^\vee)$ is locally free. Then, the isotropic cone $C_{E^\bullet}$ is the image of $T$ in $E$ and we have an analog of \cite[Proposition 5.6]{BF}:
\begin{equation} \label{BF56} [P]^{\vir}=\sqrt{e}(\op{Ob}_P).\end{equation}

Another situation that we will consider later is when we have a smooth morphism $u\colon P\to M$, where $M$ is a scheme with a self dual obstruction theory $\bb E' \to \bb L_M$. Suppose moreover, that $\Omega:=\Omega_{P/M}$ is a subbundle of $T$ and $\bb E'$ admits a 3-term self-dual locally free resolution $E'^\bu$ such that $$u^* E'^{\bullet} \cong  [T/\Omega\stackrel{a}{\longrightarrow}E\stackrel{a'^*}{\longrightarrow}(T/\Omega)^*],$$ where $a=a' \circ p$ for the natural projection  $p\colon T\to T/\Omega$.
The perfect obstruction theories given by $\tau E^\bu$ and $\tau E'^\bu$ are then compatible over $u$ in the sense of \cite[Definition 5.8]{BF}. By the proof of \cite[Proposition 5.10]{BF}
$C_{E^\bullet}=u^* C_{E'^\bullet}$ and hence by \cite[Lemma 3.9]{OT} \begin{equation} \label{BF510} [P]^{\vir}=u^*[M]^{\vir}. \end{equation}

For integral cohomology classes $\alpha_i\in H^{m_i}(X,\zz)$, we let 
\begin{equation}\label{eqn_tau_insertion}
\tau: H^{m_i}(X,\zz)\to H^{m_i-4}(P,\zz)
\end{equation}
defined by 
$$\tau(\alpha_i)=\pi_{P*}(\pi_X^*\alpha_i\cup\Ch_2(\ff)),$$
where $\Ch_2(\ff)$ is the Poincar\'e dual to the fundamental cycle of $\ff$.

\begin{defn}\label{defn_stable_pair_invariants}
The stable pair invariant with insertions is defined as:
$$\lra{\alpha_1,\dots, \alpha_s}_{\Ch}^P:=\int_{[P(X,\Ch)]^{\vir}}\prod_{i=1}^{s}\tau(\alpha_i).$$
\end{defn}

Fix the components $\gamma \in H^{4}(X,\bb Z)$ and $\beta \in H^6(X,\bb Q)$ of $\Ch$ and define a generating series of stable pair invariants by summing over the last component $\xi \in H^8(X,\bb Q)$
$$Z^P_{\gamma,\beta}(\alpha_1,\dots,\alpha_s)(q)=\sum_{\xi \in \frac{1}{12}\bb Z}\lra{\alpha_1,\dots, \alpha_s}_{\Ch}^Pq^{n},\qquad n:=\xi+\gamma\cdot \operatorname{td}_2(X) \in \bb Z.$$  

\begin{prop} \label{prop_gener}
$Z^P_{\gamma,\beta}(\alpha_1,\dots,\alpha_s)(q)$ is a Laurent series in $q$. 
If $\gamma$ satisfies Condition \eqref{cond2}  then $Z^P_{\gamma,\beta}(\alpha_1,\dots,\alpha_s)(q)$ is a Laurent polynomial in $q$.
\end{prop}
\begin{proof}
For $\xi \ll 0$ and fixed $\gamma$ and  $\beta$ the virtual dimension $\vd_P$ becomes negative and hence $[P(X,\Ch)]^{\vir}=0$ proving the first part. The second part follows from Lemma \ref{bdab}.

\end{proof}

\subsection{Quasi-projective case} In \S \ref{sec_local_surfaces} and \S \ref{sec_local_three} we consider situations in which $X$ is quasi-projective. In this case, following the adjustments discussed in \cite[Section 4.3]{OT}, we take the Chern character vector \eqref{chern_char} to be \emph{compactly supported}. In all of these situations, there are open immersions $X\subset \overline{X}$ where $\overline{X}$ is a nonsingular projective 4-fold and $D:=\overline{X}\setminus X$ is a nonsingular divisor. We can then view $P(X,\Ch)$ naturally as an open subset of a moduli space $P(\overline X,\overline \Ch)$ consisting of stable pairs $\sO_X\to F$ for which $\operatorname{Supp} F \cap D=\emptyset$. This will enable us to apply the results of \S \ref{sec_stability_moduli}.

Since Serre duality does not apply in quasi-projective case we will do further analysis  to show that the complex $\bb E$ in \eqref{eqn_obstruction_theory} still satisfies \eqref{eqn_Serre_duality}. By the following Lemma and basechange $\bb E$ has perfect amplitude in $[-2,0]$ in quasi-projective situations as well.
\begin{lem}
Suppose $X$ is quasi-projective and $J=[\sO_X\to F]$ is a stable pair  corresponding to a closed point of $P(X,\Ch)$ then $\Ext^{i\le 0}(J,J)_0=0=\Ext^4(J,J)_0$.
\end{lem}
\begin{proof}
The first equality follows by the same argument as in \cite[Lemma 1.15]{PT} followed by the local to global spectral sequence. As before, we have the decomposition $$\Ext^4(J,J)\cong \Ext^4(J,J)_0 \oplus H^4(\sO_X)$$ splitting the trace map. On the other hand, applying $\Hom(\sO_X,-)$ and  $\Hom(-,J[-1])$  to the natural exact triangle $J[-1]\to \sO_X\to F\to J$, we find that $\Ext^4(\sO_X,J[-1])$ is identified with the second summand $H^4(\sO_X)$ in the decomposition above, and surjects onto $\Ext^4(J,J)$ proving the second equality in the Lemma.
\end{proof}

\subsection{Two exact triangles over the moduli space} \label{sec_triang}
The universal stable pair gives a natural exact triangle over $X\times P$ \begin{equation} \label{equ_natural} \bb J[-1]\to \sO_{X\times P} \to \bb F \to \bb J.\end{equation}  
For simplicity, we will sometimes use the notation $$\lra{-,.}:= R\pi_{P*}\mathcal{H}om(-, \cdot).$$ 
Applying the functor $\lra{-,\bb F}$ to \eqref{equ_natural} gives us the first exact triangle over $P$
\begin{equation} \label{equ_first} 
\lra{\bb F, \bb F}\to \lra{\sO,\bb F}\to \lra{\bb J[-1],\bb F}\to \lra{\bb F, \bb F}[1].
\end{equation} 
On the other hand, by applying the functor $\lra{\bb J,\cdot}$ to \eqref{equ_natural} we can form a commutative diagram
$$
\xymatrix{
& R\pi_{P*}\sO[1]\ar@{=}[r]\ar[d] & R\pi_{P*}\sO[1]\ar[d]\\
\lra{\bb J[-1],\bb F}\ar[r]& \lra{\bb J,\bb J}[1]\ar[r]& \lra{\bb J, \sO_{X\times P}}[2]}
$$ in which the bottom row is part of an exact triangle and the left vertical arrow is the identity map. Taking cones over the vertical arrows, we get the second exact triangle
\begin{equation} \label{equ_sec} 
\lra{\bb J[-1], \bb F}\to \lra{\bb J,\bb J}_0[1]\to \lra{\bb F,\sO}[2]\to \lra{\bb J[-1], \bb F}[1],
\end{equation} where as usual the subscript $0$ indicates the trace-free part.

\subsection{Relation to two types of DT-invariants}\label{DT-4_invariants}
For $\Ch$ as in \eqref{chern_char} let $$I(X,\Ch), \quad M(X,\Ch)$$ respectively denote the moduli spaces of ideal sheaves (the Hilbert scheme) $I<\mathcal O_X$, and semistable 2-dimensional sheaves $F$ with $$1-\Ch(I)=\Ch=\Ch(F).$$ In the case of $M(X,\Ch)$, assume $\Ch$ is such that the semistability implies stability. Let $\bb I$ and $\mathbb F$ be the universal ideal sheaf and the universal (twisted) sheaf over $X\times I(X,\Ch)$ and $X\times M(X,\Ch)$, respectively.
 The natural obstruction theories with the virtual cotangent bundles 
$$ R\pi_{I*}\rH om(\bb I, \bb I)_0[3], \qquad \tau^{[-2,0]}(R\pi_{M*}\rH om(\ff, \ff)[3])$$
and the work of Oh-Thomas provides virtual cycles
$$[I(X,\Ch)]^{\vir}\in A_{\frac{1}{2}\vd_I}(I(X,\Ch), \zz\Big[\frac{1}{2}\Big]), \quad [M(X,\Ch)]^{\vir}\in A_{\frac{1}{2}\vd_M}(M(X,\Ch), \zz\Big[\frac{1}{2}\Big]),$$
where $\vd_I=2(\xi+\gamma\cdot \operatorname{td}_2(X))-\gamma^2$, and $\vd_M=2-\gamma^2$. One can define DT invariants $$\lra{\alpha_1,\dots, \alpha_s}_{\Ch}^I,\qquad \lra{\alpha_1,\dots, \alpha_s}_{\Ch}^M$$ by integrating these virtual cycles similar to stable pair invariants in \S \ref{eqn_obstruction_complex}.

%
%
%
%

One notes that $P(X,\Ch)$ and $I(X,\Ch)$ have the same virtual dimensions. If we define a generating series $Z^I_{\gamma,\beta}(-)$ for the DT invariants $\lra{\alpha_1,\dots, \alpha_s}_{\Ch}^I$ similar to that of stable pair invariants then by \cite[Proposition 3.6]{GK} $Z^I_{\gamma,\beta}(-)$ will be a Laurent series in $q$ (cf. Proposition \ref{prop_gener}). Based on the discussion in \S \ref{subsec_polynomial_ideal_sheaves}, we expect (wall-crossing) relations among the stable pair invariants and DT invariants of ideal sheaves. We plan to pursue this question in the future. In a work in preparation \cite{BKP}, Bae, Kool and Park have formulated some conjectural formulas relating these two theories and have done several toric calculations to verify their conjectures.

In the case $\gamma \in H^4(X,\mathbb Z)$ is an irreducible class\footnote{This means that $\gamma$ cannot be written as a sum of two effective classes.} any pure 2-dimensional sheaf $F$ with $\Ch(F)=\Ch$ is stable, and hence we can define a natural morphism $$\rho\colon P(X,\Ch)\to M(X,\Ch)$$ by mapping a stable pair $[\sO_X{\rightarrow}F]$ to $F$. If $H^{i>0}(F)=0$ for any pure 2-dimension sheaf $F$ with $\Ch(F)=\Ch$ the morphism $\rho$ is the projection map of a projective bundle $\bb P(\pi_{P*}\bb F)$ of fiber dimension $\xi+\gamma\cdot \operatorname{td}_2(X)-1$.
\begin{prop} \label{prop_pullback}
 Suppose $\gamma \in H^4(X,\mathbb Z)$ is an irreducible class and $H^{i>0}(F)=0$ for any $F\in M(X,\Ch)$ then $$[P(X,\Ch)]^{\vir}=\rho^* [M(X,\Ch)]^{\vir}.$$ 
 \end{prop}
\begin{proof} 
We will drop obvious pullback symbols in the following and denote both projections $X\times P\to P$ and $X\times M\to M$ by $\pi$.
We identify the relative tangent bundle $T:=T_{P/M}$ with $\pi_* \bb F/\sO_P$. Let us define $$\bb A:=\tau^{\le 2} \lra{\bb J[-1],\bb F}.$$  By trimming the complexes in the exact triangles \eqref{equ_first} and \eqref{equ_sec} we get the following exact triangles over $P$
\begin{align*} & T\to \bb A \to \tau^{[0,2]} \lra{\bb F,\bb F}[1]\to T[1],\\
&T^*[-3]\to \bb A \to \lra{\bb J,\bb J}_0[1]\to T^*[-2]. \end{align*}  The  terms $\lra{\bb F,\bb F}[1]$ and $\lra{\bb J,\bb J}_0[1]$ are the (pullback of) virtual tangent bundles on $M$ and $P$, respectively. For any  $J=[\sO_X\to F]\in P(X,\Ch)$ we see that $\Ext^4(J[-1],F)=0$ and $\Ext^3(J[-1],F)\cong \bb C$ from the cohomology long exact sequence, so by basechange theorem $\bb A$ has a perfect amplitude in $[0,2]$. Let $$A_0\to A_1 \to A_2$$ be a locally free resolution for $\bb A$. Using Jouanolou trick we can pullback everything to an affine bundle over $P$, whose total space is an affine scheme (see \cite[Section 1.1]{GT2}). So we can pretend that all short exact sequences of vector bundles are split and morphisms in derived category between perfect complexes are given by genuine map of complexes. The exact triangles above then imply that 
\begin{align*} &\tau^{[0,2]} \lra{\bb F,\bb F}[1] \cong [A_0/T\to A_1 \to A_2],\\
&\lra{\bb J,\bb J}_0[1] \cong [A_0\to A_1\to A_2\oplus T^*].\end{align*} Using the fact that the dual of both of these complexes satisfy \eqref{eqn_Serre_duality}, as in the proof of \cite[Proposition 4.1]{OT}, there exists a vector bundle $A$ such that 
\begin{align*} & \lra{\bb F,\bb F}[1] \cong [A_0/T\to A \to (A_0/T)^*],\\
&\lra{\bb J,\bb J}_0[1] \cong [A_0\to A\to A_0^*].\end{align*}
The claim now follows from \eqref{BF510}.
\end{proof}
 
In the next Subsection, we will discuss the case where $H^1(F)$ is not necessarily zero. 


\subsection{Moduli space of stable pairs as a degeneracy locus} \label{sec_degen}
Let $Y$ be a scheme with a perfect obstruction theory $\bb E\to \bb L_Y$ and the resulting virtual cycle $[Y]^{\vir}$. Suppose that $V$ is a vector bundle over $Y$ with a section $s$. Let $Z:=Z(s)$ be the zero locus of $s$ defined by an ideal $\sI \le \sO_Y$. This will give a commutative diagram $$\xymatrix{ V^*|_Z\ar[r]^-{\lambda} \ar@{->>}[d]_-{s|_Z} & \Omega_Y|_Z \ar@{=}[d] \ar@{->>}[r] & \Omega_Z \ar@{=}[d] \\
\sI/\sI^2 \ar[r]^-{d_{Z\subset Y}} & \Omega_Y|_Z \ar@{->>}[r] & \Omega_Z,}$$ which implies in particular that $\ker \lambda$ surjects onto $K:=\ker d_{Z\subset Y}$.

Let $A$ be a nonsingular variety containing $Y$ (and hence $Z$) as a closed subscheme. Then the truncated cotangent complexes of $Y$ and $Z$ are given by $$\tau^{\ge -1} \bb L_Y\cong[I_Y/I^2_Y\xrightarrow{d_Y} \Omega_A|_Y],\qquad \tau^{\ge -1} \bb L_Z\cong [I_Z/I^2_Z\xrightarrow{d_Z} \Omega_A|_Z].$$ Here, $I_Y\subset I_Z \subset \sO_A$ are the ideals of these subschemes and $\sI$ is identified with $I_Z/I_Y$, which therefore gives an exact sequence \begin{equation} \label{IYZ} I_Y/(I_ZI_Y)\xrightarrow{a} I_Z/ I^2_Z\xrightarrow{b} \sI/ \sI^2 \to 0.\end{equation}

Pick a locally free resolution $\bb E\cong[E^{-1}\to E^0] $, such that the perfect obstruction theory above is given by a map of 2-term complexes (with both vertical arrows surjective)  
$$\xymatrix{ E^{-1}\ar[r]^-{\theta} \ar@{->>}[d]_-{t'} & E^0 \ar@{->>}[d]^-{t} \\
I_Y/I_Y^2 \ar[r]^-{d_{Y}} & \Omega_A|_Y}$$ inducing an isomorphism on $h^0$ and surjection on $h^{-1}$. 

From here on in this Subsection, we pullback everything to an affine bundle over $A$ which is an affine variety (using Jouanolou trick as explained in \cite[Section 1.1]{GT2}). This allows us, using the surjections $b$ and $t|_Z$,  to lift $$V^*|_Z\xrightarrow{s|_Z} \sI/\sI^2 \quad \text{to} \quad V^*|_Z\xrightarrow{\sigma} I_Z/I_Z^2,$$ and also to lift $$V^*|_Z\xrightarrow{d_Z\circ \sigma} \Omega_A|_Z\quad \text{to} \quad V^*|_Z\xrightarrow{\tau} E^0|_Z.$$
The map $\tau$ defines a map of complexes $V^*|_Z\to E^\bullet|_Z$, whose mapping cone is the top row of the following commutative diagram \begin{equation}\label{dpot} \xymatrix{ V^*|_Z\oplus E^{-1}|_Z\ar[r]^-{(\tau,\theta)} \ar[d]_-{(\sigma,}^-{a\circ t')} &E^0|_Z \ar[d]^-{t} \\
I_Z/I_Z^2\ar[r]^-{d_Z} & \Omega_A|_Z.} 
\end{equation}
Here, the commutativity follows from the choices of lifts made above, and we have dropped $|_Z$ for $t$, $t'$ and $\theta$ for simplicity.
\begin{prop} \label{prop_potZ}
Diagram \eqref{dpot} gives a perfect obstruction theory for $Z$.
\end{prop}
\begin{proof}
We need to show that the map of complexes making the rows of \eqref{dpot} induces isomorphism on $h^0$ and surjection on $h^{-1}$. Since $\op{coker} \theta\cong \Omega_Y$ and $\op{coker} \lambda \cong \Omega_Z \cong \op{coker} d_Z$, so the $h^0$ requirement is fulfilled.

For the $h^{-1}$ requirement, we need to show that the induced map between the kernels $$T:=\ker (\tau,\theta)\xrightarrow{f}R:=\ker d_Z$$ is surjective.
 Denote also the induced map $R\to K=\ker d_{Z\subset Y}$ by $g$. For any $k\in K$, since $\ker \lambda$ surjects onto $K$, there exists a $v\in \ker \lambda \subset V^*|_Z$ such that $s(v)=k$. We must therefore have $$\tau(v)\in \ker(E_0|_Z\to \Omega_Y|_Z),$$ so there exists an $e\in E^{-1}|_Z$ such that $\theta(e)=\tau(v)$ and so $(v,-e)\in T$. Since $b\circ a \circ t'(e)=0$, we have $f\circ g(v,-e)=k$. 

Let $r\in R$, then by what we said above for $k:=f(r)$, there exists $(v,-e)\in T$ such that $f\circ g(v,-e)=f(r)$. Let $r':=g(v,-e) \in R$, then $$r-r' \in \ker(b) \cap R.$$ Since $I_Y/(I_ZI_Y)$ surjects onto $\ker b$, we see that $\ker(d_Y|_Z)$ surjects onto $R\cap \ker b$ as well. 
Using the induced commutative diagram
$$\xymatrix{(\ker \theta)|_Z \ar@{->>}[r] \ar@{->>}[d] & \ker(\theta|_Z) \ar[d] & \\
(\ker d_Y)|_Z \ar@{->>}[r] & \ker(d_Y|_Z)\ar@{->>}[r]& R\cap \ker b,}$$
we see that $r-r'$ must be in the image of $\ker(\theta|_Z)$ and hence in that of $T$.  Now since both $r'$ and $r-r'$ are in the image of $T$, so must be $r$.
\end{proof}
Let $\bb E_Z$ be the 2-term complex in the top row of \eqref{dpot}. The perfect obstruction theory corresponding to this Diagram gives a virtual cycle denoted by $[Z]^{\vir}$. The following Corollary is a special case of Manolache's functoriality of the virtual pullback (\cite[Theorem 4.8]{man}).
\begin{cor} \label{cor_kkp}
Let $\iota\colon Z=Z(s)\hookrightarrow Y$. Then $\iota_*[Z]^{\vir}=e(V)\cap [Y]^{\vir}$.
\end{cor}
\begin{proof}
By our construction we have a commutative diagram whose rows form the natural exact triangles 
\begin{equation} \label{equ_comp}\xymatrix{\iota^* \bb E \ar[r] \ar[d] & \bb E_Z \ar[r] \ar[d] &  \iota^*V^*[1] \ar[d]\\
\iota^* \bb L_Y \ar[r] & \bb L_Z\ar[r]& \bb L_{Z/Y}.}\end{equation} (We need to choose a lift of $t'\colon E^{-1}\to I_Y/I_Y^2\cong \op{coker}(l)$ to $E^{-1}\to L^{-1}$ and a lift of $\sigma\colon V^* \to I_Z/I_Z^2\cong \op{coker}(l')$ to $V^*|_Z\to L'^{-1}$, where $$\bb L_Y\cong [\dots \to L^{-2}\xrightarrow{l} L^{-1}\to \Omega_A|_Y],\qquad \bb L_Z\cong [\dots \to L'^{-2}\xrightarrow{l'} L'^{-1}\to \Omega_A|_Z]$$ are locally free resolutions. This will give the left commutative square of \eqref{equ_comp}. The right square will be then induced by the properties of exact triangles.) Consider the Cartesian diagram $$\xymatrix{Z\ar[r] \ar[d] &Y \ar[d]^-s\\Y\ar[r]^-{0_V} & V.} 
$$ The perfect obstruction theories given by the first two vertical arrows of \eqref{equ_comp} are compatible over $0_V$ \cite[Section 7]{BF}. By \cite[Theorem 1]{KKP} we have $[Z]^{\vir}=0_V^![Y]^{\vir}$. The claim now follows from \cite[Proposition 14.1]{F}. 
\end{proof}

Next, suppose that $Y$ is a scheme with a self-dual 3-term obstruction theory \begin{equation} \label{eq_sdot} [T\to E\to T^*] \to \bb L_Y\end{equation} as in \eqref{eqn_obstruction_complex}, and the resulting Oh-Thomas' virtual cycle $[Y]^{\vir}\in A_{\vd_Y/2}(Y,\bb Z[\frac 1 2])$. Let $V$, $s\in \Gamma(Y,V)$, $Z:=Z(s)$, and $\sI \le \sO_Y$ be as in the beginning of this Subsection.
The induced map from the truncation $[E\to T^*]$ to  $\bb L_Y$ is a perfect obstruction theory to which we apply the construction above leading to Proposition \ref{prop_potZ}. This gives a perfect obstruction theory for $Z$ \begin{equation} \label{eq_2term} [E|_Z\oplus V^*|_Z\to T^*|_Z]\to \bb L_Z.\end{equation} Using the dual of the map $V^*|_Z\to T^*|_Z$ in \eqref{eq_2term} and the zero map $V|_Z[1]\to \bb L_Z$,  this can in turn be enlarged to a self-dual 3-term obstruction theory \begin{equation} \label{equ_3term} [T|_Z\to V|_Z\oplus E|_Z \oplus V^*|_Z\to T^*|_Z]\to \bb L_Z.\end{equation}

The following Proposition is a special case of the functoriality of the \emph{square root virtual pullback} recently proven by Hyeonjun Park \cite[Theorem 2.2]{park}.
\begin{prop} \label{prop_zero}
If $[Z]^{\vir}\in A_{\vd_Z/2}(Z,\bb Z[\frac 1 2])$ is Oh-Thomas' virtual cycle resulting from \eqref{equ_3term} then $$\iota_*[Z]^{\vir}=e(V)\cap [Y]^{\vir}.$$
\end{prop}
\begin{proof}
 If $\vd_Y$ is odd then so is $\vd_Z$ and hence both sides of the claimed formula are zero. Assume $\vd_Y$ is even. Let $C_0\subset E$ and $C_1 \subset E|_Z \oplus V|_Z$ be the isotropic cones obtained from \eqref{eq_sdot} and \eqref{equ_3term}, respectively, as described in \S \ref{subsec_virtual_cycle_Oh-Thomas}.  We have $$[Y]^{\vir}=\sqrt{0_{E}^!}\; [C_0],\qquad [Z]^{\vir}=\sqrt{0_{V|_Z\oplus E|_Z \oplus V^*|_Z}^!} \;[C_1].$$ 
As in the proof of Corollary \ref{cor_kkp}, the perfect obstruction theories  
$$[E\to T^*]\to  \bb L_Y,\qquad [E|_Z\oplus V^*|_Z\to T^*|_Z]\to \bb L_Z$$ mentioned above are compatible over $0_V$. From this, one can furthermore see that \eqref{eq_sdot} and \eqref{equ_3term} satisfy the compatibility condition of \cite[Definition 2.1]{park}. 
By \cite[Lemma 2.6]{park}  $$0_V^! \sqrt{0_{E}^!}\; [C_0]=\sqrt{0_{V|_Z\oplus E|_Z \oplus V^*|_Z}^!} \;[C_{Z/C_0}].$$ Here, $C_{Z/C_0} \subset  C_{Z/Y}\times_Z C_0|_Z\subset  V|_Z\oplus E_0|_Z$ is therefore an isotropic subcone of $V|_Z\oplus E|_Z \oplus V^*|_Z$. On the other hand, by  \cite[Proposition 2.7]{park} and \cite[Lemma 3.9]{OT} $$\sqrt{0_{V|_Z\oplus E|_Z \oplus V^*|_Z}^!} \;[C_{Z/C_0}]=\sqrt{0_{V|_Z\oplus E|_Z \oplus V^*|_Z}^!} \;[C_1].$$
From these equalities we get $0_V^! [Y]^{\vir}=[Z]^{\vir}$,\footnote{There was a gap in making this conclusion in the first draft of the  paper: the cones $V|_Z\times_Z C_0|_Z$, $C_1$ and $C_{Z/C_0}$ are rationally equivalent in $V|_Z\oplus E|_Z$ (see the proof of \cite[Theorem 1]{KKP}), but as pointed out to us by Hyeonjun Park, this is not sufficient to conclude this equality relating Oh-Thomas' virtual cycles. This was resolved in the next revision of the paper by using further properties of the square root Gysin map proven in \cite{park}, which appeared shortly after our first draft.} and so the claim follows from \cite[Proposition 14.1]{F}. 


\end{proof}

We apply Proposition \ref{prop_zero} to the moduli spaces of stable pairs $P:=P(X,\Ch)$.  
Letting $\Ch(m):=\Ch(F(m))$ for $(F,s)\in P$, we consider the moduli space $$P_m:=P(X,\Ch(m)).$$ 
Choose an integer $m\gg 0$ such that $H^{i>0}(F)=0$ for all $(F,s)\in P_m$, and suppose that there exists a nonsingular effective divisor $D\subset X$ corresponding to a section of $\sO_X(m)$ such that $D$ is normal to all underlying coherent sheaves in $P_m$ i.e.
\begin{equation}\label{equ_normal} \op{Tor}_1^{\sO_X}(F,\sO_D)=0 \qquad \forall (F,s)\in P_m.\end{equation}
Suppose moreover\footnote{Using the natural exact sequence $0\to \sO_S\xrightarrow{s} F\to Q\to 0$, where $S$ is the support of $F$ and $Q$ is at most 1-dimensional, one way to ensure this condition is to assume $H^2(\sO_S(-D))=0$  for all possible supporting surfaces $S\subset X$ for stable pairs in $P_m$.}, \begin{equation} \label{equ_H2F} H^{2}(F(-D))=0 \qquad \forall (F,s)\in P_m.\end{equation}
The canonical section $\sO_X\to \sO_X(D)$ gives a closed immersion \begin{equation} \label{equ_zeroemb} \iota \colon P\hookrightarrow P_m\qquad [\sO_X\to F]\mapsto [\sO_X\to F(D)].\end{equation}
Consider 
$$
\xymatrix{& &\sO_{X\times P_m} \ar[d]^-{\mathsf{s}_m} & &\\ 0\ar[r] &\bb F_m(-D)  \ar[r] & \bb F_m \ar[r] & \bb F_m|_D \ar[r] & 0}$$ over $X\times P_m$, where the vertical arrow is the universal stable pair, denoted by $\bb J_m$ in the following, and the bottom row is the natural short exact sequence that exists because of the normality assumption \eqref{equ_normal}. Here, for simplicity we use the same symbol for $D\subset X$ and its inverse image  $D\times P_m\subset X\times P_m$. Then, \eqref{equ_H2F} implies that \begin{equation} \label{vecV}V:=\pi_*\bb (F_m|_D)\end{equation} is a vector bundle over $P_m$, and $\mathsf s_m$ pushes down to a section of $V$, whose zero locus is $P\subset P_m$. The restriction of \eqref{equ_first} for $P_m$ to $P$ fits into the following commutative diagram of exact triangles on $P$
\begin{equation}\label{DFF}\xymatrix{\lra{\bb F,\bb F}\ar[r] \ar@{=}[d] &\lra{\sO,\bb F} \ar[r] \ar[d]& \lra{\bb J[-1],\bb F}\ar[d] \\ 
\lra{\bb F(D),\bb F(D)}\ar[r]  &H \ar[r] \ar[d]_-{\alpha}& \lra{\bb J_m[-1],\bb F(D)}\ar[d]^-\zeta \\
& V\ar@{=}[r] &V }\end{equation}  
where $H:=\lra{\sO,\bb F(D)}$ is a vector bundle because of our choice of $D$. As we keep working on the affine bundle mentioned above, there is no nonzero map $H^*[-3]\to V$. As a result, using the self-duality \eqref{eqn_Serre_duality} for $\lra{\bb J_m,\bb J_m}_0[1]$ and Diagram \eqref{DFF}, we can form the following commutative diagram

\begin{equation}\label{DJJ}\xymatrix{& V^*[-2] \ar@{=}[r] \ar[d] & V^*[-2]\ar[d]^-{\alpha^*[-2]}\\
  \lra{\bb J_m[-1],\bb F(D)}\ar[r] \ar[d]_-{\zeta} &\lra{\bb J_m,\bb J_m}_0[1]\ar[r] \ar[d]& H^*[-2] \\
 V\ar@{=}[r] &V &}\end{equation} 
in which the middle row is the restriction of \eqref{equ_sec} for $P_m$ to $P$. Given a locally free resolution $\lra{\bb J_m,\bb J_m}_0[1]\cong [T\xrightarrow{u} E  \xrightarrow{u^*} T^*]$  over $P_m$ as in \eqref{eqn_obstruction_complex}, we may assume the middle vertical arrows in \eqref{DJJ} are given by a map $\eta\colon T\to V$ and $\eta^*[-2]$.  By using Diagrams \eqref{DFF} and \eqref{DJJ} it is straightforward to  deduce the following isomorphism $$  \lra{\bb J,\bb J}_0[1]\cong  [T|_P\xrightarrow{[\eta,u,0]^t}V|_P\oplus E|_P\oplus V^*|_P\xrightarrow{[0,u^*,\eta^*]} T^*|_P].$$ From the construction above, we also see that if $I_P< \sO_{P_m}$ is the ideal sheaf of $P$ the natural map $V^*|_P\to I_P/I_P^2$ factors through $\alpha^*$, and by \eqref{DJJ} it also factors through $\eta^*$, so we get a commutative square as in \eqref{dpot} with the top row $$E|_P\oplus V^*|_P\xrightarrow{[u^*,\eta^*]} T^*|_P,$$ and the bottom row the truncated cotangent complex of $P$ giving a perfect obstruction theory for $P$ by Proposition \ref{prop_potZ}. These considerations place us in the situation of Proposition \ref{prop_zero}, so we have
\begin{thm} \label{thm_degen} 
Given a Chern Character vector $\Ch$ and a nonsingular divisor $D\in |\sO_X(m)|$ for $m\gg 0$ satisfying \eqref{equ_normal} and \eqref{equ_H2F}, the closed immersion \eqref{equ_zeroemb} is realized as the zero locus of a section of the vector bundle $V$  over $P_m(X,\Ch(m))$ defined in \eqref{vecV}, and  
$$\iota_*[P(X,\Ch)]^{\vir}= e(V)\cap [P_m(X,\Ch(m))]^{\vir}.$$
\end{thm} \qed

Suppose the component $\gamma \in H^4(X,\mathbb Z)$ of $\Ch$ is an irreducible class, and  consider the morphisms $$\rho\colon P(X,\Ch)\to M(X,\Ch), \qquad \rho_m\colon P_m(X,\Ch(m))\to M(X,\Ch(m))$$ introduced in \S \ref{DT-4_invariants}.  We have an isomorphism of moduli spaces $$M(X,\Ch)\cong M(X,\Ch(m))\qquad F\mapsto F(m),$$ and for $m\gg 0$ the morphism $\rho_m$ is the projection of a projective bundle.
By Proposition \ref{prop_pullback} we have (The following result is independently proven by Bae, Kool and Park in an upcoming paper \cite{BKP}.)

\begin{cor}\label {cor_degen} In the situation of Theorem \ref{thm_degen}, suppose additionally that $\gamma$ is an irreducible class then
$$\iota_*[P(X,\Ch)]^{\vir}= e(V)\cap \rho_m^*[M(X,\Ch)]^{\vir}.$$
In particular, if $\vd_M=2-\gamma^2<0$ then $\iota_*[P(X,\Ch)]^{\vir}=0$.
\end{cor} \qed

\section{Local surfaces}\label{sec_local_surfaces}  
The first class of examples that we consider is that of local surfaces. Let $S$ be a smooth projective surface and $V$ be a rank $2$ locally free sheaf on $S$ such that $\det(V)=K_S$.   
Then $$p\colon X:=\spec(\Sym V^*)\to S,$$ the total space of the associated vector bundle, is a noncompact Calabi-Yau 4-fold.   In \S \ref{subsec_local_surface} we make an assumption on the Chern character vectors $\Ch$ so that the moduli spaces of stable pairs that we consider are smooth. We are then able to express stable pair invariants as integration over Hilbert schemes of points on $S$. In \S \ref{sec_virloc} we choose the geometry of the fourfold so that it admits $\bb C^*$-action preserving its holomorphic 4-form. This enables us to apply the virtual localization formula and express stable pair invariants of the fourfold in term of the stable pair invariants of $S$.

\subsection{Smooth moduli of stable pairs and Hilbert scheme of points}\label{subsec_local_surface}
In this Subsection, we assume  $H^{i>0}(\sO_S)=0$ and also $H^0(V)=0$.    Let 
$i: S\hookrightarrow X$ be the inclusion of the zero section of 
$p$, and let $[S]\in H^4_c(X,\bb Z)$ be the class of the zero section. We work with the (compactly supported) Chern character vector $$\Ch=(0,0,[S], \beta, \xi),$$ where $\beta$ and $\xi$ are the Poincare duals of pushforward of classes on $S$ that are denoted by the same symbols by abuse of notation. Because of $H^0(V)=0$ condition and the description given in \S \ref{sec_description} for a stable pair  $[\sO_X\to F]$ in class $\Ch$ we must have $F=i_*(\sI(D))$, where $\sI\subset \sO_S$
is the ideal sheaf of a zero dimensional subscheme $Z\subset S$ of length $n=n(\beta,\xi)$, and $D\subset S$ is an effective divisor in class $\beta$. As a result, the moduli space of stable pairs $P(X,\Ch)$
is identified with the moduli space of stable pairs on $S$ \cite{PT3, KT}, or equivalently with the nested Hilbert scheme 
$$
S^{[0,n]}_\beta \subset \coprod_{L\in \pic_{\beta}(S)}S^{[n]}\times |L|$$
consisting of pairs $(D,Z)$ such that $Z\subset D$. Here, $S^{[n]}$ denotes the Hilbert scheme of $n$ points on $S$ and  $\pic_{\beta}(S)$, the isomorphism of classes of line bundles in class $\beta$, is a finite set by $H^{1}(\sO_S)=0$ condition.

In this Subsection, we also assume that any $L\in \pic_{\beta}(S)$ is $(n-1)$-very ample:
\begin{equation}\label{eqn_condition1}
H^1(\sI(D))=0; \quad \forall (D, \sI)\in |L|\times S^{[n]}.
\end{equation}
Condition (\ref{eqn_condition1}) ensures that $P$ is smooth with components of dimension 
$n+\dim|L|$ for $L\in \pic_{\beta}(S)$.  In fact for any such component the projection 
$P\to S^{[n]}$ is a projective bundle with fibers $\pp(H^0(\sI(D)))$, which is of dimension $h^0(\sO(D))+n-1$.

Let $\pi_S, \pi_{P}$ be the projections to the first and the second factors of $S\times P$, respectively, and let $\j:=[\sO_{S\times P}\to \i(D)]$ be the universal stable pair over $S\times P$. 

\begin{prop}\label{prop_virtual_class_P}
The obstruction bundle for the moduli space $P$ is given by 
$$\Ob_{P}=\sE xt^1_{\pi_{P}}(\i, \pi_S^*V\otimes \i)$$
and the Oh-Thomas virtual fundamental class is given by 
$$[P]^{\vir}=\sqrt{e}(\Ob_{P})\in A_{\frac{\vd}{2}}\left(P, \zz\Big[\frac{1}{2}\Big]\right),$$
where $\vd=2\chi(\sI(D))-[S]^2=2h^0(D)-2n-e(V)$.
\end{prop}
\begin{proof}
Let $E$ be a locally free sheaf on $S$, then we have the Koszul complex 
$$0\to p^*(K_S^*\otimes E)\rightarrow p^*(V^*\otimes E) \rightarrow p^*E\rightarrow i_*E\to 0,$$
in which the arrows are zero when restricted to $S$. As a result, we have 
\begin{equation}\label{eqn_restriction_S}
Li^*i_*E\cong E\oplus (V^*\otimes E)[1]\oplus (K_S^*\otimes E)[2].
\end{equation}
A stable pair $J_X:=[\sO_X\to i_* \sI(D)]$ in class $\Ch$ corresponds by adjunction to a pair $J_S:=[\sO_S\to \sI(D)]$ on $S$. Writing a finite locally free resolution for $i_*\sI(D)$, applying (\ref{eqn_restriction_S}), and then using the natural exact triangle \begin{equation*}\label{eqn_dis_triangle}J_X[-1]\to \sO_X\to i_*\sI(D),\end{equation*} we get
\begin{equation}\label{eqn_restriction_S-2}
Li^*J_X\cong J_S\; \oplus \;V^*\otimes \sI(D)[1]\; \oplus \;\sI(D-K_S)[2].
\end{equation}

Consider the exact triangle
\begin{equation}\label{eqn_exact_triangle1}
R \Hom(J_X[-1], i_*\sI(D))\rightarrow R\Hom(J_X, J_X)_0[1]\rightarrow R\Hom(i_*\sI(D), \sO_X)[2]
\end{equation}
resulted from \eqref{equ_sec}. 
Using an adjunction between $Li^*$ and $i_*$ and (\ref{eqn_restriction_S-2}), the leftmost term in the exact triangle (\ref{eqn_exact_triangle1}) can be written as
$$R\Hom(J_S[-1], \sI(D))\oplus R\Hom(\sI, V\otimes \sI)\oplus R\Hom(\sI, \sI(K_S))[-1].$$
By Serre duality and adjunction again, the rightmost term in the exact triangle  (\ref{eqn_exact_triangle1})  is isomorphic to 
$(R\Gamma (\sI(D)))^{\vee}[-2]$. Moreover, the arrow $$R\Hom(J_X[-1],\sO_X)\to R\Hom(J_X[-1],i_*\sI(D))$$ in the exact triangle  in the bottom row of diagram above factors through
$$R\Hom(J_X[-1],i_*\sO_S)\cong R\Hom(J_S[-1], \sO_S)\oplus R\Hom(\sI(D), V)\oplus R\Hom(\sI(D), K_S)[-1]$$
in which the last summand is identified with $(R\Gamma (\sI(D)))^{\vee}[-3]$ by Serre duality. This means that under identifications above the arrow $$R\Hom(i_*\sI(D), \sO_X)[1]\to R \Hom(J_X[-1], i_*\sI(D))$$ in the exact triangle \eqref{eqn_exact_triangle1} is simply $[0\; 0\; h]^T$, where $$h\colon (R\Gamma (\sI(D)))^{\vee}[-3]\to  R\Hom(\sI(D), \sI(D))^\vee[-3]\cong R\Hom(\sI, \sI(K_S))[-1]$$ is the natural arrow induced by $\sO_S\to \sI(D)$ and Serre duality.
Therefore, we finally arrive at \begin{align} \label{JXJX} \bb E^\vee_{J_X}=&R\Hom(J_X,J_X)_0[1]\\ \cong \notag& R\Hom(J_S[-1], \sI(D))\oplus R\Hom(\sI, V\otimes \sI) \oplus R\Hom(J_S[-1], \sI(D))^\vee[-2].\end{align}
Note that the natural paring  $V\otimes V\to \det(V)=K_S$ gives an isomorphism $V\cong V^*\otimes K_S$, and so by Serre duality $R\Hom(\sI, V\otimes \sI)^\vee\cong R\Hom(\sI, V\otimes \sI)[-2]$, and this shows that $\bb E_{J_X}^\vee\cong \bb E_{J_X}[-2]$.

Next, we take cohomology from \eqref{JXJX}: 
\begin{align*}T_{P, J_X}&=\Ext^1_X(J_X, J_X)_0\cong \Hom_S(J_S[-1], \sI(D)),\\
\Ob_{P, J_X}&=\Ext^2_X(J_X, J_X)_0\cong \Ext^1_S(\sI, V\otimes\sI).\end{align*}
Here, we have used the vanishing (\ref{eqn_condition1}) as well as 
\begin{align*} &\Hom_S(\sI, V\otimes\sI)=H^0(V)=0,\\
&\Hom_S(\sI, \sI(K_S))=H^0(K_S)=0  \quad \text{ and } \quad \Ext^{i\ge1}_{S}(J_S[-1], \sI(D))=0, \end{align*} where the vanishing of $\Ext^{i\ge 1}_{S}$ follows from the cohomology exact sequence of an analog of \eqref{equ_first} for $J_S$.
by using the condition  (\ref{eqn_condition1}) again. We also have
$$\Ext^2_{S}(\sI, V\otimes \sI)\cong \Hom_S(\sI, V\otimes \sI)^*=0$$
and that $\Ext^1_{S}(\sI, V\otimes \sI)$ is self-dual. 
Moreover, by Riemann-Roch theorem the latter has the constant dimension 
$r=4n+e(V)-2$ as $\sI$ varies in $S^{[n]}$.

This analysis can be similarly done in families. Therefore,  $$\Ob_{P}=\sE xt^1_{\pi_{P}}(\i, \pi_S^*V\otimes \i)$$ is a self-dual obstruction bundle on the moduli space 
$P$ and its square root Euler class gives the virtual fundamental class by \eqref{BF56}, as claimed. 
\end{proof}
\bigskip
\noindent \textbf{Examples:} Let $S=\pp^2$.
\begin{itemize}
\item  If $V=\Omega_{\pp^2}$, then $\vd$ is odd and hence $[P_S]^{\vir}=0$.  

\item If $V=\sO_{\pp^2}(-1)\oplus \sO_{\pp^2}(-2)$, and  $\beta$ corresponds to $m$ times the class of a line  $[\pp^1]$ inside $\pp^2$, then 
$$[P_S]^{\vir}=\sqrt e(\Ob_{P})=e\left(\sE xt_{\pi_P}^1(\i,\i(-1))\right)\in A_{m(m+3)/2-n}\left(P,\zz\left[\small \frac 1 2\right]\right).$$
In particular, for $n=m(m+3)/2$ we have $P=(\pp^2)^{[n]}$ (recall $P\to \pp^2$ is a projective bundle with fibers of dimension 
$m(m+3)/2-n$) and the degree of $[P]^{\vir} \in A_0(\left(P,\zz\left[\small \frac 1 2\right]\right)$  can be read off from Carlsson-Okounkov formula 
\cite[Corollary 1]{CO}:
$$\sum_{n\geq 0}q^n\int_{(\pp^2)^{[n]}}e\left(\sE xt_{\pi_P}^1(\i,\i(-1))\right)=\prod_{n\ge 0}\frac{1}{(1-q)^7}.$$
If $n<m(m+3)/2$ the virtual cycle $[P]^{\vir}$ is the pullback of the 0-dimensional virtual class of the previous case via the projective bundle projection $ P\to \pp^2$. If $n>m(m+3)/2$ then the condition \eqref{eqn_condition1} fails but on the other hand, the virtual dimension is negative and hence $[P]^{\vir}=0$.
 
\end{itemize}

\subsection{Virtual localization} \label{sec_virloc}

We use most of the set up in \S \ref{subsec_local_surface}, but no longer need any of the vanishing conditions in there. In this Subsection, we assume $V=L_1\oplus L_2$ is the direct sum of two invertible sheaves. 

 Let $\cc^*$ act on $X$ such that the action preserves the holomorphic 4-form. This can be obtained by acting on $S$ trivially and 
on each fiber of $V$ with the weight $(1,-1)$.  In particular, the induced action on $\Lambda^2V\cong K_S$ is trivial.  If 
$\mathfrak{t}$ is the standard $\cc^*$-representation then we can write
$$V=L_1\mathfrak{t}\oplus L_2\mathfrak{t}^{-1}.$$
We will sometimes drop the tensor product from the notation. Let $t:=c_1(\mathfrak{t})$.

The moduli space of stable pairs $P:=P(X,\Ch)$ on $X$ is no longer compact. 
However, the induced $\cc^*$-action on the moduli space has the compact fixed locus 
$$\iota \colon P^{\cc^*}\cong S^{[0,n]}_{\beta} \subset P(X,\Ch)$$ by the discussion in \S \ref{sec_description}.
Let $J_X:=[\sO_X\to i_*\sI(D)]$ be a stable pair in the fixed locus $P^{\cc^*}$ and 
$J_S:=[\sO_S\to \sI(D)]$ be the corresponding pair on $S$ using adjunction. These two are related via the $\cc^*$-equivariant version of 
(\ref{eqn_restriction_S-2})
$$
Li^*J_X\cong J_S\oplus (L^*_1\mathfrak{t}^{-1}\oplus L^*_2\mathfrak{t})\otimes \sI(D)[1]\oplus \sI(D-K_S)[2].
$$
We use this to extract the fixed and moving parts of the obstruction theory of $P$ restricted to $P^{\cc^*}$.
\begin{lem}
The fixed part of the stable pairs obstruction theory gives the virtual cycle of the moduli space of stable pairs on $S$ \cite{PT3, KT} as well as the virtual cycle of the nested Hilbert schemes \cite{GSY, GT1, GT2}
$$[P^{\cc^*}]^{\vir}=[S_\beta^{[0,n]}]^{\vir}\in A_{n+\beta(\beta-K_S)/2}(P^{\cc^*}).$$
\end{lem}
\begin{proof}
By \eqref{JXJX} for any $J_X \in P^{\bb C^*}$ as above we have 
$$R\Hom(J_X, J_X)_0^{\fix}[1]\cong R\Hom(J_S[-1], \sI(D))\oplus R\Hom(J_S[-1], \sI(D))^{\vee}[-2].$$
Using the first column of the diagram in \S \ref{sec_description}, the complex $J_S$ is identified with $\sI_{Z\subset D}(D)$, therefore by \cite[Appendix A.3]{KT} $$R\Hom(J_S[-1], \sI(D))\cong R\Hom(\sI_{Z\subset D}, \sI)[1]$$ is identified with the natural virtual tangent bundle of $P^{\cc^*}$ at the point $J_S$.  Carrying out this analysis with the universal family 
$\j:=[\sO_{S\times P^{\bb C^*}}\to \i(\mathcal D)]$ of the fixed locus we can identify the virtual tangent bundles of $P^{\cc^*}$ arising from these two stable pair theories on $X$ and $S$ (see \cite[Section 8]{OT}), and hence their virtual cycles as they depend on the K-theory classes of the virtual tangent bundles.
Finally, this  virtual cycle was identified with $[S_\beta^{[0,n]}]^{\vir}$ in \cite{GSY}, and so the Lemma is proven.
\end{proof}

For the universal stable pair $\j$ in the proof above $\mathcal D$ and $\bb I$  are respectively the pullbacks of the universal divisor on $S\times S_\beta$ and the universal ideal $\bb I < \sO_{S\times S^{[n]}}$. Let $\pi$ be the projection $\pi\colon S\times S^{[n]}\times S_\beta \to S^{[n]}\times S_\beta$ and $$\mathsf{CO}^{[0,n]}_\beta=R\pi_*\sO(\sD)-R\pi_*\bb I(\sD)$$ be the K-theory class of the rank $n$ tautological bundle on $S^{[n]}\times S_\beta$ (see \cite{GT2} for a more general definition). Then by \cite[Theorem 5]{GT2} (see also \cite[Proposition 2.1]{K}) we have 
\begin{cor} Under the inclusion in \eqref{nested}
$$j_*[P^{\bb C^*}]^{\vir}=c_n(\mathsf{CO}^{[0,n]}_\beta)\cap [S^{[n]}]\times [S_\beta]^{\vir}.$$
\end{cor} \qed

The moving part of \eqref{JXJX} gives:
$$
R\Hom(J_X, J_X)_0^{\mov}[1]\cong R\Hom(\sI, (L_1\mathfrak{t}\oplus L_2\mathfrak{t}^{-1})\otimes \sI).
$$
The same calculation using the universal family $\j=[\sO_{S\times P^{\bb C^*}}\to \i(\mathcal D)]$ shows that the virtual normal bundle $N^{\vir}$ of $P^{\cc^*}$ is given by 
$$\sqrt{e_{\cc^*}}(N^{\vir})=
\frac{e_{\cc^*}(\mathcal H om_{\pi_P}(\i, \i V))}{\sqrt{e_{\cc^*}}(\sE xt^1_{\pi_P}(\i,  \i V))}=
\frac{t^{h^0(L_1)}(-t)^{h^0(L_2)}}{e_{\cc^*}(\sE xt^1_{\pi_P}(\i, \i L_1\mathfrak{t}))}.
$$ 
The Oh-Thomas localization formula gives
\begin{prop}
 The  equivariant virtual fundamental class of $P$ is given by:
$$[P]^{\vir}=\frac{
(-1)^{h^0(L_2)}}{t^{\chi(L_1)}}\;\iota_*\Big( e_{\cc^*}\big(\sE xt^1_{\pi_P}(\i, \i L_1 \mathfrak{t})_0\big)\cap[P^{\cc^*}]^{\vir}\Big)
\in A_{\vd/2}(P, \qq)[t, 1/t],$$
where $\vd$ is given by $$2\chi(\sI(D))-[S]^2=\beta(\beta-K_S)+2\chi(\sO_S)-2n-e(V)=\beta(\beta-K_S)+2\chi(L_1)-2n.$$
\end{prop} \qed

\bigskip
\noindent \textbf{Example:} Suppose that $p_g(S)>0$, $L_1=\sO_S$, $L_2=K_S$, and $\mathsf{SW}_\beta=\deg[S_\beta]^{\vir}$ be the Seiberg-Witten invariant in class $\beta$ (it is 0 if $\beta(\beta-K_S)\neq 0$).  If in Definition \ref{defn_stable_pair_invariants} we  choose $\alpha_i \in H^{*}_{\bb C^*}(X, \bb Q)$ supported on the zero section $S$. By \cite[Theorem 6]{GT2}  $$\langle \prod_{i=1}^s \alpha_i \rangle^P_{\Ch}=\frac{
(-1)^{p_g(S)}}{t^{\chi(\sO_S)}} \mathsf{SW}_\beta  \cdot \int_{S^{[n]}}  c_n(L^{[n]})\; e_{\cc^*}\big(\sE xt^1_{\pi_P}(\i, \i\mathfrak{t})_0\big)\prod_{i=1}^s\tau(\alpha_i)\in  \qq[t, 1/t]$$ for any $L \in \op{Pic}_\beta(S)$.

\section{Local threefolds} \label{sec_local_three} In this Section, we take $X$ to be the total space of the canonical bundle $K_Y$ of a nonsingular projective threefold. Denote by $p\colon X\to Y$ and $i\colon Y\to X$ the bundle projection and the inclusion of the zero-section, respectively. Then $X$ is a noncompact Calabi-Yau fourfold. We consider the moduli spaces of stable pairs $P(X,\Ch)$ in which $\Ch$ are the pushforward by $i$ of some Chern character vectors on $Y$.

\subsection{Stable pairs of 2-dimensional sheaves on Fano threefolds}\label{sec_Fano_threefolds}

Let $Y$ be a nonsingular Fano projective threefold and  
$$\Ch_Y=(0,\gamma,\beta,\xi) \in H^{2*}(Y,\qq)$$ be a Chern character vector with $\gamma=c_1(L)$ for some  ample invertible sheaf $L$. Let $$P_Y:=P(Y,\Ch_Y)$$ be the moduli space of stable pairs $(G,s)$, where $G$ is a pure 2-dimensional sheaf and $s$ is a section of $G$ with at most 1-dimensional cokernel. This stability condition can be realized as a limit of Le Potier stability condition as in \S \ref{sec_mst}, and then $P(Y,\Ch_Y)$ is constructed as a moduli space of Le Potier's pairs. Our goal in this Subsection is to define a virtual cycle for $P(Y,\Ch_Y)$ using the theory of stable pairs on Calabi-Yau fourfold $X$.

We denote by $\Ch:=i_* \Ch_Y$ the compactly supported class on $X$ corresponding to $\Ch_Y$. Note that $p_*\Ch=\Ch_Y$, so we may as well define $\Ch_Y(F):=p_*\Ch(F)$ for any compactly supported coherent sheaf $F$ on $X$.

\begin{lem} \label{PY} $P(X,\Ch)$ is projective over $\cc$ containing $P(Y,\Ch_Y)$ as a closed subscheme. If $\Ch_Y$ is such that all stable pairs over $X$ in class $\Ch$ are scheme theoretically supported in $Y$ (e.g. if $\gamma$ is a reduced class ) then $P(Y,\Ch_Y)=P(X,\Ch)$.
\end{lem}
\begin{proof} 
Suppose $J_X= [\sO_X\to F]$ is a stable pair in class $\Ch$. We will show that if $S\subset X$ is the reduced support of $F$ then $S\subset i(Y)$. To see this, note that $S\subset X$ can be thought of as a section of the bundle $K_Y^r|_{p(S)}$ for some $r>0$.  Since $L$ is ample, so is the divisor $D=p(S)$. Since by assumption $K_Y^*$ is also ample by Kodaira vanishing and Serre duality $$H^0(K_Y^r)=0=H^1(\sO(-D)\otimes K_Y^r).$$ Taking cohomology of the exact sequence $$0\to \sO(-D)\otimes K_Y^r \to K_Y^r\to K_Y^r|_{D}\to 0,$$  we see that $H^0(K_Y^r|_{D})=0$ and so $S$ must be the zero section of $K_Y^r|_{D}$ and the claim is proven. This shows that $P(X,\Ch)$ is proper and hence projective over $\cc$.

There is a natural morphism $\iota \colon P(Y,\Ch_Y) \to P(X,\Ch)$ that maps a stable pair $J_Y:=[\sO_Y\to G]$ in class $\Ch_Y$ to the composition $\sO_X\to i_* \sO_Y \to i_*G$, which gives a stable pair on $X$ in class $\Ch$ that we denote by $J_X$. $X$ is equipped with a $\cc^*$-action that scales the fibers of the bundle map $p$. This induces a $\cc^*$-action on $P(X,\Ch)$ and  $\iota$ can be easily identified with the inclusion of one of the $\cc^*$-fixed components. Therefore $\iota$ is a closed immersion. 
If $\Ch_Y$ satisfies the condition in the Lemma for any stable pair $J_X=[\sO_X\to F]$ in class  $\Ch$ we must have $F=i_*G$ and then by adjunction we get a stable pair $J_Y=[\sO_Y\to G]$ in class $\Ch_Y$ such that $\iota(J_Y)=J_X$. Therefore, $\iota$ is an isomorphism.    


\end{proof}
 


Let $\bb J:=[\sO_{X\times P}\to  \bb F]$ and $\bb J_Y:=[\sO_{Y\times P_Y}\to  \bb G]$ be the universal stable pairs.  We use the same symbol $i$ (resp. $p$ and $\iota$) for any basechange $i\times 1$  (resp. $p \times 1$ and $1\times \iota$). We also denote the derived restriction of $\bb J$ to $X\times \iota(P(Y,\Ch_Y)$ by the same symbol. We finally identify $i_* \bb G$ with the restriction of $ \bb F$ to $X\times \iota(P_Y)$. 

Our goal is to express the restriction of the virtual tangent bundle $\bb E^ \vee$ of the obstruction theory \eqref{eqn_obstruction_theory} to $\iota(P_Y)$ in terms of $\bb J_Y$ and $\bb G$. We denote the restriction by $\bb E^{ \vee}_Y$.
\begin{lem} \label{dsum} There are isomorphisms 
$$\bff L i^* i_*\bb G\cong \bb G\oplus \bb G(-K_Y)[1],\qquad \bff Li^*\bb J \cong  \bb J_Y \oplus \bb G(-K_Y)[1]$$ on $Y\times  \iota(P_Y)$.
\end{lem}
\begin{proof}
For any locally free sheaf $V$ on $Y\times  \iota(P_Y)$ we have a natural short exact sequence on $X\times  \iota(P_Y)$
$$0\to p^*V(-K_Y)\stackrel{t}{\rightarrow} p^*V\to i_*V\to 0$$ with $t|_{Y\times  \iota(P(Y,\Ch))}=0$. 
The last vanishing implies that $Li^*i_* V\cong V\oplus V(-K_Y)[1]$. The first isomorphism in the statement follows by writing a finite locally free resolution for $i_*\bb G$ and using the decomposition above for each term of the resolution. The second isomorphism in the statement easily follows from the first by applying $Li^*$ to the natural exact triangle $\bb J[-1]\to \sO_{X\times  \iota(P_Y)}\to i_* \bb G.$
\end{proof}

Consider the following parts of the natural exact triangles on $\iota(P_Y)$:
\beq{text1}
\xymatrix{& \lra{\bb J,\bb J}\ar[d] &\\
\lra{\sO_{X\times \iota(P_Y)}, \sO_{X\times \iota(P_Y)}} \ar[r] &  \lra{\bb J[-1],\sO_{X\times \iota(P_Y)}} \ar[r] \ar[d]&\lra{\bb F, \sO_{X\times \iota(P_Y)}}[1].\\ &\lra{\bb J[-1], \bb F} &
}
\eeq
%
The first horizontal arrow in \eqref{text1} factors as the identity map followed by the first vertical arrow. In particular, the composition of the first horizontal arrow and the second vertical arrow is zero. As a result, we get a map \beq{map} \lra{\bb F, \sO_{X\times \iota(P_Y)}}[1] \to \lra{\bb J[-1], \bb F}.\eeq  Note also that the second vertical arrow in \eqref{text1} factors naturally through $\lra{\bb J[-1],i_*\sO_{Y\times \iota(P_Y)}}$.
By Duality, adjunction and Lemma \ref {dsum} \begin{align} \label{decm} &\lra{\bb F, \sO_{X\times \iota(P_Y)}}[1] \cong \lra{\bb G(-K_Y), \sO_{Y\times \iota(P_Y)}}, \\ \notag&  \lra{\bb J[-1], \bb F}\cong \lra{\bb J_Y[-1], \bb G}\oplus \lra{\bb G(-K_Y), \bb G}, \\ \notag & \lra{\bb J[-1], i_*\sO_{Y\times \iota(P_Y)}}\cong \lra{\bb J_Y[-1], \sO_{Y\times \iota(P_Y)}}\oplus \lra{\bb G(-K_Y), \sO_{Y\times \iota(P_Y)}}.\end{align}
Thus, we can write \eqref{map} as 
$$\lra{\bb G(-K_Y), \sO_{Y\times \iota(P)}}\xrightarrow{[0\; h]^T} \lra{\bb J_Y[-1], \bb G} \oplus \lra{\bb G(-K_Y), \bb G},$$ where
$h$ is naturally induced by the universal  section $\sO_{Y\times P}\to  \bb G$. But the mapping cone of \eqref{map} is $\bb E^{ \vee}_Y=\lra{\bb J, \bb J}_0[1]$ by the octahedron axiom, and also $\operatorname{Cone}(h)\cong \lra{\bb G(-K_Y), \bb J_Y}$. So we have the exact triangle 
$$\lra{\bb G(-K_Y), \bb J_Y}\xrightarrow{0} \lra{\bb J_Y[-1], \bb G} \to \bb E_Y^\vee, $$ and so using duality 
\beq{text3}
 \bb E_Y \cong \lra{\bb J_Y[-1], \bb G}[2]  \oplus  \lra{ \bb J_Y[-1],\bb G}^\vee.
\eeq
This in particular shows that $ \bb E_Y\cong  \bb E_Y^\vee[2]$.


We have proven 
\begin{thm} \label{fano}
Let $Y$ be a nonsingular projective Fano 3-fold. Suppose that $\Ch_Y$ satisfies the condition in Lemma \ref{PY}. Then $P(Y,\Ch_Y)$ is equipped with a self dual obstruction theory $\bb E_Y\to \bb L_{P_Y}$ in the sense of \cite{OT}, where $\bb E_Y$ is given by \eqref{text3}. This gives a virtual cycle 
$$[P(Y,\Ch_Y)]^{\vir} \in A_{\vd/2}\big(P(Y,\Ch_Y), \bb Z[\frac 12]\big),$$ where $\vd=2(\gamma \cdot\td_2(Y)-\beta\cdot K_Y/2+\xi)-\gamma^2\cdot K_Y.$
\end{thm} \qed

Here, we discuss some implications of this Theorem. In the situation of Theorem \ref{fano}, we know that the virtual tangent bundle $\bb E^\vee_Y$ has perfect amplitude in $[0,2]$. This means that for any stable pair $J_Y=[\sO_Y\to G]$ in class $\Ch_Y$
$$h^i(\bb E^\vee|_{J_Y})\cong\Ext^i(J_Y[-1],G)\oplus \Ext^{i+1}(G(-K_Y),J_Y[-1])^*=0 \qquad i\neq 0, 1, 2.$$
For $i> 3$ or $i<-1$ the vanishing is obvious. For $i=-1$ and $i=3$ the vanishing follows from Serre duality and the exact sequence $$0=H^3(G)\to \Ext^3(J_Y[-1],G) \to \Ext^4(G,G)=0.$$ Note also that since $$h^0(\bb E_Y^\vee)\cong T_{P_Y,J_Y}\cong  \Hom(J_Y[-1],G)$$ we conclude
\begin{equation} \label{vanext} \Ext^1(G(-K_Y),J_Y[-1])^*\cong\Ext^2(J_Y[-1],G)=0.\end{equation}
 From the exact sequence 
$$ \Ext^3(G,G)\xrightarrow{\operatorname{tr}} H^2(G)\to \Ext^2(J_Y[-1],G) \to \Ext^3(G,G)\to 0$$ we see that the vanishing of the  $\Ext^2(J_Y[-1],G)$ obtained in \eqref{vanext}
implies $\Ext^3(G,G)=0$ and also that the trace map above is surjective.  

\begin{rmk} \label{rem_non}
i) If the scheme theoretic support of $G$ is a nonsingular surface $j:S\subset Y$ then $G=j_*\sI_Z(D)$ by the discussion in \S \ref{sec_description}. In this case $$H^2(G)=0=\Ext^3(G,G).$$ Here, the vanishing of $H^2$ follows from the cohomology exact sequence of the natural short exact sequence $0\to \sI_Z(D)\to \sO_{S}(D)\to \sO_Z\to 0$, and the vanishing of $\Ext^3$ follows from Serre duality and the stability of $G$. This gives an explanation for the vanishing \eqref{vanext}. At the moment, we do not have such a direct explanation for \eqref{vanext} if the support of $G$ is not nonsingular.

ii) The vanishing $\Ext^{i\ge 2}(J_Y[-1],G)=0$ discussed above implies that $\lra{ \bb J_Y[-1],\bb G}$ has perfect amplitude in $[0,1]$. This suggests that one may be able to construct a (natural) perfect obstruction theory $\lra{\bb J_Y[-1],\bb G}^\vee\to \bb L_{P_Y}$ that will lead to the same virtual cycles given in Theorem \ref{fano} (see \cite[Section 8]{OT}).
\end{rmk}


\subsection{Local $\pp^3$}
We study the example $Y=\pp^3$ using the notation and set up of this Section so far. Let $H\in H^2(\bb P^3,\bb Z)$ be the hyperplane class. We can write
$$\Ch_{\bb P^3}=(0,aH,bH^2/2,cH^3/6) \in H^{2*}(\pp^3,\qq)$$ for some integers $a, b, c$ with $a>0$. As before, let $\Ch:=i_* \Ch_{\pp^3}$.


\subsubsection{Planes inside $\pp^3$} \label{subsec_planes}
We start by taking $a=1$.
For any stable pair $J_X=[\sO_X\to F]$ in class $\Ch=i_*\Ch_{\bb P^3}$, the 2-dimensional sheaf $F$ is of the form $F=i_*j_*\sI(m)$, where $j\colon S\hookrightarrow \bb P^3$ is the inclusion of some $\bb P^2\cong S \in |\sO_{\bb P^3}(1)|$, 
$\sI<\sO_S$ is the ideal of a 0-dimensional subscheme $Z\subset D$ of length $n$, and $D\subset S$ is a degree $m\ge 0$ curve. If $m=0$ then $n=0$ also.  The integers $m$ and $n$ are determined by the choices of $\beta$ and $\xi$ in $\Ch_{\bb P^3}$.
Since $\Ch_{\bb P^3}$ satisfies the condition in the second part of Lemma \ref{PY}  we have $P(\bb P^3, \Ch_{\bb P^3})=P(X, \Ch)$. Moreover, we are in the situation of Remark \ref{rem_non}.


Suppose now that $\dim |\sO_{\bb P^2}(m)|=m(m+3)/2\ge n$. We claim that $P(\bb P^3, \Ch_{\pp^3})$ is nonsingular with the obstruction bundle (using \eqref{text3})
\beq{h1E} \op{Ob_P}=h^1(\bb E^\vee)\cong \sE xt^1_p(\bb J_{\pp^3}[-1],j_* \bb I(m,1))\oplus \sE xt^1_p(\bb J_{\pp^3}[-1],j_* \bb I(m,1))^*,\eeq 
where $\bb I(m,1)$ is a universal object with $P(\bb P^3, \Ch_{\pp^3})$-fibers $\sI(m)$ as described below. To see this, 
taking the support of $j_*\sI(m)$ gives a morphism $$\rho\colon P(\bb P^3, \Ch_{\pp^3})\to |\sO_{\bb P^3}(1)|$$ with fibers are identified with $$P(\op{tot}(\sO_{\bb P^2}(1)\oplus \sO_{\bb P^2}(-4)), \Ch)$$  studied in \S \ref{sec_local_surfaces}. 
If $\bb S \subset |\sO_{\bb P^3}(1)| \times \bb P^3$ denotes the universal divisor then $P(\bb P^3, \Ch_{\pp^3})$ is identified with the relative Hilbert scheme of points $\Hilb^n(\bb S/|\sO_{\bb P^3}(1)|)$. Let $\bb I$ be the universal ideal sheaf over \beq{eye} j\colon \bb S\times_{|\sO_{\bb P^3}(1)|} \Hilb^n(\bb S/|\sO_{\bb P^3}(1)|)\subset \bb P^3\times \Hilb^n(\bb S/|\sO_{\bb P^3}(1)|).\eeq The universal stable pair over $ \bb P^3\times P(\bb P^3, \Ch_{\pp^3})$ is then given by $$\bb J_{\bb P^3}=[\sO_{ \bb P^3 \times P(\bb P^3, \Ch_{\pp^3})}\to j_* \bb I(m,1)],$$ where $\sO(m,1)$ is the pullback of $ \sO(m) \boxtimes \sO(1) $  from $\bb P^3\times   |\sO_{\bb P^3}(1)|$. 
By our assumption, $\sO_{\bb P^2}(m)$ is $(n-1)$-very ample, similar to \S \ref{subsec_local_surface} we can conclude that $P(\bb P^3, \Ch_{\pp^3})$ is nonsingular of dimension $n+m(m+3)/2+3.$

Let $\pi'$ be the composition of  $j$ and the projection to the second factor in \eqref{eye}. Using adjunction and the vanishing of $\sE xt^1_{\pi'}(\bb I, \bb I(m,1))$ that follows from  $\sO_{\bb P^2}(m)$ being $(n-1)$-very ample, similar to \S \ref{subsec_local_surface}, we can rewrite \eqref{h1E} as
$$ \Ob_P\cong  \sE xt_{\pi'}^1(\bb I,\bb I(1,1)) \oplus \sE xt_{\pi'}^1(\bb I,\bb I(1,1))^*.$$ 
By \eqref{BF56} the virtual cycle of $P(\bb P^3, \Ch_{\bb P^3})$ is therefore $$[P(\bb P^3, \Ch_{\bb P^3})]^{\vir}=\sqrt{e}(\Ob_P)=e(\sE xt_{\pi'}^1(\bb I,\bb I(1,1)))\in A_{\vd/2}\left(P(\bb P^3, \Ch_{\pp^3}),\bb Z\left[\frac {1} 2\right]\right),$$
where $\vd=m(m+3)-2n+6$. 


\subsubsection{Quadrics in $\pp^3$} \label{quadrics} Next, we take $a=2$. Let $J_X:=[\sO_X\to F]$ be a stable pair in class $\Ch=i_*\Ch_{\pp^3}$. 
 There exists a 2-dimensional closed subscheme $S\subset X$ and a short exact sequence $0\to \sO_S\to F\to Q\to 0$, where $Q$ is at most 1-dimensional. This shows that if we write $\Ch_{\pp^3,2} (\sO_S)=b' H^2/2$ then $b\ge b'$. 
  
There are only two possibilities in this case: either $S$ is a subscheme of the zero section $\bb P^3 \subset X$ or $S$ has a non-reduced structure that thickens out of $\bb P^3 \subset X$. The 2-term locally free resolutions of the ideals of $S$ in these two cases are 
$$p^*\sO(2)\to p^*\sO(-2)\oplus p^*\sO(4), \qquad p^*\sO(7)\to p^*\sO(-1)\oplus p^*\sO(8),$$ respectively. The Chern character $\Ch_{\pp^3}(\sO_S)$ in these two cases are respectively
$$(0,2H,-2H^2,4H^3/3),\qquad (0,2H,3H^2,19H^3/3).$$

We conclude that for $b<-4$ the moduli space is empty, and as long as $b\le 5$ there will be no contributions from the stable pairs thickening out of the zero section $\bb P^3 \subset X$, and hence Theorem \ref{fano} can be applied to get a virtual cycle
$$[P(\bb P^3,\Ch_{\pp^3})]^{\vir} \in A_{\vd/2}\big(P(\pp^3,\Ch_{\pp^3}), \bb Z[\frac 12]\big),$$ where $\vd=2(11/3+b+c/6)+16.$ 

\subsubsection{Higher degree surfaces} If $a>2$ there will be more possibilities of cases to consider as in \S \ref{quadrics}. It is always possible to find an upper bound $b_0$ such that for $b\le b_0$ Theorem \ref{fano} can be applied.

\section{Compact examples} \label{sec_compact}

In this Section, we consider a few examples of the stable pair moduli spaces on projective Calabi-Yau 4-folds.

\subsection{Sextic 4-fold} \label{sec_sextic}

 Let $j\colon X\subset \pp^5$ be a nonsingular sextic 4-fold. Then,
$j^*\colon H^k(\pp^5,\zz)\to   H^k(X, \zz)$ is an isomorphism for $k<4$ and injective for $k=4$ (Lefschetz theorem), and moreover, $H^k(\pp^5,\zz)\cong   H^k(X, \zz)$ for $k\ge 6$  with $j^*$ is multiplication by 6.  

Let $\gamma \in H^4(X,\zz)$ (resp. $\ell \in H^6(X,\zz)$) be the image of the class of a plane (resp. $\frac{1}6$ the class of a line) under $j^*$, and let $\op{pt} \in H^8(X,\zz)$ be the point class.  Suppose $i\colon S\subset X$ is a nonsingular sextic surface, i.e. the intersection of a general $\bb P^3 \subset \bb P^5$ with $X$. Then
$$\Ch:=\Ch(i_*\sO_S)=(0,0,\gamma, -6\ell, \frac 72\op{pt}).$$
Let $P=P(X,\Ch)$ be the moduli space of stable pairs in class $\Ch$. The virtual dimension of $P$ is $\vd=2\chi(\sO_S)-\gamma^2=22-6=16$. 
If \begin{equation}\label{jos} 0\to I_S\to \sO_X\to i_*\sO_S\to 0\end{equation} is the short exact sequence defining $S\subset X$ there is a component of $P$ containing the stable pair $$J_X:=[\sO_X\to i_*\sO_S]\cong I_S[1].$$ We claim that $P$ is nonsingular at this point by showing that the corresponding obstruction space $\Ext^2(J_X,J_X)_0=\Ext^2(I,I)_0$ vanishes.

By the choice of $S$ we can write $J_X \cong [\sO_X(-2)\to \sO_X(-1)^2]$. As in \S \ref{subsec_local_surface} we can write $$Li^*i_* \sO_S \cong \sO_S\oplus \sO_S(-1)^{2}[1]\oplus \sO_S(-2)[2],$$ and by Grothendieck-Verdier duality $(i_*\sO_S)^\vee \cong i_*\sO_S(2) [-2]$. From this and adjunction we can see that the natural map $R\Hom(i_*\sO_S, \sO_X)\to R\Hom(i_*\sO_S, i_*\sO_S)$ is given by  $$R\Gamma (\sO_S(2))[-2]\xrightarrow{[0\; 0\; 1]^t} R\Gamma(\sO_S)\oplus R\Gamma(\sO_S(1))^{2}[-1]\oplus R\Gamma(\sO_S(2))[-2].$$ Dualizing \eqref{jos} gives an exact triangle $\sO_X\to I_S^\vee \to (i_*\sO_S)^\vee[1]$, from which, by the fact that the natural map $R\Gamma(\sO_X) \to R\Hom(I_S,\sO_X)$ factors through $ R\Hom(I_S,I_S)$ via identity, we see that the natural map $$R\Hom(I_S, \sO_X)\to RHom(I_S, i_*\sO_S)\cong R\Gamma(\sO_S(1))^{2}\oplus R\Gamma(\sO_S(2))[-1]$$ must factor through \begin{equation}\label{eqn_first_arrow}RHom(i_*\sO_S, \sO_X)[1]\cong R\Gamma(\sO_S(2))[-1]\xrightarrow{[0\; 1]^t} R\Gamma(\sO_S(1))^{2}\oplus R\Gamma(\sO_S(2))[-1].\end{equation}
From \eqref{equ_sec} and similar to \eqref{eqn_exact_triangle1}, we have an exact triangle
$$R\Hom(i_*\sO_S, \sO_X)[1]\to R \Hom(I_S, i_*\sO_S)\rightarrow R\Hom(I_S, I_S)_0[1]$$ in which the first arrow is \eqref{eqn_first_arrow}. From this, we see that $$R\Hom(I_S, I_S)_0[1]\cong R\Gamma(\sO_S(1))^{2}.$$ Taking cohomology 
$$\Ext^{1}(I_S,I_S)_0\cong H^0(\sO_S(1))^2,\qquad \Ext^{2}(I_S,I_S)_0=0,$$ and so the claim is proven.
The component of $P$ containing $J_X$ is then identified with the Grassmannian $ \mbox{Gr}(4,6)$ parameterizing $\bb P^3 \subset \bb P^5$. It is therefore nonsingular and has the expected dimension $\vd/2=8$. 

We can use this to compute a simple stable pair invariant $$\lra{\op{pt}^4}_{\Ch}^P=1,$$ because $\sigma_2^4=1$ in the cohomology ring of $ \mbox{Gr}(4,6)$ corresponding to the fact that there is a unique $\bb P^3\subset \bb P^5$ passing through 4 generic points, which in turn  implies that  there is a unique sextic surface in $X$ passing through the given 4 generic points.

\subsection{K3 surface fibrations}\label{subsec_K3_fibration}

Let $X$ be a Calabi-Yau 4-fold admitting a flat morphism $p: X\to B$ to a nonsingular smooth projective surface $B$ and with general fibers  nonsingular K3 surfaces. 

 Let $\Ch$ be the Chern character of the structure sheaf of a fiber. A component of the moduli space $P(X,\Ch)$ is identified with $B$ by assigning to any $b\in B$ the stable pair $J_X:=[\sO_X\to i_*\sO_S]$,  where $i\colon S\subset X$ is the fiber over $b\in B$.  This component is nonsingular of expected dimension $\vd/2=2$.

Now suppose that $b\in B$ is a general point , so the fiber $S$ is a nonsingular K3 surface. Pulling back a Koszul resolution for $\sO_b$ in $B$ gives a locally free resolution $$0\to \wedge^2 p^*N\to p^*N\to  \sO_X\to i_*\sO_S\to 0$$ for some rank 2 locally free sheaf $N$ on $B$. Clearly $p^*N|_S\cong \sO_S^2$ and so $Li^*i_* \sO_S\cong \sO_S^2[1]\oplus \sO_S[2]$. 
As in \S \ref{sec_sextic}, $R\Hom(J_X, J_X)_0[1]\cong R\Gamma(\sO_S)^{2}$, and hence 
$$\Ext^{1}(J_X,J_X)_0\cong H^0(\sO_S)^2,\qquad \Ext^{2}(J_X,J_X)_0=0,$$ and as a result $\lra{\op{pt}}_{\Ch}^P=1$, being the integral over $B$  of the class of a point. 
%

\subsection{Abelian surface fibrations}\label{subsec_Ab_fibration}
Let $X$ be a Calabi-Yau 4-fold admitting a smooth morphism $p: X\to B$ to a nonsingular smooth projective surface $B$ and with fibers  abelian surfaces. We  assume that \begin{equation} \label{equ_rpstar} Rp_* \sO_X\cong \sO_B \oplus (\sO_B \oplus K_B)[-1]\oplus K_B[-2].\end{equation} An example of such $X$ is the product of an abelian surface by either another abelian surface or  a K3 surface.

Let $S$ be a fiber of $p$, and $\Ch=\Ch(\sO_S^n)$ for some integer $n>0$. We claim that the moduli space of stable pairs in class $\Ch$ is identified with the Hilbert scheme of points on $B$ (see \cite[Lemma 2.5]{CMT} for a parallel result).
\begin{lem} $P(X,\Ch)\cong B^{[n]}$.
\end{lem}
\begin{proof}
Let $I_Z<\sO_B$ be the ideal of a length $n$ subscheme $Z$. To this, we assign the stable pair $$[\sO_X \to \sO_{p^{-1}(Z)}]\cong p^*I_Z[1].$$
Conversely, given a stable pair $[\sO_X\to F]$ with $\Ch(F)=\Ch$, using Harder-Narasimhan
and Jordan-H\"older filtrations, we have
$$0 = F_0 \subset  F_1 \subset F_2 \subset \cdots\subset F_r = F,$$
where the quotient $E_i = F_i/F_{i-1}$'s are non-zero stable sheaves with decreasing reduced Hilbert polynomials with respect to a very ample divisor $\omega$. In particular, if $\Ch(E_i)=(0,0,n_i\gamma, \beta_i, \xi_i)$ with $n_i>0$ and $\gamma$ the class of the fiber, we must have 
$$\frac{\omega\cdot \beta_1}{n_1}\ge \frac{\omega\cdot \beta_2}{n_2}\ge \cdots \ge \frac{\omega\cdot \beta_r}{n_r}.$$ Note that $n_1> 0$ by the purity of $F$, and so all $n_i>0$ by the inequalities above. Moreover, each $E_i$ is stable so it must be the pushforward of some rank $n_i$ torsion free sheaf $G_i$ from a single fiber.
Now in the given stable pair the section $\sO_X\to F$ is surjective in dimension 2, so is the induced section $\sO_X\to F\twoheadrightarrow E_r$. By adjunction one gets a non zero section $\sO_S\to G_r$ with an at most 1-dimensional cokernel for some fiber $S$ of $p$. This in particular shows that $\omega\cdot \beta_r \ge 0$. But since $0=\omega \cdot\Ch_3(F)=\sum_{i=1}^r \omega\cdot \beta_i$, we conclude that $\omega\cdot \beta_i=0$ for all $i$. 

As a result of the discussion above, the cokernel of $\sO_S\to G_r$  must be at most 0-dimensional, so $\chi(G_r)=\chi(E_r)\ge 0$. In the filtration above we also have  $$\frac{\chi(E_1)}{n_1}\ge \frac{\chi(E_2)}{n_2}\ge \cdots \ge \frac{\chi(E_r)}{n_r}.$$ Therefore, from $0=\chi(F)=\sum_{i=1}^r \chi(E_i)$, we conclude that $\chi(E_i)=0$ for all $i$, and in particular $G_r=\sO_S$.

If $I_S<\sO_X$ is the ideal of $S$, the composition $I_S\to F\twoheadrightarrow E_r$ is clearly 0 ($E_r$ is the push forward of $G_r=\sO_S$), so the first arrow in the composition above must factor through $F_{r-1}$, and the induced map $I_S\to F_{r-1}\twoheadrightarrow E_{r-1}$ has an at most 1-dimensional cokernel. Let $\iota\colon S'\subset X$ be the fiber $p$ with the ideal  $I_{S'}<\sO_X$ supporting $E_{r-1}=\iota_* G_{r-1}$. By adjunction, we get a map $\iota^* I_{S}\to G_{r-1}$. But $\iota^* I_S=\sO_{S'}$ if $S\neq S'$ and $\iota^* I_S=\sO_{S'}^2$ if $S=S'$. In any case, we get a nonzero map $\sO_{S'}\to G_{r-1}$ with an at most 1-dimension cokernel. Repeating the argument above, we can show that $G_{r-1}=\sO_{S'}$. Eventually, we arrive at the conclusion that in our stable pair the section $\sO_X\to F$ is surjective and $F$ is the pull back of the structure sheaf of some 0-dimensional subscheme of $B$.
\end{proof}

This means that the moduli space of stable pairs $P(X,\Ch)$ is nonsingular of dimension $2n$ in this case. The virtual dimension is $\vd/2=0$. To find the obstruction bundle, let $I_Z< \sO_B$ be the ideal of a closed subscheme of $Z$ of $B$ of length $n$. The corresponding stable pair in $P(X,\Ch)$ is isomorphic to $p^*I_Z[1]$. By adjunction, projection formula and  our assumption \eqref{equ_rpstar} we can write \begin{align*} &R\Hom(p^*I_Z[1],p^*I_Z[1])_0\cong R\Hom(I_Z,Rp_* p^*I_Z)_0\cong \\& R\Hom(I_Z,I_Z)_0\oplus R\Hom(I_Z,I_Z\oplus I_Z(K_B))_0[-1]\oplus R\Hom(I_Z,I_Z(K_B))_0[-2].\end{align*}
In particular, $$\Ext^2(p^*I_Z[1],p^*I_Z[1])_0\cong \Ext^1(I_Z,I_Z)_0\oplus \Ext^1(I_Z,I_Z)^*_0.$$

We know $\Ext^1(I_Z,I_Z)_0$ is the tangent space of $B^{[n]}$ at the point corresponding to $I_Z$. Arguing by means of the universal families instead, we can identify the obstruction bundle of $P(X,\Ch)$ with $T_{B^{[n]}}\oplus \Omega_{B^{[n]}}$. So by \eqref{BF56} we find that  $$\lra{\;}^P_{\Ch}=\deg [P(X,\Ch)]^{\vir}=\sqrt e(T_{B^{[n]}}\oplus \Omega_{B^{[n]}})=e({B^{[n]}})$$ that can be found by means of the famous G\"ottsche's formula.

\subsection*{}


\begin{thebibliography}{12}  

\bibitem{ABL} D.~ Arcara, A.~ Bertram,  M.~ Lieblich. \newblock Bridgeland-stable moduli spaces for $K$-trivial surfaces. {\em{ J. Eur. Math. Soc.}}, 15(1):1-38 (2013).

\bibitem{BKP} Y.~ Bae, M.~Kool, H.~Park, \newblock Counting surfaces on Calabi-Yau 4-folds. In preparation.

\bibitem{BayerPBSC}  A. ~Bayer, \newblock Polynomial Bridgeland stability conditions and the large volume limit, 
{\em{ Geom. Topol.}}, (2009), Vol 13, 2389-2425. 
  

\bibitem{BF}   K. Behrend and B. Fantechi, \newblock The intrinsic normal cone, {\em{ Invent. Math}}. 128 (1997), no. 1, 45-88.


\bibitem{Bridgeland10} T.~ Bridgeland, \newblock An introduction to motivic Hall algebras, {\em  Adv. Math.} 229, no. 1, 102-138 (2012).                    

\bibitem{bridgeland2011hall} T.~ Bridgeland, \newblock Hall algebras and curve-counting invariants. {\em J. Amer. Math. Soc.}, 24(4):969-998 (2011).

\bibitem{CGJ} Y. Cao, J. Gross, and D. Joyce, \newblock Orientability of moduli spaces of Spin(7)-instantons and coherent sheaves on Calabi-Yau 4-folds, arXiv:1811.09658.



\bibitem{CMT} Y. Cao, D. Maulik and Y. Toda, \newblock   Stable pairs and Gopakumar-Vafa type invariants for Calabi-Yau 4-folds, arXiv:1902.00003. 


\bibitem{CO}E. Carlsson and A. Okounkov, \newblock  Exts and vertex operators, {\em Duke Mathematical Journal.},  2012 Jun 15;161(9):1797-815.


\bibitem{EG} D. Edidin, W. Graham, \newblock  Characteristic classes and quadratic bundles, {\em Duke Math. J.}, 78 (1995), 277-299.


\bibitem{flenner1984restrictions} H.~ Flenner H, \newblock Restrictions of semistable bundles on projective varieties, {\em Commentarii Mathematici Helvetici}, 59(1):635--50 (1984).



\bibitem{F} W.~ Fulton. \newblock Intersection theory, {\em Princeton University Press} (2016).

\bibitem{GK} A.~Gholampour, M.~Kool, \newblock Stable reflexive sheaves and localization. {\em Journal of Pure and Applied Algebra}, 221 (2017), 1934-1954.

\bibitem{GSY} A.~ Gholampour, A.~ Sheshmani, S.-T.~ Yau, \newblock Nested Hilbert schemes on surfaces: virtual fundamental class, {\em Advances in Mathematics}, 365:107046 (2020).

\bibitem{GT1} A.~ Gholampour, R.~ P.~Thomas, \newblock Degeneracy loci, virtual cycles and nested Hilbert schemes, I. {\em Tunisian Journal of Mathematics}, 2(3): 633--665 (2019).

\bibitem{GT2} A.~ Gholampour, R.~ P.~Thomas, \newblock Degeneracy loci, virtual cycles and nested Hilbert schemes, II {\em Compositio Mathematica}, 156(8):1623--1663 (2020).

\bibitem{grothendieck1964elements}  A.~ Grothendieck. \'El\'ements de g\'eom\'etrie alg\'ebrique: IV. \'Etude locale des sch\'emas et des morphismes de sch\'emas.
Publications Math\'ematiques de l'IH\'ES, 20:5--259(1964).

\bibitem{Gr}  A.~Grothendieck,  \newblock Techniques de construction et th\'eor\`emes d'existence en g\'eom\'etrie alg\'ebrique, IV: Les sch\'emas de Hilbert. {\em S\'eminaire Bourbaki} (1961).

\bibitem{H} R.~Hartshorne, \newblock Generalized divisors and biliaison, {\em Illinois Journal of Mathematics}, 51(1):83--98 (2007).

\bibitem{HL}D.~ Huybrechts and M. ~Lehn, \newblock  The geometry of moduli spaces of sheaves, {\em Aspects of Mathematics}, E31, Friedr. Vieweg \& Sohn, Braunschweig, 1997. MR MR1450870 (98g:14012).
 
\bibitem{huybrechts2006fourier}  D. Huybrechts, \newblock   Fourier-Mukai Transforms in Algebraic Geometry, {\em{ Oxford Mathematical Monographs}}, isbn=9780199296866,
  lccn={2006298244}. 
\bibitem{HT}D. Huybrechts and R. P. Thomas, \newblock   Deformation-obstruction theory for complexes via Atiyah and Kodaira-Spencer classes, {\em Math. Ann.}, 346 (2014), 545-569. 


\bibitem{kleiman1971theoremes} S.~Kleiman, \newblock Les th\'eoremes de finitude pour le foncteur de Picard, {\em InTh\'eorie des intersections et th\'eoreme de Riemann-Roch} 616--666 (1971), Springer, Berlin, Heidelberg.

\bibitem{kollar1990} J. Koll\'{a}r, \newblock  Projectivity of complete moduli, {\em  J. Differential Geom.},  Vol. 32, No. 1, (1990), 235-268. 
 
\bibitem{K} M.~ Kool, \newblock Stable pair invariants of surfaces and Seiberg-Witten invariants, {\em The Quarterly Journal of Mathematics} 67(3):365--386 (2016). 
 
 






\bibitem{KKP} B.~Kim, A.~ Kresch, T.~ Pantev, \newblock Functoriality in intersection theory and a conjecture of Cox, Katz, and Lee. {\em Journal of Pure and Applied Algebra}, 179(1-2):127-36 (2003).


\bibitem{KT}  M.~ Kool and R.~ P.~ Thomas, \newblock Reduced classes and curve counting on surfaces I: theory.
{\em Algebraic Geometry}, 1, 334--383, 2014.





\bibitem{LeP1} J.~ Le Potier. \newblock Syst\`emes coh\'erents et structures de niveau, {\em Ast\'erisque}, 214, 143, (1993).
\bibitem{LeP2} J.~ Le Potier. \newblock Faisceaux semi-stables et syst\`emes coh\'erents. In {\em Vector bundles in algebraic geometry (Durham, 1993)}, vol 208 of London Math. Soc.
Lecture Note Ser., 179--239. Cambridge Univ. Press, Cambridge, (1995).



\bibitem{lin2018moduli} Y.~ Lin, \newblock Moduli spaces of stable pairs, {\em Pacific Journal of Mathematics}, 294(1):123--58 (2018).

\bibitem{Lo1} J.~ Lo, \newblock  Moduli of PT-semistable objects I, {\em J. Algebra}, Vol. 339, No.1, 203--222 (2011). 
\bibitem{Lo2} J.~ Lo, \newblock Moduli of PT-semistable objects II, {\em Trans. Amer. Math. Soc.}, 365(9):4539--4573 (2013).

\bibitem{man} C.~Manolache, \newblock Virtual pull-backs, {\em J. Algebraic Geom.} 21(2):201--245 (2012).
\bibitem{nmj/1118795443}  M.~ Maruyama, \newblock Stable vector bundles on an algebraic surface. {\em Nagoya Mathematical Journal}, 58:25--68 (1975).
\bibitem{nmj/1118786090} M.~ Maruyama, \newblock Boundedness of semi-stable sheaves of small ranks, {\em Nagoya Mathematical Journal}, 78:65--94 (1980).
\bibitem{10.1215/kjm/1250521908} M.~Maruyama, \newblock On boundedness of families of torsion free sheaves, {\em Journal of Mathematics of Kyoto University}, 21(4):673--701 (1981).



\bibitem{OS1} G.~ Oberdieck and J.~ Shen. \newblock Curve counting on elliptic Calabi-Yau threefolds via derived categories. Preprint. arXiv:1608.07073 [math.AG].

\bibitem{OT}J. Oh, and R. P. Thomas, \newblock  Counting sheaves on Calabi-Yau 4-folds, I, arXiv:2009.09542.
 \bibitem{PT}R. Pandharipande and R. P. Thomas, \newblock  Curve counting via stable pairs in the derived category, {\em Invent. Math.}, vol.178, 407-447 (2009).
\bibitem{PT3}  R. Pandharipande and R. P. Thomas, \newblock Stable pairs and BPS invariants, {\em Journal of the
American Mathematical Society}, 23, 267-297, 2010.
\bibitem{park} H.~ Park . \newblock Virtual pullbacks in Donaldson-Thomas theory of Calabi-Yau 4-folds, arXiv:2110.03631.
\bibitem{piyaratne2019moduli} D. Piyaratne and Y. Toda, \newblock Moduli of {B}ridgeland semistable objects on 3-folds and {D}onaldson-{T}homas invariants,
  {\em J. Reine Angew. Math.}, (2019), No. 747, 175-219. 

\bibitem{Pol} A.~ Polishchuk.  \newblock Constant families of t-structures on derived categories of coherent sheaves. {\em Mosc. Math. J.}, 7:109--134 (2007).

\bibitem{Pol2} A.~ Polishchuk. \newblock Holomorphic bundles on 2-dimensional noncommutative toric orbifolds. {\em Noncommutative Geometry and Number Theory, Max-Planck-Institute for Mathematics}  341--359 (2006).



                       
\bibitem{Toda2} Y.~ Toda, \newblock  Hall algebras in the derived category and higher rank DT invariants, {\em  Algebr. Geom}, 7(3):240--262 (2020).


\bibitem{TODA20082736} Y.~Toda, \newblock Moduli stacks and invariants of semistable objects on K3 surfaces, {\em Advances in Mathematics}, 217(6):2736--81 (2008).


\bibitem{wandel2015moduli} M.~ Wandel, \newblock Moduli spaces of semistable pairs in Donaldson-Thomas theory, {\em Manuscripta Mathematica}, 147(3):477--500 (2015).

\end{thebibliography}
\end{document}